\documentclass[10pt,reqno]{amsart}

\usepackage{amsthm,amsmath,amstext,amssymb,amscd,euscript, mathrsfs, dsfont,multicol,times,enumerate,subfig,sidecap}
\usepackage{mathrsfs}
\usepackage{esint}
\usepackage{xcolor}
\usepackage{mathtools}
\usepackage{bm}
\usepackage{bbm}

\usepackage[colorlinks=true, pdfstartview=FitV, linkcolor=blue, citecolor=red, urlcolor=blue]{hyperref} 

\newtheorem{thm}{Theorem}[section]
\newtheorem{corollary}[thm]{Corollary}
\newtheorem{lem}[thm]{Lemma}
\newtheorem{prop}[thm]{Proposition}
\theoremstyle{definition}
\newtheorem{defn}[thm]{Definition}

\newtheorem{rem}[thm]{Remark}

\newcommand\bE{\mathbb{E}}

\newcommand\bN{\mathbb{N}}
\newcommand\bM{\mathbb{M}}

\newcommand\bR{\mathbb{R}}
\newcommand\bS{\mathbb{S}}
\newcommand\bT{\mathbb{T}}

\newcommand\bZ{\mathbb{Z}}

\newcommand\fR{\mathbb{R}}

\newcommand\cB{\mathcal{B}}

\newcommand\cF{\mathcal{F}}

\newcommand\cI{\mathcal{I}}

\newcommand\cL{\mathcal{L}}

\newcommand\cS{\mathcal{S}}
\newcommand\cM{\mathcal{M}}
\newcommand\cT{\mathcal{T}}
\newcommand\cO{\mathcal{O}}

\newcommand\frP{\mathfrak{P}}

\newcommand\frI{\mathfrak{I}}

\newcommand{\p}{\partial}

\newcommand{\aint}{-\hspace{-0.40cm}\int}

\newcommand{\mysection}[1]{\section{#1}
\setcounter{equation}{0}}

\def\XXint#1#2#3{{\setbox0=\hbox{$#1{#2#3}{\int}$ }
\vcenter{\hbox{$#2#3$ }}\kern-.58\wd0}}

\makeatletter
\def\@tocline#1#2#3#4#5#6#7{\relax
  \ifnum #1>\c@tocdepth 
  \else
    \par \addpenalty\@secpenalty\addvspace{#2}%
    \begingroup \hyphenpenalty\@M
    \@ifempty{#4}{%
      \@tempdima\csname r@tocindent\number#1\endcsname\relax
    }{%
      \@tempdima#4\relax
    }%
    \parindent\z@ \leftskip#3\relax \advance\leftskip\@tempdima\relax
    \rightskip\@pnumwidth plus4em \parfillskip-\@pnumwidth
    #5\leavevmode\hskip-\@tempdima
      \ifcase #1
       \or\or \hskip 1em \or \hskip 2em \else \hskip 3em \fi%
      #6\nobreak\relax
    \dotfill\hbox to\@pnumwidth{\@tocpagenum{#7}}\par
    \nobreak
    \endgroup
  \fi}
\makeatother

\begin{document}
\title[An IVP with a time-measurable pseudo-differential operator in a weighted $L_p$-space]
{A regularity theory for an initial value problem with a time-measurable pseudo-differential operator in a weighted $L_p$-space}

\author[J.-H. Choi]{Jae-Hwan Choi}
\address[J.-H. Choi]{School of Mathematics, Korea Institute for Advanced Study, 85 Hoegi-ro, Dongdaemun-gu, Seoul 02455, Republic of Korea} 
\email{jhchoi@kias.re.kr}

\author[I. Kim]{Ildoo Kim}
\address[I. Kim]{Department of Mathematics, Korea University, 145 Anam-ro, Seongbuk-gu, Seoul 02841, Republic of Korea}
\email{waldoo@korea.ac.kr}

\author[J.B. Lee]{Jin Bong Lee}
\address[J.B. Lee]{Research Institute of Mathematics, Seoul National University, Seoul 08826, Republic of Korea}
\email{jinblee@snu.ac.kr}



\maketitle
\begin{abstract}
In this study, we investigate the existence, uniqueness, and maximal regularity estimates of solutions to homogeneous initial value problems involving time-measurable pseudo-differential operators within the framework of weighted mixed norm Lebesgue spaces.
The class of temporal weights in our regularity estimates contains Muckenhoupt's class, and the initial data is in weighted Besov spaces with variable order.
\end{abstract}
\tableofcontents

\mysection{Introduction}

In this paper, we investigate the homogeneous initial value problem described by the following equation:
\begin{equation}\label{eqn:model eqn}
    \begin{cases}
        \p_tu(t,x)=\psi(t,-i\nabla)u(t,x),\quad &(t,x)\in(0,T)\times\bR^d,\\
        u(0,x)=u_0(x),\quad &x\in\bR^d.
    \end{cases}
\end{equation}
Here $\psi(t,-i\nabla)$ denotes a time-measurable pseudo-differential operator defined as
\begin{align*}
\psi(t,-i\nabla)u(t,x)= \cF^{-1}\left[ \psi(t,\cdot) \cF[u(t,\cdot)] \right] (x),
\end{align*}
where $\mathcal{F}$ and $\mathcal{F}^{-1}$ represent the $d$-dimensional Fourier transform and its inverse on $\mathbb{R}^d$, respectively.
Specifically, they are defined by
$$
    \mathcal{F}[f](\xi) := \frac{1}{(2\pi)^{d/2}} \int_{\mathbb{R}^d} \mathrm{e}^{-i \xi \cdot x} f(x) \mathrm{d}x, \,\,\,\,
    \mathcal{F}^{-1}[f](x) := \frac{1}{(2\pi)^{d/2}} \int_{\mathbb{R}^d} \mathrm{e}^{ix\cdot \xi} f(\xi) \mathrm{d}\xi.
$$
The symbol $\psi$ associated with the operator $\psi(t,-i\nabla)$ is required to fulfill two conditions:
\begin{itemize}
    \item We say that a symbol $\psi(t,\xi)$ or the operator $\psi(t,-i\nabla)$ satisfies {\bf an ellipticity condition} (with ($\gamma$, $\kappa$))
 if there exists a $\gamma\in(0,\infty)$ and $\kappa\in(0,1]$ such that
\begin{align}\label{condi:ellipticity}
    \Re [-\psi(t,\xi)]\geq \kappa |\xi|^{\gamma},\quad \forall (t,\xi)\in (0,\infty)\times\bR^d,
\end{align}
where $\Re[z]$ denotes the real part of the complex number $z$.
\item For $n \in \bN$, we say that a symbol $\psi(t,\xi)$ or the operator $\psi(t,-i\nabla)$ has {\bf a $n$-times regular upper bound} (with ($\gamma$, $M$)) if there exist positive constants $\gamma$ and $M$ such that 
\begin{align}\label{condi:reg ubound}
    |D^{\alpha}_{\xi}\psi(t,\xi)|\leq M|\xi|^{\gamma-|\alpha|},\quad \forall (t,\xi)\in (0,\infty) \times(\bR^{d}\setminus\{0\}),
\end{align}
for any ($d$-dimensional) multi-index $\alpha$ with $|\alpha| \leq n$.
\end{itemize}
The parameter $\gamma$ in conditions \eqref{condi:ellipticity} and \eqref{condi:reg ubound} denotes the \emph{order} of the operator $\psi(t,-i\nabla)$, which is consistently set within the range $(0, \infty)$ throughout this paper.

Our objective is to determine a suitable value for $n$ in \eqref{condi:reg ubound} that allows us to establish a well-posedness result and develop a maximal regularity theory for equation \eqref{eqn:model eqn} within the framework of $L_q((0,T),\mu(\mathrm{d}t);L_p(\mathbb{R}^d,w\,\mathrm{d}x))$.
Here, $\mu$ denotes a nonnegative Borel measure, and $w$ refers to a weight function belonging to Muckenhoupt's $A_p(\mathbb{R}^d)$ class.

\begin{defn}
\label{def ap}
For $p\in(1,\infty)$, let $A_p(\bR^d)$ be the class of all nonnegative and locally integrable functions $w$ satisfying
\begin{align*}
							\notag
&[w]_{A_p(\bR^d)}\\
:=&\sup_{x_0\in\bR^d,r>0}\left(\aint_{B_r(x_0)}w(x)\mathrm{d}x\right)\left(\aint_{B_r(x_0)}w(x)^{-\frac{1}{p-1}}\mathrm{d}x\right)^{p-1}\\
:=&\sup_{x_0\in\bR^d,r>0}\left(\frac{1}{|B_r(x_0)|}\int_{B_r(x_0)}w(x)\mathrm{d}x\right)\left(\frac{1}{|B_r(x_0)|}\int_{B_r(x_0)}w(x)^{-\frac{1}{p-1}}\mathrm{d}x\right)^{p-1}
<\infty,
\end{align*}
where $|B_r(x_0)|$ denotes the Lebesgue measure of $B_r(x_0) := \{ x \in \bR^d : |x-x_0| < r\}$. 
The class $A_{\infty}(\bR^d)$ could be defined as
    $$
    A_\infty(\bR^d)=\bigcup_{p\in(1,\infty)}A_p(\bR^d).
    $$
\end{defn}

We begin by revisiting prior studies on the $L_p$-regularity theory for evolution equations.
A multitude of research papers address zero-initial inhomogeneous problems within $L_p$ spaces, as described by:
\begin{equation}
										\label{20230126 01}
    \begin{cases}
        \p_tu(t,x)=\psi(t,-i\nabla)u(t,x) + f(t,x),\quad &(t,x)\in(0,T)\times\bR^d,\\
        u(0,x)=0,\quad &x\in\bR^d.
    \end{cases}
\end{equation}
Initially, we highlight findings concerning time-measurable pseudo-differential operators.
For the case of data $f$ and solutions $u$ in unweighted $L_p$ spaces, we refer the reader to \cite{kim2015parabolic,kim2016lplq,kim2016lp,kim2018lp}.
The exploration extends to the results in weighted $L_p$ spaces by the first and second authors \cite{Choi_Kim2022}.
Significantly, if the operator $\psi(t,-i\nabla)$'s order, $\gamma$, is less than or equal to two, such operators are generators of additive processes, which are time-inhomogeneous L\'evy processes.
Regularity theories in $L_p$ spaces for \eqref{20230126 01} involving generators of stochastic processes and non-local operators, can be found in \cite{dong2021sobolev,gyongy2021lp,kang2021lp,kim2019lp,kim2012lp,kim2021lq,kim2021sobolev,mikulevivcius1992,mikulevivcius2014,mikulevivcius2017p,mikulevivcius2019cauchy,zhang2013maximal,zhang2013lp}.
Additionally, studies employing an $H^\infty$-calculus approach for general operators with smooth symbols in $L_p$ spaces are presented in \cite{neerven2012maximal,neerven2020,portal2019stochastic}.

On the other hand, studies addressing non-zero initial value problems with general operators in $L_p$ spaces, akin to our focus in equation \eqref{eqn:model eqn}, are comparatively scarce.
We found \cite{Choi_Kim2023,Dong_Kim2021,Dong_Liu2022,Gallarati_Veraar2017} as recent results to non-zero initial value problems with general operators.
The first and second authors \cite{Choi_Kim2023} considered initial value problems with generators of general additive processes.
Gallarati and Veraar \cite{Gallarati_Veraar2017} studied an $L_p$-maximal regularity theory of evolution equations with general operators using $H^\infty$-calculus. 
Especially, in \cite[Section 4.4]{Gallarati_Veraar2017}, they obtained solvability to an equation of the type of \eqref{eqn:model eqn} in $L_p((0,T),v\, \mathrm{d}t; X_1)\cap W_{p}^{1}((0,T),v\, \mathrm{d}t; X_0)$.
Here $X_0$ and $X_1$ are Banach spaces that have a finite cotype, and the initial data space $X_{v,p}$ is given by a trace of a certain interpolated space determined by $X_0$ and $X_1$, \textit{i.e.}
\begin{equation*}
    \begin{gathered}
        X_{v,p}:=\{x\in X_0:\|x\|_{X_{v,p}}<\infty\},\\
        \|x\|_{X_{v,p}}:=\inf\{\|u\|_{L_p((0,T),v\, \mathrm{d}t; X_1)\cap W_{p}^{1}((0,T),v\, \mathrm{d}t; X_0)}:u(0)=x\}.
    \end{gathered}
\end{equation*}
It is not easy to identify  $X_{v,p}$ as a certain interpolation space in general.
However, this abstract space can be characterized by a real interpolation space if $v(t)$ is given by a power-type weight $v(t)=t^\gamma$.
Indeed, due to Triebel \cite[Theorem 1.8.2]{triebel1978interpolation},
$$
X_{v,p}=(X_0,X_1)_{1-\frac{(1+\gamma)}{p},p},\quad \gamma\in(-1,p-1),
$$
where $(X_0,X_1)_{\theta,p}$ stands for the real interpolation space between $X_0$ and $X_1$.
We refer to \cite{Hemmel_Lindemulder2022,Kohne2010,Kohne2014,Lindemulder2020,Meyries2012}, which provide such identification for initial data spaces with the power type weights in various situations such as smooth domains, quasi-linear equations, and systems of equations.

Furthermore, an exploration into non-zero initial value problems with time-fractional derivative, addressing second-order issues with rough coefficients, extends the discourse, underscoring the adaptability of weighted Slobodeckij spaces as initial data spaces to ascertain solvability for a broader range of $\alpha$;
\begin{equation}\label{eqn:time frac}
	\begin{cases}
	\partial_t^\alpha u(t,x) = a^{ij}(t,x)D_{x^i x^j}u(t,x) + f(t,x), \quad &(t,x) \in (0,T)\times \mathbb{R}^d,\\
	u(0, x) = u_0(x), \quad &x\in \bR^d,
	\end{cases}
\end{equation}
Dong and Kim \cite{Dong_Kim2021} use a weighted Slobodeckij space as initial data space to obtain the solvability of \eqref{eqn:time frac} in $L_q((0,T),t^\gamma\, \mathrm{d}t; W_p^2(\bR^d,w\,\mathrm{d}x))$ for $\alpha \in (0,1]$.
Dong and Liu \cite{Dong_Liu2022} extended the range of $\alpha$ to $(0,2)$.

Despite significant advancements, to the authors' knowledge, results elucidating the solvability of equation \eqref{eqn:model eqn} under non-power-type temporal weights and providing explicit characterizations of the initial data spaces have not been fully pursued.
This gap can be attributed to two principal challenges.
Firstly, the absence of trace theorems for non-power-type temporal weights, denoted as $\mu(\mathrm{d}t) \neq t^{\gamma}\,\mathrm{d}t$ at the initial moment, leaves the optimal initial data space not characterized.
Secondly, the solution operator, defined by $u_0 \mapsto \int_{\mathbb{R}^d} K(t,0, x-y) u_0(y) \mathrm{d}y$, transitions to an extension operator from $L_p(\mathbb{R}^d)$ to $L_p(\mathbb{R}^{d+1})$, with the kernel $K(t,s, x-y)$ given by
\begin{align*}
K(t,s, x-y) = \cF^{-1}\left[  \exp\left( \int_s^t \psi(r,\cdot)\mathrm{d}r \right)  \right](x-y).
\end{align*}
This transformation complicates the application of the weighted $L_p (\mathbb{R}^{d+1})$-boundedness of singular integral operators, which is crucial for inhomogeneous problems \eqref{20230126 01}.
Furthermore, while an $L_q(L_p)$-extension derived from $L_p$-estimates is typically straightforward due to Rubio de Francia's extrapolation theorem, it is not possible to apply the extrapolation theorem when one considers a weight class beyond the Muckenhoupt $A_p(\bR)$-class.

We overcome the difficulties by leveraging the Littlewood--Paley theory and the Laplace transform.
Initially, we prove spatial mean oscillation estimations for each Littlewood--Paley projection of a solution, as delineated in Theorem \ref{prop:maximal esti}.
This approach establishes a framework for weighted $L_p$ estimates that circumvents the reliance on singular integral operator theory.
Within this framework, integrating the solution along the time variable yields the Laplace transform of a measure $\mu(\mathrm{d}t)$, and the measure $\mu$ will be our temporal weight class.
We note that the temporal weight is beyond the $A_p(\bR)$ class.
The application of the Laplace transform to $\mu$ introduces variable order Besov spaces as a natural component of our analysis.
These methodologies are underpinned by pointwise upper bound estimates for each Littlewood--Paley projection of the fundamental solution, as specified in Theorem \ref{22.02.15.11.27}.

The novelty of our results is twofold.
\begin{enumerate}
    \item \textbf{Optimal initial data space}: Our findings are particularly significant when the temporal measure is defined as $\mu(\mathrm{d}t) = w(t)\mathrm{d}t$, with $w$ belonging to the Muckenhoupt class.
    A question arises regarding whether the generalized Besov space proposed in this study is equivalent to the trace space.
    We affirmatively address this question, offering a comprehensive explanation and references in Remark \ref{24.02.21.14.50}.

    \item \textbf{Distinctiveness from inhomogeneous problems}: Our results extend beyond the scope of existing work on trace theorems and inhomogeneous problem solutions. 
    Specifically, our solutions are characterized within the space $L_q((0,T),\mu(\mathrm{d}t); L_p(\mathbb{R}^d, w\,\mathrm{d}x))$, where $q\in(0,\infty)$ and $\mu$ represents a nonnegative Borel measure on $\mathbb{R}$.
    To our knowledge, there are no preceding results that encompass the trace theorem or inhomogeneous theorem within this context.
    For illustrative purposes, consider a solution $u_1$ to the following equation:
    \begin{equation*}
    \begin{cases}
        \p_tu_1(t,x)= \Delta^{\gamma/2}u_1(t,x),\quad &(t,x)\in(0,T)\times\bR^d,\\
        u(0,x)=u_0(x),\quad &x\in\bR^d,
    \end{cases}
\end{equation*}
and, by setting $f(t,x) := \psi(t,-i\nabla)u_1(t,x) - \Delta^{\gamma/2}u_1(t,x)$ in $\eqref{20230126 01}$, also consider a solution $u_2$ to:
\begin{equation*}
    \begin{cases}
        \p_tu_2(t,x)=\psi(t,-i\nabla)u_2(t,x) +f(t,x),\quad &(t,x)\in(0,T)\times\bR^d,\\
        u(0,x)=0,\quad &x\in\bR^d.
    \end{cases}
\end{equation*}    
    Given the linearity of these equations, the function $u:= u_1 + u_2$ solves:
\begin{equation*}
    \begin{cases}
        \p_tu(t,x)=\psi(t,-i\nabla)u(t,x),\quad &(t,x)\in(0,T)\times\bR^d,\\
        u(0,x)=u_0(x),\quad &x\in\bR^d.
    \end{cases}
\end{equation*}
A weighted estimate of the solution $u$ to \eqref{eqn:model eqn} can be derived from the weighted estimates of both $u_1$ and $u_2$.
However, the approach to a weighted estimate for $u$ is not entirely straightforward, as the weight class for $u_2$ is more restrictive than that for $u_1$.
This distinction underscores the novelty of our direct estimate for the initial value problem \eqref{eqn:model eqn}. 
Moreover, our weighted estimates introduce new insights, even for the model operator $\Delta^{\gamma/2}$, by virtue of the measure $\mu$'s generality. 
\end{enumerate}

We finish the introduction by presenting the second-order case of the main result for an easy application.
To the best of our knowledge, even the simplest case $\psi(t,-i\nabla)=\Delta$ of our results is new.
We choose a function $\Psi$ in the Schwartz class $\mathcal{S}(\mathbb{R}^d)$ whose Fourier transform $\mathcal{F}[\Psi]$ is nonnegative, supported in the annulus $\{\xi \in \bR^d : \frac{1}{2}\leq |\xi| \leq 2\}$, and $\sum_{j\in\mathbb{Z}} \mathcal{F}[\Psi](2^{-j}\xi) = 1$ for all $\xi \not=0$.
    Then we define the Littlewood--Paley projection operators $\Delta_j$ and $S_0$ as $\mathcal{F}[\Delta_j f](\xi) = \mathcal{F}[\Psi](2^{-j}\xi) \mathcal{F}[f](\xi)$ and $S_0f = \sum_{j\leq 0} \Delta_j f$, respectively.
\begin{thm}
\label{second-order case}
Let $T \in (0,\infty)$, $p,\nu\in(1,\infty)$, $q\in(0,\infty)$,  $w\in A_{\nu}(\bR)$ and $w'\in A_p(\bR^d)$.
    Assume that $a^{ij}(t)$ is a measurable function on $(0,T)$ for all $i,j \in \{1,\ldots,d\}$ and there exist positive constants $\kappa$ and $M$ such that
\begin{align}
							\label{ellip coefficient}
\kappa |\xi|^2 \leq \sum_{i,j=1}^{d}a^{ij}(t) \xi^i \xi^j \leq M |\xi|^2 \qquad \forall (t,\xi) \in (0,T) \times \bR^d.
\end{align}
Suppose that $u_0\in\cS'(\bR^d)$ satisfies 
$$
    \|u_0\|_{B_{p,q}^{W_2}(\bR^d,w'\,\mathrm{d}x)}:=\|S_0u_0\|_{L_p(\bR^d,w'\,\mathrm{d}x)}+\left(\sum_{j=1}^{\infty}2^{2qj}W(2^{-2j})\|\Delta_ju_0\|_{L_p(\bR^d,w'\,\mathrm{d}x)}^q\right)^{1/q}<\infty,
$$
where $W(\lambda):=\int_0^{\lambda}w(t)\mathrm{d}t$.
Then the initial value problem
\begin{equation}
\label{second eqn}							
        \begin{cases}
            \p_tu(t,x)=a^{ij}(t)u_{x^ix^j}(t,x),\quad &(t,x)\in(0,T)\times\bR^d,\\
							
        u(0,x)=u_0(x),\quad &x\in\bR^d.
        \end{cases}
\end{equation}
has a unique solution $u\in L_{q}\left((0,T),w\,\mathrm{d}t;H^{2}_p(\bR^d,w'\,\mathrm{d}x)\right)$.
Moreover, $u$ satisfies the following estimation:
\begin{align}
							\label{second a priori}
\int_{0}^T \|u(t,\cdot)\|_{ H^{2}_p(\bR^d,w'\,\mathrm{d}x)}^{ q} w(t) \mathrm{d}t
    \leq N\|u_0\|_{B_{p,q}^{W_2}(\bR^d,w'\,\mathrm{d}x)}^q,
\end{align}
where $[w]_{A_{\nu}(\bR)},[w']_{A_p(\bR^d)}\leq K$ and $N=N(d,\kappa,M,K,T,p,q)$.
In particular, if $\nu=q$, then
\begin{align}
							\label{second a priori2}
\int_{0}^T \|u(t,\cdot)\|_{ H^{2}_p(\bR^d,w'\,\mathrm{d}x)}^{ q} w(t) \mathrm{d}t
    \simeq \|u_0\|_{B_{p,q}^{W_2}(\bR^d,w'\,\mathrm{d}x)}^q,
\end{align}
\end{thm}
The proof will be given in Section \ref{22.12.27.14.47}.

This paper is organized as follows.
In Section \ref{22.12.27.14.47}, we state our main result, Theorem \ref{22.12.27.16.53}.
We will derive Theorem \ref{second-order case} from Theorem \ref{22.12.27.16.53}.
In Section \ref{22.12.28.13.34}, we provide useful examples and corresponding inhomogeneous results.
The proof of the main result (Theorem \ref{22.12.27.16.53}) is given in Section \ref{pf main thm}, and a proof of a key estimate to obtain the main theorem is given in Section \ref{sec:prop}. Finally, in the appendix, we present the weighted multiplier and Littlewood--Paley theories used to prove our main theorem.

\subsection*{Notation}

\begin{itemize}
    \item 
        For given $p \in [1,\infty)$, a normed space $F$, and a  measure space $(X,\mathcal{M},\mu)$,  $L_{p}(X,\cM,\mu;F)$ denotes the space of all $\mathcal{M}^{\mu}$-measurable functions $u : X \to F$ with the norm 
        \[
            \left\Vert u\right\Vert _{L_{p}(X,\cM,\mu;F)}:=\left(\int_{X}\left\Vert u(x)\right\Vert _{F}^{p}\mu(\mathrm{d}x)\right)^{1/p}<\infty,
        \]
        where $\mathcal{M}^{\mu}$ denotes the completion of $\cM$ with respect to the measure $\mu$. We also denote by $L_{\infty}(X,\cM,\mu;F)$ the space of all $\mathcal{M}^{\mu}$-measurable functions $u : X \to F$ with the norm
        $$
            \|u\|_{L_{\infty}(X,\cM,\mu;F)}:=\inf\left\{r\geq0 : \mu(\{x\in X:\|u(x)\|_F\geq r\})=0\right\}<\infty.
        $$
        If there is no confusion for the given measure and $\sigma$-algebra, we usually omit them.
    \item 
        For $\cO\subseteq \bR^d$, we denote by $\cB(\cO)$ the set of all Borel sets contained in $\cO$.

    \item  
        For $\cO\subset \fR^d$ and a normed space $F$, we denote by $C(\cO;F)$ the space of all $F$-valued continuous functions $u : \cO \to F$ with the norm 
        \[
            |u|_{C}:=\sup_{x\in O}|u(x)|_F<\infty.
        \]

\item 
We write $a \lesssim b$ if there is a positive constant $N$ such that $ a\leq N b$.
We use $a \approx b$ if $a \lesssim b$ and $b \lesssim a$. 
If we write $N=N(a,b,\cdots)$, this means that the
constant $N$ depends only on $a,b,\cdots$. 
A generic constant $N$ may change from a location to a location, even within a line. 
The dependence of a generic constant $N$ is usually specified in each statement of theorems, propositions, lemmas, and corollaries.
   
    \item 
        For $a,b\in \bR$,
        $$
            a \wedge b := \min\{a,b\},\quad a \vee b := \max\{a,b\},\quad \lfloor a \rfloor:=\max\{n\in\bZ: n\leq a\}.
        $$
    \item
        For $r>0$,
        $$
            B_r(x):=\{y\in\bR^d:|x-y|< r\},\quad
            \overline{B_r(x)}:=\{y\in\bR^d:|x-y|\leq r\}.
        $$
        
\item Let $\mu$ be a nonnegative Borel measure on $(0,\infty)$ and $c$ be a positive constant. 
We use the notation $\mu(c\,\mathrm{d}t)$ to denote the scaled measure defined by
$$
\int_{\bR^d}f(t)\mu(c\,\mathrm{d}t):=\int_{\bR^d}f(t/c)\mu(\mathrm{d}t).
$$
\end{itemize}

\mysection{Main result}
\label{22.12.27.14.47}

To state the main result, we need the following definitions of sequences and functions.
\begin{defn}
					\label{doubling sequence}
We say that a sequence $\boldsymbol{r} : \bZ \to \bR$ has {\bf a controlled difference} if
$$
\|\boldsymbol{r}\|_d:=\sup_{j\in\bZ}|\boldsymbol{r}(j+1)-\boldsymbol{r}(j)|<\infty.
$$
\end{defn}

\begin{defn}
				\label{weight}
Let $ a \in \bR$. We say that a function $\phi : (0,\infty) \to \bR$ is {\bf controlled by a sequence $\boldsymbol{\mu} : \bZ \to \bR$ in a $\gamma$-dyadic way with a parameter $a$} if there exists a positive constant $N_\phi$ such that
$$
\phi(2^{\gamma j}) \leq N_\phi 2^{j\gamma a}2^{-\boldsymbol{\mu}(j)},\quad \forall j\in\bZ.
$$	
\end{defn}
Given that we are dealing with a general weight, the interpretation of a solution to \eqref{eqn:model eqn}, even in its weak formulation, becomes elusive.
Nevertheless, it remains feasible to approximate the solution $u$ with functions of higher regularity, where such approximations fulfill the equation in the conventional, or strong, sense. 
Consequently, we need certain function spaces capable of approximating solutions, thereby clarifying the precise nature of a solution to \eqref{eqn:model eqn}. 
The inherent lack of smoothness in the symbol $\psi$ precludes the suitability of the classical Schwartz space for this context, prompting the adoption of a more expansive class that encompasses smooth functions that are locally integrable.
This broader class is denoted by $C_{p}^{1,\infty}([0,T]\times\bR^d)$.
Herein, we provide a detailed mathematical definition to articulate this concept rigorously.
 \begin{enumerate}[(i)]
 \item  The space $C_p^{\infty}([0,T]\times\bR^d)$ denotes the set of all $\cB([0,T]\times\bR^d)$-measurable functions $f$ on $[0,T]\times\bR^d$ such that for any multi-index $\alpha$ with respect to the space variable $x$, 
\begin{equation*}
D^{\alpha}_{x}f\in L_{\infty}([0,T];L_2(\bR^d)\cap L_p(\bR^d)).
\end{equation*}
\item The space $C_p^{1,\infty}([0,T]\times\bR^d)$ denotes the set of all $f\in C_p^{\infty}([0,T]\times\bR^d)$ such that
$\partial_tf\in C_p^{\infty}([0,T]\times\bR^d)$ and for any multi-index $\alpha$ with respect to the space variable $x$,
\begin{equation*}
D^{\alpha}_{x}f\in C([0,T]\times \bR^d),
\end{equation*}
where $\partial_tf(t,\cdot)$ denotes the right derivative and the left derivative at $t=0$ and $t=T$, respectively.
\end{enumerate}

Next, we introduce generalized weighted Besov spaces and Sobolev spaces. 
\begin{defn}\label{def:bessel besov} 
Let $p\in(1,\infty)$, $q\in(0,\infty)$, and $w\in A_p(\bR^d)$.
For sequences $\boldsymbol{r}: \bN \to \bR$ and $\tilde{\boldsymbol{r}}: \bZ \to \bR$, we can define the following  Besov and Sobolev spaces with variable smoothness.
    \begin{enumerate}[(i)]
        \item 
        ((Inhomogeneous) Weighted Bessel potential space)
        We denote by $H_p^{\boldsymbol{r}}(\bR^d,w\,\mathrm{d}x)$ the space of all  $f \in \cS'(\bR^d)$ satisfying
        $$
            \|f\|_{H_p^{\boldsymbol{r}}(\bR^d,w\,\mathrm{d}x)}:=\|S_0f\|_{L_p(\bR^d,w\,\mathrm{d}x)}+\left\|\left(\sum_{j=1}^{\infty}|2^{\boldsymbol{r}(j)}\Delta_jf|^{2}\right)^{1/2}\right\|_{L_p(\bR^d,w\,\mathrm{d}x)}<\infty,
        $$
where $\cS'(\bR^d)$ denotes the tempered distributions on $\cS(\bR^d)$. 
        \item
        ((Inhomogeneous) Weighted Besov space)
        $B_{p,q}^{\boldsymbol{r}}(\bR^d,w\,\mathrm{d}x)$ denotes the space of all  $f \in \cS'(\bR^d)$  satisfying
        $$
            \|f\|_{B_{p,q}^{\boldsymbol{r}}(\bR^d,w\,\mathrm{d}x)}:=\|S_0f\|_{L_p(\bR^d,w\,\mathrm{d}x)}+\left(\sum_{j=1}^{\infty}2^{q\boldsymbol{r}(j)}\|\Delta_jf\|_{L_p(\bR^d,w\,\mathrm{d}x)}^q\right)^{1/q}<\infty.
        $$
        \item 
        ((Homogeneous) Weighted Bessel potential space) 
We use $\dot{H}_{p,q}^{\tilde{\boldsymbol{r}}}(\bR^d,w\,\mathrm{d}x)$ to denote the space of all  $f \in \cS'(\bR^d)$ satisfying
        $$
            \|f\|_{\dot{H}_{p,q}^{\tilde{\boldsymbol{r}}}(\bR^d,w\,\mathrm{d}x)}:=\left\|\left(\sum_{j\in\bZ}|2^{\tilde{\boldsymbol{r}}(j)}\Delta_jf|^{2}\right)^{1/2}\right\|_{L_p(\bR^d,w\,\mathrm{d}x)}<\infty.
        $$
        \item
        ((Homogeneous) Weighted Besov space)
         $\dot{B}_{p,q}^{\tilde{\boldsymbol{r}}}(\bR^d,w\,\mathrm{d}x)$  denotes the space of all  $f \in \cS'(\bR^d)$ satisfying
        $$
            \|f\|_{\dot{B}_{p,q}^{\tilde{\boldsymbol{r}}}(\bR^d,w\,\mathrm{d}x)}:=\left(\sum_{j\in\bZ}2^{q\tilde{\boldsymbol{r}}(j)}\|\Delta_jf\|_{L_p(\bR^d,w\,\mathrm{d}x)}^q\right)^{1/q}<\infty.
        $$
    \end{enumerate}
 \end{defn}
\begin{rem}
Note that  if $\boldsymbol{r}(j) = sj$ for $s\in\bR$, then the space $H_p^{\boldsymbol{r}}(\bR^d, w\,\mathrm{d}x)$ is equivalent to 
the classical weighted Bessel potential space $H_p^s(\bR^d, w\,\mathrm{d}x)$ whose norm is given by 
$$
    \| f \|_{H_p^s(\bR^d,w\,\mathrm{d}x)} := \| (1 - \Delta)^{s/2} f \|_{L_p(\bR^d,w\,\mathrm{d}x)}
$$
(see Corollary \ref{classical sobo}).
\end{rem}

We can relate a constant to each $w \in A_p(\bR^d)$ to characterize sufficient smoothness of our symbols.
For each  $w \in A_p(\bR^d)$, we define
\begin{equation}
\label{2021-01-19-01}
R_{p,d}^{w} := \sup \left\{ p_0 \in (1,2] :  w \in A_{p/p_0}(\bR^d) \right\}
\end{equation}
and say that $R_{p,d}^{w}$ is \textbf{the regularity constant of the weight $w\in A_p(\bR^d)$}.
Due to  the reverse H\"older property of Muckenhoupt's class, $R_{p,d}^{w}$ is well-defined, \textit{i.e.} $R_{p,d}^{w} \in (1,2]$.
We assume that $\psi$ has a $\left( \left\lfloor \frac{d}{R_{p,d}^{w}}\right\rfloor+2\right)$-times regular upper bound with $(\gamma,M)$.

Our main result is the following sufficient condition for the existence and uniqueness of a solution in weighted $L_p$-spaces with variable smoothness.
\begin{thm}
    \label{22.12.27.16.53}
    Let $T \in (0,\infty)$, $p\in(1,\infty)$, $q\in(0,\infty)$, $w\in A_p(\bR^d)$, $\mu$ be a nonnegative Borel measure on $(0,\infty)$, and 
    $\boldsymbol{r}:\bZ\to(-\infty,\infty)$, $\boldsymbol{\mu}:\bZ\to(-\infty,\infty)$ be sequences having a controlled difference.
    Suppose that $\psi$ is a symbol satisfying the ellipticity condition with $(\gamma,\kappa)$ and having the $\left( \left\lfloor \frac{d}{R_{p,d}^{w}}\right\rfloor+2\right)$-times regular upper bound with $(\gamma,M)$ $($see \eqref{condi:ellipticity}, \eqref{condi:reg ubound}, and \eqref{2021-01-19-01}$)$. 
    Additionally, assume that the Laplace transform of $\mu$ is controlled by a sequence $\boldsymbol{\mu}$ in a $\gamma$-dyadic way with parameter $a\in(0,\infty)$, \textit{i.e.}
\begin{align}
									\label{20230210 01}
\cL_{\mu}(2^{\gamma j})
:= \int_0^\infty \exp\left( - 2^{\gamma j} t \right) \mu(\mathrm{d}t) \leq N_{\cL_\mu} \cdot 2^{j\gamma a}2^{-\boldsymbol{\mu}(j)},\quad \forall j\in\bZ.
\end{align}
Then for any $u_0\in B_{p,q}^{\boldsymbol{r}-\frac{\boldsymbol{\mu}}{q}}(\bR^d,w\,\mathrm{d}x)$, there exists a unique solution 
$$
u\in L_q\left((0,T),t^a \mu\left(\frac{(1\wedge q) \kappa}{16^\gamma}\mathrm{d}t\right);H_p^{\boldsymbol{r}}(\bR^d,w\,\mathrm{d}x)\right)
$$
to \eqref{eqn:model eqn}.
    Moreover, if $u \in C_{p}^{1,\infty}([0,T]\times\bR^d) $ and $u_0 \in C_c^\infty(\bR^d)$, then the following \textit{a priori} estimates hold:
    \begin{align}
							\label{main a priori est 0}
        \int_{0}^T\left\|u(t,\cdot)\right\|_{H_p^{\boldsymbol{r}}(\bR^d,w\,\mathrm{d}x)}^q  
        t^a\mu\left(\frac{(1\wedge q) \kappa }{16^\gamma}\mathrm{d}t\right) 
       \leq N(1+\mu_{a,T,\kappa,\gamma,q})\|u_0\|^q_{B_{p,q}^{\boldsymbol{r}-\frac{\boldsymbol{\mu}}{q}}(\bR^d,w\,\mathrm{d}x)},
    \end{align}
    and
    \begin{equation}
    \label{main a priori est}
        \begin{aligned}
        &\int_{0}^T\left\|\psi(t,-i\nabla)u(t,\cdot)\right\|_{H_p^{\boldsymbol{r}-\boldsymbol{\gamma}}(\bR^d,w\,\mathrm{d}x)}^q t^a \mu\left(\frac{(1\wedge q) \kappa}{16^\gamma}\mathrm{d}t\right)\\
        &+\int_{0}^T\left\|\Delta^{\gamma/2}u(t,\cdot)\right\|_{H_p^{\boldsymbol{r}-\boldsymbol{\gamma}}(\bR^d,w\,\mathrm{d}x)}^q t^a \mu\left(\frac{(1\wedge q) \kappa}{16^\gamma}\mathrm{d}t\right)\\
        \leq& N\|u_0\|^q_{\dot B_{p,q}^{\boldsymbol{r}-\frac{\boldsymbol{\mu}}{q}}(\bR^d,w\,\mathrm{d}x)},
    \end{aligned}
    \end{equation}
    where $[w]_{A_p(\bR^d)}\leq K$, $N=N(a,\|\boldsymbol{r}\|_d, \|\boldsymbol{\mu}\|_d, N_{\cL_\mu},d,\gamma,\kappa,K,M,p,q)$, and
    \begin{align*}
\mu_{a,T,\kappa,\gamma,q}=\int_0^T  t^a\mu\left(\frac{(1\wedge q) \kappa}{16^\gamma}\mathrm{d}t\right).
\end{align*}
\end{thm}

\begin{rem}
    Suppose that the symbol $\psi$ admits a $\left( \left\lfloor \frac{d}{R_{p,d}^{w}} \right\rfloor + 2 \right)$-times regular upper bound with parameters $(\gamma, M)$. 
    Then there exists a constant $N>0$ such that
    \begin{align}
        \label{25.06.14.19.43}
        \|\psi(t, -i\nabla) f\|_{L_p(\mathbb{R}^d, w\,\mathrm{d}x)} \leq N \|\Delta^{\gamma/2} f\|_{L_p(\mathbb{R}^d, w\,\mathrm{d}x)}, \quad \forall f \in H_p^{\gamma}(\mathbb{R}^d, w\,\mathrm{d}x).
    \end{align}
    Indeed, define
    $$
        m(t, \xi) := \frac{\psi(t, \xi)}{|\xi|^{\gamma}}.
    $$
    Then $m$ satisfies the following uniform bound:
    $$
        |D^\alpha_\xi m(t, \xi)| \leq N(\alpha, \gamma, M) |\xi|^{-|\alpha|}
    $$
    for all multi-indices $\alpha \in \mathbb{N}^d$ with $|\alpha| \leq \left\lfloor \frac{d}{R_{p,d}^{w}} \right\rfloor + 2$. 
    Since
    $$
        \frac{d}{\left\lfloor \frac{d}{R_{p,d}^{w}} \right\rfloor + 2} < R_{p,d}^{w} \leq p,
    $$
    we can apply \cite[Theorem 4]{kurtz1979results} to obtain the estimate \eqref{25.06.14.19.43}.

    Furthermore, if $\psi$ also satisfies the ellipticity condition with  $(\gamma, M)$, then considering the symbol
    $$
        m(t, \xi) := \frac{|\xi|^{\gamma}}{\psi(t, \xi)},
    $$
    one can similarly deduce from \cite[Theorem 4]{kurtz1979results} that
    \begin{align}
        \label{25.06.14.19.49}
        \|\Delta^{\gamma/2} f\|_{L_p(\mathbb{R}^d, w\,\mathrm{d}x)} \leq N \|\psi(t, -i\nabla) f\|_{L_p(\mathbb{R}^d, w\,\mathrm{d}x)}, \quad \forall f \in H_p^{\gamma}(\mathbb{R}^d, w\,\mathrm{d}x).
    \end{align}
    Therefore, under the regularity and ellipticity assumptions, the operator $\psi(t, -i\nabla)$ is equivalent to the fractional Laplacian $\Delta^{\gamma/2}$ in $L_p(\mathbb{R}^d, w\,\mathrm{d}x)$.
\end{rem}

The proof of this main theorem will be given in Section \ref{pf main thm}.
We finish the section by presenting the definition of our solution.
\begin{defn}[Solution]
								\label{def sol}
Assume that all parameters are given as in Theorem \ref{22.12.27.16.53}.
We say that $u \in L_q\left((0,T),t^a \mu\left(\frac{(1\wedge q) \kappa }{16^\gamma}\mathrm{d}t\right);H_p^{\boldsymbol{r}}(\bR^d,w\,\mathrm{d}x)\right)$ is a solution to \eqref{eqn:model eqn} if 
there exists a sequence $u_n\in C_p^{1,\infty}([0,T]\times\bR^d)$ such that $u_n(0,\cdot)\in C_c^{\infty}(\bR^d)$ and
\begin{equation}
							\label{202301014 01}
\begin{gathered}
\partial_tu_n(t,x)=\psi(t,-i\nabla)u_n(t,x),\quad \forall (t,x)\in(0,T)\times\bR^d,\\
u_n(0,\cdot)\to u_0 \quad\text{in}\quad B_{p,q}^{\boldsymbol{r}-\frac{\boldsymbol{\mu}}{q}}(\bR^d,w\,\mathrm{d}x),\\
u_n\to u\quad \text{in}\quad L_q\left((0,T),t^a \mu\left(\frac{(1\wedge q) \kappa }{16^\gamma}\mathrm{d}t\right);H_p^{\boldsymbol{r}}(\bR^d,w\,\mathrm{d}x)\right)
\end{gathered}
\end{equation}
as $n\to\infty$. 
Due to this definition, we have \eqref{main a priori est 0} for any solution 
$$
u\in L_q\left((0,T),t^a \mu\left(\frac{(1\wedge q) \kappa}{16^\gamma}\mathrm{d}t\right);H_p^{\boldsymbol{r}}(\bR^d,w\,\mathrm{d}x)\right)
$$
to \eqref{eqn:model eqn} with the corresponding $u_0 \in B_{p,q}^{\boldsymbol{r}-\frac{\boldsymbol{\mu}}{q}}(\bR^d,w\,\mathrm{d}x)$.
\end{defn}
\begin{rem}
Since our equation is linear, by using \textit{a priori} estimate, we may consider that our solution is a strong solution in $L_q\left((0,T),t^a \mu\left(\frac{(1\wedge q) \kappa}{16^\gamma} \mathrm{d}t\right);H_p^{\boldsymbol{r}}(\bR^d,w\,\mathrm{d}x)\right)$.
Indeed, due to \eqref{main a priori est}, for all $n, m \in \bN$, we have 
        \begin{align*}
&\int_{0}^T\left\|\partial_t(u_n - u_m)(t,\cdot)\right\|_{H_p^{\boldsymbol{r}-\boldsymbol{\gamma}}(\bR^d,w\,\mathrm{d}x)}^q t^a \mu\left(\frac{(1\wedge q) \kappa}{16^\gamma}\mathrm{d}t\right)\\
        =&\int_{0}^T\left\|\psi(t,-i\nabla)(u_n - u_m)(t,\cdot)\right\|_{H_p^{\boldsymbol{r}-\boldsymbol{\gamma}}(\bR^d,w\,\mathrm{d}x)}^q t^a \mu\left(\frac{(1\wedge q) \kappa}{16^\gamma}\mathrm{d}t\right)\\
        \leq& N\|u_n(0,\cdot) - u_m(0,\cdot)\|^q_{B_{p,q}^{\boldsymbol{r}-\frac{\boldsymbol{\mu}}{q}}(\bR^d,w\,\mathrm{d}x)}.
    \end{align*}
Thus defining  $\partial_tu$ and $\psi(t,-i\nabla)u$ as the limits of $\partial_tu_n$ and $\psi(t,-i\nabla)u_n$ in 
$$
L_q\left((0,T),t^a \mu\left(\frac{(1\wedge q) \kappa}{16^\gamma}\mathrm{d}t\right);H_p^{\boldsymbol{r}-\boldsymbol{\gamma}}(\bR^d,w\,\mathrm{d}x)\right),
$$
respectively, we understand that our solution $u$ is a strong solution. 
However, the limits $\partial_tu$ and $\psi(t,-i\nabla)u$ are not well-defined in general, which means that they could be dependent on the choice of a sequence $u_n$ approximating to $u$.  
To make the limits $\partial_tu$ and $\psi(t,-i\nabla)u$ uniquely determined, we need the condition $\boldsymbol{r}-\boldsymbol{\gamma} \geq 0$ and another condition on a measure $\mu$, which will be specified in the next remark. 
\end{rem}

\begin{rem}
To show that our solution becomes the classical weak solution, we need an extra condition on a measure $\mu$.
For simplicity, we may assume that the scaling constant $\frac{(1\wedge q) \kappa}{16^\gamma}$ is 1.
Additionally, assume that the measure $\mu(\mathrm{d}t)$ has a density $\mu(t)$ and $t^{-a} (1/\mu)(t)$ is locally in $L_{q/(q-1)}$, \textit{i.e.} $\mu(\mathrm{d}t)=\mu(t)\mathrm{d}t$ and 
$\int_0^{T_0} t^{-\frac{aq}{q-1}} (\mu(t))^{-\frac{q}{q-1}} \mathrm{d}t < \infty$ for all $T_0 \in (0,T)$.
Let $u$ be a solution to \eqref{eqn:model eqn}. Then there exists a $u_n\in C_p^{1,\infty}([0,T]\times\bR^d)$  so that \eqref{202301014 01} holds.
Then for any $n \in \bN$ and $\varphi \in C_c^\infty ( (0,T) \times \bR^d)$, we have
\begin{align}
						\label{20230114 02}
-\int_0^T\int_{\bR^d} u_n(t,x) \partial_t\varphi(t,x) \mathrm{d}x\mathrm{d}t   = \int_0^T \int_{\bR^d} u_n(t,x) \overline{\psi}(t,-i\nabla) \varphi(t,x) \mathrm{d}x\mathrm{d}t ,
\end{align}
where 
\begin{align*}
\overline{\psi}(t,-i\nabla)u(t,x)= \cF^{-1}\left[ \overline{\psi(t,\cdot)} \cF[u(t,\cdot)] \right] (x)
\end{align*}
and $\overline{\psi(t,\xi)}$ denotes the complex conjugate of $\psi(t,\xi)$.
Moreover, applying H\"older's inequality, we have
\begin{align*}
&\left|\int_0^T\int_{\bR^d} u_n(t,x) \partial_t\varphi(t,x)  \mathrm{d}x \mathrm{d}t -\int_0^T\int_{\bR^d} u(t,x) \partial_t\varphi(t,x) \mathrm{d}x \mathrm{d}t \right| \\
&\leq  N\| u_n -u\|_{L_q((0,T),t^a \mu\,\mathrm{d}t;L_p(\bR^d,w\,\mathrm{d}x))} \left(\int_0^{T_0} t^{-\frac{aq}{q-1}} (\mu(t))^{-\frac{q}{q-1}} \mathrm{d}t\right)^{\frac{q-1}{q}},
\end{align*}
where $T_0$ is a positive number so that $\partial_t \varphi(t,x) = 0$ for all  $t \in [T_0,T)$.
Thus
\begin{align*}
\lim_{n \to \infty}\int_0^T\int_{\bR^d} u_n(t,x) \partial_t\varphi(t,x)\mathrm{d}x \mathrm{d}t 
=\int_0^T\int_{\bR^d} u(t,x) \partial_t\varphi(t,x) \mathrm{d}x \mathrm{d}t.
\end{align*}
Similarly, 
\begin{align*}
\lim_{n\to \infty}\int_0^T \int_{\bR^d} u_n(t,x) \overline{\psi}(t,-i\nabla) \varphi(t,x) \mathrm{d}x \mathrm{d}t
=\int_0^T \int_{\bR^d} u(t,x) \overline{\psi}(t,-i\nabla) \varphi(t,x) \mathrm{d}x \mathrm{d}t.
\end{align*}
Finally, taking the limit in \eqref{20230114 02}, we show that our solution becomes a classical weak solution, \textit{i.e.}
\begin{align}
						\label{20230114 03}
-\int_0^T\int_{\bR^d} u(t,x) \partial_t\varphi(t,x) \mathrm{d}x \mathrm{d}t  = \int_0^T \int_{\bR^d} u(t,x) \overline{\psi}(t,-i\nabla) \varphi(t,x) \mathrm{d}x \mathrm{d}t
\end{align}
for all $\varphi \in C_c^\infty ( (0,T) \times \bR^d)$.
\end{rem}

\begin{rem}
We emphasize that \eqref{main a priori est 0} and \eqref{main a priori est} can be distinguished. 
In other words, there could be different two measures $\mu_1$ and $\mu_2$ so that $\boldsymbol{\mu_1}(j) = \boldsymbol{\mu_2}(j)$ for all $j \in \bN$ but there exists a $k \in \bZ$ such that $\boldsymbol{\mu_1}(k) \neq \boldsymbol{\mu_2}(k)$.

Let $\mu_1(\mathrm{d}t) := |t|^{b_1}\mathrm{d}t$ and $\mu_2(\mathrm{d}t) := \left(|t|^{b_1} + |t|^{b_2}\right)\mathrm{d}t$. 
Since both $|t|^{b_1}$ and $|t|^{b_1}+|t|^{b_2}$ are in $A_{\infty}(\bR)$, by Lemma \ref{24.02.16.16.11}-$(ii)$ and $(iii)$, we have
\begin{align*}
&\mu_1((0,2^{-j\gamma}))\simeq2^{-j\gamma(b_1+1)}\implies\boldsymbol{\mu_1}(j)=  j \gamma (a+b_1+1)  \quad \forall j \in \bZ,\\
&\mu_2((0,2^{-j\gamma}))\simeq2^{-j \gamma (b_1+1)} +  2^{-j \gamma (b_2+1)}.
\end{align*}
If $-1<b_1 < b_2  <\infty$, then
\begin{align*}
\mu_2((0,2^{-j\gamma})) \lesssim
\begin{cases}
&2^{-j \gamma (b_1+1)}  \quad \text{if}~ j \in \bN, \\
&2^{-j \gamma (b_2+1)}  \quad \text{if $j$ is a non-positive integer}.
\end{cases}
\end{align*}
Therefore, if we put 
\begin{align*}
\boldsymbol{\mu_2}(j):=
\begin{cases}
&j \gamma (a+b_1+1) \quad \text{if}~ j \in \bN, \\
&j \gamma (a+b_2+1)  \quad \text{if $j$ is a non-positive integer},
\end{cases}
\end{align*}
then $\boldsymbol{\mu_1}(j)=\boldsymbol{\mu_2}(j)$ for all $j\in\bN$, but, $\boldsymbol{\mu_1}\neq\boldsymbol{\mu_2}$.
\end{rem}

\vspace{3mm}

\mysection{Examples and Applications}
\label{22.12.28.13.34}
In this section, we present some sufficient conditions for $(\mu,\boldsymbol{\mu})$ satisfying \eqref{20230210 01} and corresponding well-posedness results.
We also present applications of our results.
\begin{lem}
\label{24.02.16.16.11}
    Let $\mu$ be a nonnegative Borel measure defined on $(0,\infty)$ and $\mu^{-a_0}(\mathrm{d}t):=t^{-a_0}\mu(\mathrm{d}t)$.
    Suppose that $a_0$ is a nonnegative constant that the Laplace transform $\mathcal{L}_{\mu^{-a_0}}$ is well-defined.
 \begin{enumerate}[(i)]
     \item If
     \begin{align}
									\label{20230210 03}
\inf_{j \in \bZ}\frac{\cL_{\mu^{-a_0}}(2^{\gamma (j+1)})}{\cL_{\mu^{-a_0}}(2^{\gamma j})}  >0,
\end{align}
then $
\boldsymbol{\mu}(j) := j\gamma a  -\log_2\left( \cL_{\mu^{-a_0}}(2^{\gamma j}) \right) 
$ is a sequence having the controlled difference, and $(\mu^{-a_0},\boldsymbol{\mu})$ satisfies \eqref{20230210 01} for $a_0\in[0,\infty)$.
\item Suppose that for any $k\in(1,\infty)$, there exist positive constants $b_k$ and $B_k$ such that
\begin{equation}
\label{22.04.22.16.39}
    0<b_k\leq \frac{\mu((0,\theta))}{\mu((0,k\theta))}\leq B_k<1,\quad \forall \theta\in(0,\infty).
\end{equation}
Then $\boldsymbol{\mu}(j):=(j\gamma (a-a_0)-\log_{2}(\mu(0,2^{-j\gamma})))$ is a sequence having the controlled difference, and $(\mu^{-a_0},\boldsymbol{\mu})$ satisfies \eqref{20230210 01} for $a_0\in[0,-\log_2(B_2))$.
\item If $w'\in A_{\infty}(\bR)$, then $w'$ satisfies \eqref{22.04.22.16.39}.
\item Let $\delta(t)$ be a nonnegative locally integrable function, $\eta(t):=\int_0^t\delta(s)\mathrm{d}s$, $\eta^{-1}(t):=\inf\{s>0:\eta(s)\geq t\}$  and $\mu(\mathrm{d}t):=1_{[0,T']}(t)/(\delta(\eta^{-1}))\mathrm{d}t$.
Suppose that there exist positive constants $\beta,t_0,N_0$ such that
\begin{equation}
\label{24.02.06.14.00}
    \left|\left\{t\in[0,t_0]:h\leq \eta(t)<4h\right\}\right|\leq N_0h^{1/\beta}.
\end{equation}
Then $\boldsymbol{\mu}(j):=j\gamma (a-a_0)+\frac{j\gamma}{\beta}$ is a sequence having the controlled difference, and $(\mu^{-a_0},\boldsymbol{\mu})$ satisfies \eqref{20230210 01} for $a_0\in[0,1/\beta)$.
\end{enumerate}   
\end{lem}
\begin{proof}
$(i)$ Since
\begin{align*}
\cL_{\mu^{-a_0}}(2^{\gamma j})
=2^{\log_2\left( \cL_{\mu^{-a_0}}(2^{\gamma j}) \right)}
=2^{j\gamma a} 2^{-j\gamma a + \log_2\left( \cL_{\mu^{-a_0}}(2^{\gamma j}) \right)}=2^{j\gamma a}2^{-\boldsymbol{\mu}(j)},
\end{align*}
\eqref{20230210 01} holds with $N_{\cL_{\mu^{-a_0}}}=1$.
The sequence $\boldsymbol{\mu}$ has a controlled difference due to \eqref{20230210 03} 
and the fact that $\cL_{\mu}(2^{\gamma j})$ is non-increasing as $j$ increases.

$(ii)$ Choose   
\begin{align}
								\label{20230128 100}
a_0 \in \left[0, -\log_2B_2 \right).
\end{align}
We claim that 
\begin{equation}
\label{22.10.31.16.02}
\cL_{\mu^{-a_0}}(\lambda) \simeq  \lambda^{a_0}\mu((0,1/\lambda)) \quad \forall \lambda \in (0,\infty).
\end{equation}
Indeed, by \eqref{22.04.22.16.39} and the ratio test,
\begin{equation}
\label{2023012401}
\begin{aligned}
\int_{1}^{\infty}\mathrm{e}^{-t}\mu^{-a_0}(dt)
&\leq \int_{1}^{\infty}\mathrm{e}^{-t}\mu(dt)
\\
&=\sum_{n=1}^{\infty}\int_{2^{n-1}}^{2^n}\mathrm{e}^{- t}\mu(dt)\\
&\leq \sum_{n=1}^{\infty}\mathrm{e}^{-2^{n-1}}\mu((0,2^n)) \\
&\leq \mu((0,1))\sum_{n=0}^{\infty}\mathrm{e}^{- 2^{n-1}}b^{-n}_2<\infty.
\end{aligned}
\end{equation}
Similarly, we have
\begin{equation}
\label{22.04.23.19.49}
\begin{aligned}
    \int_{0}^{1}\mathrm{e}^{-t}\mu^{-a_0}(\mathrm{d}t)&\leq \sum_{n=0}^{\infty}\int_{2^{-n-1}}^{2^{-n}}t^{-a_0}\mathrm{d}t\\
    &\leq \sum_{n=0}^{\infty}2^{(n+1)a_0}\mu((0,2^{-n}))\leq \mu((0,1))\sum_{n=0}^{\infty}2^{(n+1)a_0}B_2^n.
\end{aligned}
\end{equation}
Due to the choice of $a_0$ in \eqref{20230128 100}, it  is obvious that $2^{a_0}B_2 < 1$ and the last term in \eqref{22.04.23.19.49} converges. 
Therefore combining \eqref{2023012401} and \eqref{22.04.23.19.49}, we obtain
\begin{align}
								\label{20230124 100}
\cL_{\mu^{-a_0}}(1) \simeq \mu((0,1)).
\end{align}
Next, we use the scaling property. 
For $\lambda \in (0,\infty)$, define the measures as
$$
\mu_{1/\lambda} (\mathrm{d}t) := \mu \left(\frac{1}{\lambda} \,\mathrm{d}t\right)
$$
and
$$
\mu^{-a_0}_{1/\lambda} (\mathrm{d}t) := \mu^{-a_0} \left(\frac{1}{\lambda}\, \mathrm{d}t\right).
$$
Then for any $k \in (1,\infty)$,
\begin{equation*}
    0<b_k\leq \frac{\mu_{1/\lambda}((0, \theta))}{\mu_{1/\lambda}((0,k \theta))}\leq B_k<1,\quad \forall \theta\in(0,\infty).
\end{equation*}
Thus applying \eqref{20230124 100} with $\mu_{1/\lambda}$ instead of $\mu(\mathrm{d}t)$, we prove the claim that
$$
\cL_{\mu^{-a_0}}(\lambda)=\lambda^{a_0} \cL_{\mu^{-a_0}_{1/\lambda}}(1) \simeq \lambda^{a_0} \mu_{1/\lambda}((0,1))= \lambda^{a_0} \mu((0,1/\lambda)).
$$
Therefore, for all $a \in [0,\infty)$ and $j \in \bZ$,
\begin{align*}
\cL_{\mu^{-a_0}}(2^{j \gamma})  \simeq 2^{j\gamma a_0}\mu((0,2^{-j\gamma})) =  2^{j  \gamma a} \cdot 2^{-\left( j \gamma (a-a_0)-\log_2(\mu((0,2^{-j\gamma}))) \right)}
\end{align*}
In other words,
the Laplace transform of $\mu^{-a_0}$ with $a_0 \in \left[0, -\log_2B_2 \right)$ is controlled by a sequence 
\begin{align}
							\label{20230124 05}
\boldsymbol{\mu}(j)=\left( j \gamma (a-a_0)-\log_2(\mu((0,2^{-j\gamma}))) \right)
\end{align}
 in a $\gamma$-dyadic way with a parameter $a$.

$(iii)$
Since $w'\in A_{\infty}(\bR)$ , there exists $\nu\in(1,\infty)$ such that $w'\in A_{\nu}(\bR)$. By \cite[Proposition 7.1.5 and Lemma 7.2.1]{grafakos2014classical},
$$
b_k:=\frac{1}{k^{\nu}[w']_{A_{\nu}(\bR)}}\leq \frac{\mu((0,\theta))}{\mu((0,k\theta))}\leq \left(1-\frac{(1-k^{-1})^{\nu}}{[w']_{A_{\nu}(\bR)}}\right)=:B_k,\quad \forall k\in(1,\infty),\, \theta \in (0,\infty).
$$

$(iv)$ It suffices to prove that for any $T'>0$, there exists a constant $N_{T'}>0$ such that
\begin{equation}
\label{22.10.10.16.50}
    \eta(t)\geq N_{T'}t^{\beta},\quad \forall t\in[0,T'].
\end{equation}
Indeed, if we have \eqref{22.10.10.16.50}, then for $a_0\in[0,1/\beta)$,
\begin{align*}
    \cL_{\mu^{-a_0}}(2^{j\gamma})&=\int_0^{T'}\mathrm{e}^{-2^{j\gamma}t}\frac{t^{-a_0}}{\delta(\eta^{-1}(t))}\mathrm{d}t\\
    &=\int_0^{\eta^{-1}(T')}\mathrm{e}^{-2^{j\gamma}\eta(s)}(\eta(s))^{-a_0}\mathrm{d}s\\
    &\leq N(N_0,\beta,T')\int_0^{\infty}\frac{\mathrm{e}^{-2^{j\gamma}s^{\beta}}}{s^{a
_0\beta}}\mathrm{d}s=N2^{j\gamma a}2^{-\boldsymbol{\mu}(j)}.
\end{align*}
Since $T'$ is finite, it suffices to prove only that \eqref{22.10.10.16.50} holds near zero. Note that \eqref{24.02.06.14.00} can be rewritten as
$$
\left|\left\{t\in[0,t_0]:h\leq \eta(t)<4h\right\}\right|=|[0,t_0]\cap[\eta^{-1}(h),\eta^{-1}(4h))|\leq N_0 h^{1/\beta},\quad \forall h>0.
$$
Observe that for $\eta^{-1}(h)\leq t_0$
\[
    |[0,t_0]\cap[\eta^{-1}(h/4),\eta^{-1}(h))|=|[\eta^{-1}(h/4),\eta^{-1}(h))|.
\]
Then we have
\begin{align*}
    \eta^{-1}(h) 
    &= \eta^{-1}(h) - \eta^{-1}(h/4) + \eta^{-1}(h/4) - \eta^{-1}(h/4^2) +\cdots\\
    &= \sum_{k=0}^\infty \eta^{-1}(4^{-k}h) - \eta^{-1}(4^{-1-k}h)\\
    &\leq \sum_{k=0}^\infty N_0 (4^{-1-k} h)^{1/\beta}\\
    &= N(N_0, \beta) h^{1/\beta}.
\end{align*}
In a word, $\eta^{-1}(h)\leq N(N_0,\beta)h^{1/\beta}$ for all $\eta^{-1}(h)\leq t_0$. Since $\eta(\eta^{-1}(t))=t$ for all $t>0$, we put $\eta(\eta^{-1}(h))$ instead of $h$ and take $t=\eta^{-1}(h)$ to obtain
\[
    N(N_0,\beta)^{-\beta}t^{\beta} \leq \eta(t),\quad \forall t = \eta^{-1}(h)\leq t_0.
\]
The lemma is proved.
\end{proof}

\begin{rem}
    The condition \eqref{24.02.06.14.00} was introduced by I. Kim, K.-H. Kim \cite{kim2018}. 
    In the same paper, the authors showed that if \eqref{22.10.10.16.50} holds for some $\beta>0$ then the condition \eqref{24.02.06.14.00} is verified for the same $\beta$, \cite[Example 2.13]{kim2018}.
    Thus, following the proof of $(iv)$ of Lemma~\ref{24.02.16.16.11}, the condition \eqref{24.02.06.14.00} and the inequality \eqref{22.10.10.16.50} are equivalent.
\end{rem}

\begin{rem}
\label{24.02.19.11.39}
In most results on zero initial value inhomogeneous problems (\textit{e.g.} \cite{Choi_Kim2022,Dong_Kim2018,Dong_Kim2021,Dong_Liu2022}), the solution $u$ is in $L_q((0,T),w'\mathrm{d}t;L_p(\bR^d,w\,\mathrm{d}x))$.
Here $p,q\in(1,\infty)$, $w'\in A_q(\bR)$ and $w\in A_p(\bR^d)$.
We emphasize that Theorem \ref{22.12.27.16.53} and Lemma \ref{24.02.16.16.11} also hold for $q\in(0,\infty)$ and $w'\in A_{\infty}(\bR)$.
Moreover, we also emphasize that our measure $\mu(\mathrm{d}t)$ does not need to have a density. 
In other words, the class of our measures $\mu$ is larger than Muckenhoupt's class.
Indeed, for any $t_0 \in (0,\infty)$, consider the Dirac delta measure centered at $t_0$, \textit{i.e.} $\mu(\mathrm{d}t) = \delta_{t_0}(\mathrm{d}t)$.
Then \eqref{20230210 03} holds since
\begin{align*}
\inf_{j \in \bZ}\frac{\cL_{\mu}(2^{\gamma (j+1)})}{\cL_{\mu}(2^{\gamma j})} 
=\exp(2^{-\gamma}t_0).
\end{align*}
Therefore, taking 
\begin{align*}
\boldsymbol{\mu}(j) = j\gamma a -\log_2\left(\exp\left( - 2^{\gamma j} t_0 \right) \right),
\end{align*}
we can apply Theorem \ref{22.12.27.16.53} due to Lemma \ref{24.02.16.16.11}.
\end{rem}

Combining \cite[Theorem 2.14]{Choi_Kim2022}, we can also handle the following inhomogeneous problems with non-zero initial conditions.
\begin{corollary}
\label{24.02.19.11.10}
    Let $T \in (0,\infty)$, $p\in(1,\infty)$, $q\in(1,\infty)$, $w'\in A_{q}(\bR)$, $w\in A_p(\bR^d)$,  and 
    $\boldsymbol{r}:\bZ\to(-\infty,\infty)$ be a sequence having a controlled difference.
    Suppose that $[w]_{A_{p}(\bR^d)}\leq K$ and $[w']_{A_{q}(\bR)}\leq K_0$. 
Additionally, assume that $\psi$ is a symbol satisfying an ellipticity condition with $(\gamma,\kappa)$ and having a $\left( \left\lfloor \frac{d}{R_{p,d}^{w}}\right\rfloor \vee \left\lfloor \frac{d}{R_{q,1}^{w'}}\right\rfloor+2\right)$-times regular upper bound with $(\gamma,M)$. 
Then for any
    $$
    u_0\in B_{p,q}^{\boldsymbol{r}-\boldsymbol{w}'/q}(\bR^d,w\,\mathrm{d}x)~\text{and}~ f\in L_q((0,T),w'\,\mathrm{d}t;H_p^{\boldsymbol{r}-\boldsymbol{\gamma}}(\bR^d,w\,\mathrm{d}x)),
    $$
    there exists a unique solution 
    $$
u\in L_q\left((0,T),w'\,\mathrm{d}t;H_p^{\boldsymbol{r}}(\bR^d,w\,\mathrm{d}x)\right)
$$ to the equation
    \begin{equation}
						\label{inhomo problem}
    \begin{cases}
        \p_tu(t,x)=\psi(t,-i\nabla)u(t,x)+f(t,x),\quad &(t,x)\in(0,T)\times\bR^d,\\
        u(0,x)=u_0(x),\quad &x\in\bR^d,
    \end{cases}
\end{equation}
where  $\boldsymbol{\gamma}(j) = \gamma j$ and   
\begin{align*}
\boldsymbol{w'}(j)=-\log_2\left(\int_0^{2^{-j\gamma}} w'(t)\mathrm{d}t \right).
\end{align*}
    Moreover, $u$ satisfies
    \begin{equation}
    \label{20230124 08}
      \begin{aligned}
&\int_{0}^T\left\|u(t,\cdot)\right\|_{H_p^{\boldsymbol{r}}(\bR^d,w\,\mathrm{d}x)}^qw'(t)\mathrm{d}t
       \\
       \leq& N(1+T)^q\left(\|u_0\|^q_{B_{p,q}^{\boldsymbol{r}-\boldsymbol{w}'/q}(\bR^d,w\,\mathrm{d}x)}+\int_0^T\|f(t,\cdot)\|^q_{H_p^{\boldsymbol{r}-\boldsymbol{\gamma}}(\bR^d,w\,\mathrm{d}x)} w'(t)\mathrm{d}t\right)
    \end{aligned}
    \end{equation}
    where $N=N(\|\boldsymbol{r}\|_d,C_{w'},d,\gamma,\kappa,K,K_0,M,p,q)$.
\end{corollary}
\begin{proof}
Since the uniqueness comes from the homogeneous case (Theorem \ref{22.12.27.16.53}), we only prove the existence of a solution $u$ satisfying \eqref{20230124 08}. Moreover, due to Proposition \ref{22.05.03.11.34}, we may assume that $\boldsymbol{r}=\boldsymbol{\gamma}$.
We set $\mu(\mathrm{d}t):= w'(t) \mathrm{d}t$.
Let $a_0 \in \left(0, -\log_2B_2 \right)$
and recall
$$
\mu^{-a_0}(\mathrm{d}t):= t^{-a_0}w'\left( t\right)\mathrm{d}t
$$
where
\begin{align*}
B_2=\left(1-\frac{(1-2^{-1})^{q}}{[w']_{A_{q}(\bR)}}\right).
\end{align*}
Then due to Lemma \ref{24.02.16.16.11} $(ii)$ and $(iii)$, the Laplace transform of $\mu^{-a_0}$ is controlled by a sequence 
\begin{align*}
\boldsymbol{w'}(j):=\boldsymbol{\mu^{-a_0}}(j)
=-\log_2(\mu((0,2^{-j\gamma})))
\end{align*}
 in a $\gamma$-dyadic way with parameter $a_0$.
Thus by Theorem \ref{22.12.27.16.53}, there exists a solution $u^1$ to the equation 
    \begin{equation*}
    \begin{cases}
        \p_tu^1(t,x)=\psi(t,-i\nabla)u^1(t,x),\quad &(t,x)\in(0,T)\times\bR^d,\\
        u^1(0,x)=u_0(x),\quad &x\in\bR^d
    \end{cases}
\end{equation*}
such that
\begin{equation}
\label{20230124 10}
    \begin{aligned}
					\int_{0}^T\left\|u^1(t,\cdot)\right\|_{H_p^{\boldsymbol{r}}(\bR^d,w\,\mathrm{d}x)}^q  
       w'(t)\mathrm{d}t
&\lesssim \int_{0}^T\left\|u^1(t,\cdot)\right\|_{H_p^{\boldsymbol{r}}(\bR^d,w\,\mathrm{d}x)}^q  
        t^{a_0}\mu^{-a_0}\left(\frac{(1\wedge q) \kappa}{16^\gamma}\mathrm{d}t\right)  \\
&\leq N(1+\mu^{-a_0}_{a_0,T,\kappa,\gamma,q})\|u_0\|^q_{B_{p,q}^{\boldsymbol{r}-\frac{\boldsymbol{w'}}{q}}(\bR^d,w\,\mathrm{d}x)},
    \end{aligned}
\end{equation}
where
\begin{equation}
\label{20230124 10-2}
    \begin{aligned}
\mu^{-a_0}_{a_0,T,\kappa,\gamma,q}
&=\int_0^T  t^{a_0}\mu^{-a_0}\left(\frac{(1\wedge q) \kappa}{16^\gamma}\mathrm{d}t\right)
\\
&= \left(\frac{(1\wedge q) \kappa}{16^\gamma}\right)^{-a_0}\int_0^{\frac{(1\wedge q) \kappa}{16^\gamma}T}  w'\left(t\right)\mathrm{d}t \\
&= \left(\frac{(1\wedge q) \kappa}{16^\gamma}\right)^{-a_0} \mu\left(\left(0,\frac{(1\wedge q) \kappa}{16^\gamma}T\right)\right).
\end{aligned}
\end{equation}
Moreover, due to \cite[Proposition 7.1.5 and Lemma 7.2.1]{grafakos2014classical},
    \begin{align}
							\label{20230124 11}
   \mu\left(\left(0,\frac{(1\wedge q) \kappa}{16^\gamma}T\right)\right) \leq \mu(\lambda(0,1))\leq \lambda^{q}[w']_{A_q(\bR)}\mu((0,1)),\quad \lambda:=\frac{(1\wedge q) \kappa}{16^\gamma}(1+T).
    \end{align}
On the other hand, by \cite[Theorem 2.14.]{Choi_Kim2022}, there exists a solution $u^2$ to 
    \begin{equation*}
    \begin{cases}
        \p_tu^2(t,x)=\psi(t,-i\nabla)u^2(t,x)+f(t,x),\quad &(t,x)\in(0,T)\times\bR^d,\\
        u^1(0,x)=0,\quad &x\in\bR^d
    \end{cases}
\end{equation*}
such that 
  \begin{align}
									\label{20230124 12}
\int_{0}^T\left\|u^2(t,\cdot)\right\|_{H_p^{\boldsymbol{\gamma}}(\bR^d,w\,\mathrm{d}x)}^q  w'(t)\mathrm{d}t
\leq N(1+T)^q\int_0^T\|f(t,\cdot)\|^q_{H_p^{\boldsymbol{r}-\boldsymbol{\gamma}}(\bR^d,w\,\mathrm{d}x)} w'(t)\mathrm{d}t.
    \end{align}
Due to the linearity,  $u:=u^1+u^2$ becomes a solution to \eqref{inhomo problem}.
Combining all \eqref{20230124 10}, \eqref{20230124 10-2}, \eqref{20230124 11}, and \eqref{20230124 12}, we finally have \eqref{20230124 08}.
    The corollary is proved.
\end{proof}

\begin{rem}
\label{24.02.21.14.50}
     $(i)$ The initial data space $B_{p,q}^{\boldsymbol{r}-\boldsymbol{w'}/q}(\bR^d,w\,\mathrm{d}x)$ presented in Corollary \ref{24.02.19.11.10} serves as the trace space.
     Specifically, the norm of $u_0$ in this space is bounded as follows:
    \begin{equation}
    \label{24.03.04.10.55}
        \|u_0\|_{B_{p,q}^{\boldsymbol{r}-\boldsymbol{w'}/q}(\bR^d,w\,\mathrm{d}x)}\lesssim \|u\|_{L_q((0,T),w'\,\mathrm{d}t;H_p^{\boldsymbol{r}}(\bR^d,w\,\mathrm{d}x))}+\|f\|_{L_q((0,T),w'\,\mathrm{d}t;H_p^{\boldsymbol{r}-\boldsymbol{\gamma}}(\bR^d,w\,\mathrm{d}x))}.
    \end{equation}
    This relationship is elucidated in \cite{CLSW2023trace}, which addresses the trace problem in the context of weighted Lebesgue spaces.
    If \eqref{20230124 08} holds with $f\equiv0$ and $B_{p,q}^{\boldsymbol{r}'}(\bR^d,w\,\mathrm{d}x)$;
    \begin{equation}
    \label{24.03.04.11.01}
        \left(\int_0^{T}\|u(t,\cdot)\|_{H_p^{\boldsymbol{r}}(\bR^d,w\,\mathrm{d}x)}^qw'(t)\mathrm{d}t\right)^{1/q}\leq N\|u_0\|_{B_{p,q}^{\boldsymbol{r}'}(\bR^d,w\,\mathrm{d}x)},
    \end{equation}
    then \eqref{24.03.04.10.55} implies that
    $$
    B_{p,q}^{\boldsymbol{r}'}(\bR^d,w\,\mathrm{d}x)\subseteq B_{p,q}^{\boldsymbol{r}-\boldsymbol{w'}/q}(\bR^d,w\,\mathrm{d}x).
    $$
    This yields that the maximal exponent $\boldsymbol{r'}$ in \eqref{24.03.04.11.01} is $\boldsymbol{r}-\boldsymbol{w'}/q$.
    Therefore, the estimate \eqref{20230124 08} is optimal.
    
    $(ii)$ Corollary \ref{24.02.19.11.10} is also obtained by combining the results in \cite{Choi_Kim2022} and \cite{CLSW2023trace}.
    It is important to note, however, that the scope of weights considered in \cite{CLSW2023trace} is confined to the Muckenhoupt $A_q(\bR)$-class, where $q$ represents the integrability constant with respect to the time variable. 
    Contrary to the constraints outlined in \cite{CLSW2023trace}, Lemma \ref{24.02.16.16.11} and Remark \ref{24.02.19.11.39} in the present study demonstrate that the class of weights need not be limited to the $A_q(\bR)$ class.
\end{rem}

Now we prove Theorem \ref{second-order case}.

\begin{proof}[Proof of Theorem \ref{second-order case}]
For $f \in \cS(\bR^d)$, it  follows that
\begin{align*}
\cF[a^{ij}(t)D_{x^ix^j}f](\xi)
=-a^{ij}(t) \xi^i \xi^j \cF[f] (\xi).
\end{align*}
Then due to \eqref{ellip coefficient}, the symbol $\psi(t,\xi):=-a^{ij}(t) \xi^i \xi^j$ satisfies an ellipticity condition with $(2,\kappa)$ and has a $n$-times regular upper bound with $(2,M)$ for any $n \in \bN$ by using the trivial extension $a^{ij}(t)=a^{ij}(T)$ for all $ t \geq T$. 
By $(ii)$, $(iii)$ of Lemma \ref{24.02.16.16.11} and Theorem \ref{22.12.27.16.53}, there exists a unique solution $u\in L_q((0,T),w\,\mathrm{d}t;H_p^2(\bR^d,w'\,\mathrm{d}x))$ such that
$$
\int_{0}^T \|u(t,\cdot)\|_{ H^{2}_p(\bR^d,w'\,\mathrm{d}x)}^{ q} w(t) \mathrm{d}t
    \leq N\|u_0\|_{B_{p,q}^{W_2}(\bR^d,w'\,\mathrm{d}x)}^q.
$$
Here $w\in A_{\nu}(\bR)$.
If $\nu=q$, then by \cite[Corollary 5.1]{CLSW2023trace}, we also have
$$
    \|u_0\|_{B_{p,q}^{W_2}(\bR^d,w'\,\mathrm{d}x)}^q\leq N\int_{0}^T \|u(t,\cdot)\|_{ H^{2}_p(\bR^d,w'\,\mathrm{d}x)}^{ q} w(t) \mathrm{d}t.
$$
The Theorem is proved.
\end{proof}
We end this section with a demonstration of how our findings contribute to the regularity of solutions for second-order parabolic partial differential equations (PDEs) with degenerate coefficients.
Consider $\{a^{ij}(t)\}_{d\times d}$, a nonnegative definite matrix of dimensions $d\times d$, fulfilling the condition
\begin{equation}
\label{22.09.20.16.21}
    \sum_{i,j=1}^da^{ij}(t)\xi^j\xi^j\geq \delta(t)|\xi|^2,\quad \forall \xi\in\bR^d,\,t\in\bR.
\end{equation}
We introduce two distinct assumptions for \eqref{22.09.20.16.21}:
\begin{itemize}
\item \textbf{Assumption 1:} Let $a^{ij}(t) = t^{\alpha} \delta^{ij}$ and $\delta(t) = t^{\alpha}$, with $\alpha > -1$ and $\delta^{ij}$ representing the Kronecker delta function.
\item \textbf{Assumption 2:} The matrix $\{a^{ij}(t)\}_{d\times d}$ satisfies condition \eqref{22.09.20.16.21} and additionally:
\begin{enumerate}
    \item For all $t > 0$, it holds that $\int_0^t \delta(s) \, ds > 0$.
    \item There exist constants $t_0 \in (0, T)$, $\beta > 0$, and $N_0 > 0$ such that for all $h > 0$,
    $$
    \left| \left\{ t \in [0, t_0] : h \leq \int_0^t \delta(s) \, ds < 4h \right\} \right| \leq N_0 h^{1/\beta}.
    $$
    \item There exists a constant $\bar{N}_0$ ensuring that
    $$
    |a^{ij}(t)| \leq \bar{N}_0 \delta(t).
    $$
\end{enumerate}
\end{itemize}
In the results of I. Kim, K.-H. Kim \cite{kim2018} and  K.-H. Kim, K. Lee  \cite{kim2017}, the authors established regularity results under these assumptions, which can be articulated as follows:
\begin{equation}\label{22.10.10.16.44}
\begin{cases}
\partial_t u(t,x) = \sum_{i,j=1}^{d} a^{ij}(t) D_{x^i x^j} u(t,x), & (t,x) \in (0,T) \times \mathbb{R}^d, \\
u(0,x) = u_0(x), & x \in \mathbb{R}^d.
\end{cases}
\end{equation}
For the initial condition $u_0 \in B_{p,p}^{\sigma}(\mathbb{R}^d)$, the implications under each assumption are delineated as follows:
\begin{align*}
(a) & \quad \text{Under Assumption 1 with } \alpha > -1 + \frac{1}{p},\,\sigma=2-\frac{2}{p(1+\alpha)}, \\
&\quad\quad\text{ it follows that } u_{xx} \in L_p((0,T); L_p(\mathbb{R}^d)). \\
(b) & \quad \text{Under Assumption 1 with } \alpha > -\frac{1}{p},\,\sigma=\frac{2}{1+\alpha}-\frac{2}{p(1+\alpha)}, \\
&\quad\quad\text{ it implies that } t^{\alpha} u_{xx} \in L_p((0,T); L_p(\mathbb{R}^d)). \\
(c) & \quad \text{Under Assumption 2 with }\sigma=2-\frac{2}{p\beta}, \text{ it is deduced that } u_{xx} \in L_p((0,T); L_p(\mathbb{R}^d)).
\end{align*}
The results denoted as $(a)$, $(b)$, and $(c)$ depend upon estimations of the heat kernel.
However, leveraging the corollaries of our findings, we introduce a unified approach that facilitates the derivation of more generalized results for $(a)$, $(b)$, and $(c)$.

\begin{corollary}
    Let $p\in(1,\infty)$, $q\in[1,\infty)$, $w\in A_p(\bR^d)$ and $u_0 \in B_{p,q}^{\sigma}(\mathbb{R}^d,w\,\mathrm{d}x)$.
    Then
        \begin{align*}
(a') & \quad \text{Under Assumption 1 with } \alpha > -1 + \frac{1}{q},\,\sigma=2-\frac{2}{q(1+\alpha)}, \\
&\quad\quad\text{ it follows that } u_{xx} \in L_q((0,T); L_p(\mathbb{R}^d, w\,\mathrm{d}x)). \\
(b') & \quad \text{Under Assumption 1 with } \alpha > -\frac{1}{q},\,\sigma=\frac{2}{1+\alpha}-\frac{2}{q(1+\alpha)}, \\
&\quad\quad\text{ it implies that } t^{\alpha} u_{xx} \in L_q((0,T); L_p(\mathbb{R}^d,w\,\mathrm{d}x)). \\
(c') & \quad \text{Under Assumption 2 with }\sigma=2-\frac{2}{q\beta}, \text{ it is deduced that } u_{xx} \in L_q((0,T); L_p(\mathbb{R}^d,w\,\mathrm{d}x)).
\end{align*}
\end{corollary}
\begin{proof}
It is possible to obtain $(a')$ from $(c')$ using \cite[Example 2.12]{kim2018}. Therefore, we only prove $(b')$ and $(c')$.

$(b')$ Since $\alpha>-1/q$,
$$
\beta:=\frac{\alpha(q-1)}{1+\alpha}>-1.
$$
Put $w_{\beta}(t):=t^{\beta}\in A_{\infty}(\bR)$, then
\begin{equation*}
    \begin{gathered}
     2^{-\boldsymbol{w_{\beta}}(j)}:=\cL_{w_{\beta}}(2^{2j})\simeq w_{\beta}([0,2^{-2j}])\simeq 2^{-2j(1+\beta)}=2^{-\frac{2q\alpha j}{1+\alpha}-\frac{2j}{1+\alpha}}\\
     \implies B_{p,q}^{\boldsymbol{2}-\frac{\boldsymbol{w_{\beta}}}{q}}(\bR^d,w\,\mathrm{d}x)=B_{p,q}^{\frac{2}{1+\alpha}-\frac{2}{q(1+\alpha)}}(\bR^d,w\,\mathrm{d}x).
    \end{gathered}
\end{equation*}
By Theorem \ref{second-order case}, the Cauchy problem
\begin{equation}
\label{22.09.21.10.00}
        \begin{cases}
            \p_tv(t,x)=\Delta v(t,x),\quad &(t,x)\in(0,T_{\alpha})\times\bR^d,\\
            u(0,x)=u_0(x),\quad & x\in\bR^d,
        \end{cases}\quad \left(T_{\alpha}:=\frac{T^{1+\alpha}}{1+\alpha}\right)
\end{equation}
has a unique solution $v\in L_q((0,T_{\alpha}),w_{\beta}\,\mathrm{d}t;H_p^2(\bR^d,w\,\mathrm{d}x))$ satisfying
\begin{equation}\label{22.09.21.13.33}
    \begin{aligned}
    \|v\|_{L_q((0,T_{\alpha}),w_{\beta}\,dt;H_p^2(\bR^d,w\,\mathrm{d}x))}&\leq N(1+T)\|u_0\|_{B_{p,q}^{\boldsymbol{2}-\frac{\boldsymbol{w_{\beta}}}{q}}(\bR^d,w\,\mathrm{d}x)}\\
    &=N(1+T)\|u_0\|_{B_{p,q}^{\frac{2}{1+\alpha}-\frac{2}{q(1+\alpha)}}(\bR^d,w\,\mathrm{d}x)}.
    \end{aligned}
\end{equation}
Let
\begin{equation*}
    u(t,x):=v\left(\frac{t^{1+\alpha}}{1+\alpha},x\right).
\end{equation*}
Then $u$ is a solution and due to \eqref{22.09.21.13.33},
\begin{align*}
\|t^{\alpha}u_{xx}\|_{L_q((0,T);L_p(\bR^d,w\,dx))}&=N\|v_{xx}\|_{L_q((0,T_{\alpha}),w_{\beta}\,dt;L_p(\bR^d,w\,dx))}\\
&\leq N(1+T)\|u_0\|_{B_{p,q}^{\frac{2}{1+\alpha}-\frac{2}{q(1+\alpha)}}(\bR^d,w\,dx)}.
\end{align*}

$(c')$ Let
$$
\tilde{a}^{ij}(t):=\frac{a^{ij}(\eta^{-1}(t))}{\delta(\eta^{-1}(t))}\quad \left(\frac{0}{0}:=1\right).
$$
By Lemma \ref{24.02.16.16.11}-($iv$) and Theorem \ref{22.12.27.16.53}, for $u_0\in B_{p,q}^{2-\frac{2}{q\beta}}(\bR^d,w\,\mathrm{d}x)$, the Cauchy problem
\begin{equation*}
    \begin{cases}
        \p_tv(t,x)=\tilde{a}^{ij}(t)v_{x^ix^j}(t,x),\quad &(t,x)\in(0,\eta^{-1}(T))\times\bR^d,\\
        v(0,x)=u_0(x),\quad & x\in\bR^d,
    \end{cases}
\end{equation*}
has a unique solution $v\in L_q((0,\eta^{-1}(T)),\,1/\delta(\eta^{-1})\mathrm{d}t;H_p^2(\bR^d,w\,\mathrm{d}x))$.
Consider
$$
u(t,x):=v(\eta(t),x).
$$
Then, $u$ is a solution to the Cauchy problem
\begin{equation*}
    \begin{cases}
        \p_tu(t,x)=\bar{a}^{ij}(t)u_{x^ix^j}(t,x),\quad &(t,x)\in(0,T)\times\bR^d,\\
        u(0,x)=u_0(x),\quad & x\in\bR^d,
    \end{cases}
\end{equation*}
where
$$
\bar{a}^{ij}(t):=\delta(t)\frac{a^{ij}(\eta^{-1}(\eta(t)))}{\delta(\eta^{-1}(\eta(t)))}.
$$
Let
$$
(0,T)=\{t:\eta^{-1}(\eta(t))=t\}\cup\{t:\eta^{-1}(\eta(t))<t\}=:A\cup B.
$$
For $t\in A$, $\bar{a}^{ij}(t)=a^{ij}(t)$. It can be easily checked that
$$
B=\bigcup_{n=1}^{\infty}B_n,
$$
where $B_n$'s are pairwise disjoint connected sets. Indeed, $\eta$ equals constant on each $B_n$. Thus, $\delta=0$ almost everywhere on $B_n$. This certainly implies that $\bar{a}^{ij}=a^{ij}=0$ almost everywhere on $B$. Therefore, $u$ is a solution to the Cauchy problem
\begin{equation*}
    \begin{cases}
        \p_tu(t,x)=a^{ij}(t)u_{x^ix^j}(t,x),\quad &(t,x)\in(0,T)\times\bR^d,\\
        u(0,x)=u_0(x),\quad & x\in\bR^d,
    \end{cases}
\end{equation*}
and
\begin{align*}
\|u_{xx}\|_{L_q((0,T);L_p(\bR^d,w\,\mathrm{d}x))}&=\|v_{xx}\|_{L_q((0,\eta^{-1}(T)),\,1/\delta(\eta^{-1})\mathrm{d}t;L_p(\bR^d,w\,\mathrm{d}x))}\\
&\leq N\|u_0\|_{B_{p,q}^{2-\frac{2}{q\beta}}(\bR^d,w\,\mathrm{d}x)}.
\end{align*}
The corollary is proved.
\end{proof}

\mysection{Proof of Theorem \ref{22.12.27.16.53}}
										\label{pf main thm}

We define kernels related to the symbol $\psi(t,\xi)$ first.
For $(t,s,x) \in \bR \times \bR \times \bR^d$ and $\varepsilon \in [0,1]$, we set
\begin{align*}
    p(t,s,x):=1_{0 \leq s < t} \cdot \frac{1}{(2\pi)^{d/2}}\int_{\bR^d} \exp\left(\int_{s}^t\psi(r,\xi)\mathrm{d}r\right)\mathrm{e}^{ix\cdot\xi}\mathrm{d}\xi,
\end{align*}
\begin{align*}
P_{\varepsilon}(t,s,x)
&:=\Delta^{\frac{\varepsilon\gamma}{2}}p(t,s,x)\\
&:=-(-\Delta)^{\frac{\varepsilon\gamma}{2}}p(t,s,x)\\
&:=1_{0 \leq s <t} \cdot \frac{1}{(2\pi)^{d/2}}\int_{\bR^d} |\xi|^{\varepsilon \gamma}\exp\left(\int_{s}^t\psi(r,\xi)\mathrm{d}r\right)\mathrm{e}^{ix\cdot\xi}\mathrm{d}\xi,
\end{align*}
and
\begin{align*}
    \psi(t,-i\nabla)p(t,s,x)&:=(\psi(t,-i\nabla)p)(t,s,x)\\
    &:=1_{0 \leq s < t} \cdot \frac{1}{(2\pi)^{d/2}}\int_{\bR^d} \psi(t,\xi)\exp\left(\int_{s}^t\psi(r,\xi)\mathrm{d}r\right)\mathrm{e}^{ix\cdot\xi}\mathrm{d}\xi.
\end{align*}
For these kernels, we introduce integral operators as follows:
\begin{align*}
\cT_{t,s} f(x) 
:= \int_{\bR^d} p(t,s,x-y)f(y)\mathrm{d}y,\quad 
\cT_{t,s}^{\varepsilon} f(x)
:=\Delta^{\frac{\varepsilon\gamma}{2}}\cT_{t,s} f(x) 
:=\int_{\bR^d} P_{\varepsilon}(t,s,x-y)f(y)\mathrm{d}y,
\end{align*}
and
\begin{align*}
 \psi(t,-i\nabla)\cT_{t,s} f(x) := \int_{\bR^d} \psi(t,-i\nabla) p(t,s,x-y)f(y)\mathrm{d}y.
\end{align*}
These operators are closely related to solutions of our initial value problems.
Formally, it is easy to check that 
\begin{align*}
\partial_t\cT_{t,0} f(x) =  \psi(t,-i\nabla)\cT_{t,0} f(x)
\end{align*}
and
\begin{align*}
\lim_{t \to 0} \cT_{t,0} f(x)  = f(x).
\end{align*}
Thus if a symbol $\psi$ and an initial data $u_0$ are nice enough, for instance $\psi$ satisfies the ellipticity condition and $u_0\in C_c^{\infty}(\bR^d)$, then the function 
$u(t,x):=\cT_{t,0}u_0(x)$ becomes a classical strong solution to the Cauchy problem
\begin{equation*}
    \begin{cases}
        \p_tu(t,x)=\psi(t,-i\nabla)u(t,x),\quad &(t,x)\in(0,T)\times\bR^d,\\
        u(0,x)=u_0(x),\quad &x\in\bR^d.
    \end{cases}
\end{equation*}
Therefore, roughly speaking, it is sufficient to show boundedness of $\cT_{t,0}^\varepsilon u_0$ in appropriate spaces for our \textit{a priori} estimates.
More precisely, due to the definitions of the Besov and Sobolev spaces, we have to estimate their Littlewood--Paley projections.
We recall that  $\Psi$ is a function in the Schwartz class $\mathcal{S}(\mathbb{R}^d)$ whose Fourier transform $\mathcal{F}[\Psi]$ is nonnegative, supported in an annulus $\{\frac{1}{2}\leq |\xi| \leq 2\}$, and $\sum_{j\in\mathbb{Z}} \mathcal{F}[\Psi](2^{-j}\xi) = 1$ for $\xi \not=0$.
    Then we define the Littlewood--Paley projection operators $\Delta_j$ and $S_0$ as $\mathcal{F}[\Delta_j f](\xi) = \mathcal{F}[\Psi](2^{-j}\xi) \mathcal{F}[f](\xi)$, $S_0f = \sum_{j\leq 0} \Delta_j f$, respectively.
We denote
\begin{equation}
\label{def T ep j}
\begin{aligned}
\cT_{t,s}^{\varepsilon,j} f(x)&:=\int_{\bR^d}\Delta_jP_{\varepsilon}(t,s,x-y)f(y)\mathrm{d}y\\
\cT_{t,s}^{\varepsilon,\leq0}f(x)&:=\int_{\bR^d}S_0P_{\varepsilon}(t,s,x-y)f(y)\mathrm{d}y,
\end{aligned}
\end{equation}
where
\begin{align*}
\Delta_jP_{\varepsilon}(t,s,x-y)
:=(\Delta_jP_{\varepsilon})(t,s,x-y)
:=\Delta_j\left[P_{\varepsilon}(t,s,\cdot)\right](x-y).
\end{align*}
Similarly,
\begin{align*}
\psi(t,-i\nabla)\cT_{t,s}^{j} f(x):=\int_{\bR^d}\Delta_j \psi(t,-i\nabla) p(t,s,x-y)f(y)\mathrm{d}y
\end{align*}
and
\begin{align*}
\psi(t,-i\nabla)\cT_{t,s}^{\leq 0}f(x):=\int_{\bR^d}S_0\psi(t,-i\nabla)p(t,s,x-y)f(y)\mathrm{d}y.
\end{align*}

Next, we recall Hardy--Littlewood's maximal function and Fefferman--Stein's sharp (maximal) function. 
For a locally integrable function $f$ on $\bR^d$, we define
\begin{align*}
\mathbb{M} f(x) 
&:=\sup_{x \in B_r(x_0)} \aint_{B_r(x_0)}|f(y)|\mathrm{d}y := \sup_{x \in B_r(x_0)} \frac{1}{|B_r(x_0)|}\int_{B_r(x_0)}|f(y)|\mathrm{d}y
\end{align*}
and
\begin{align*}
f^\sharp(x) 
:=\mathbb{M}^\sharp f(x) 
&:=\sup_{x \in B_r(x_0)} \aint_{B_r(x_0)}\aint_{B_r(x_0)}|f(y_0)-f(y_1)|\mathrm{d}y_0\mathrm{d}y_1 \\
&:= \sup_{x \in B_r(x_0)} \frac{1}{|B_r(x_0)|^2}\int_{B_r(x_0)}\int_{B_r(x_0)}|f(y_0)-f(y_1)|\mathrm{d}y_0\mathrm{d}y_1,
\end{align*}
where the supremum is taken over all balls $B_r(x_0)$ containing $x$ with $r \in (0,\infty)$ and $x_0 \in \bR^d$.
Moreover, for a function $f(t,x)$ defined on $(0,\infty) \times \bR^d$, we use the notation 
$\mathbb{M}_x^{\sharp}\big( f(t,x) \big)$ or $\mathbb{M}_x^{\sharp}\big( f(t,\cdot) \big)(x)$ to denote the sharp function with respect to the variable $x$ after fixing $t$. 
We recall a weighted version of the Hardy--Littlewood Theorem and Fefferman--Stein Theorem, which play an important role in our main estimate. 
\begin{prop}
				\label{lem:FS ineq}
Let $p \in (1,\infty)$ and $w \in A_p(\bR^d)$. Assume that $[w]_{A_p(\bR^d)}\leq K$ for a positive constant $K$. 
Then there exists a positive constant $N=N(d,K,p)$ such that for any $f\in L_p(\bR^d)$,
    \begin{align*}
        \big\| \mathbb{M} f \big\|_{L_p(\bR^d, w\,\mathrm{d}x)}
        \leq
        N\big\|f \big\|_{L_p(\bR^d, w\,\mathrm{d}x)} 
    \end{align*}
    and
    \begin{align*}
        \big\|f \big\|_{L_p(\bR^d, w\,\mathrm{d}x)}
        \leq
        N \big\| \mathbb{M}_x^{\sharp} f \big\|_{L_p(\bR^d, w\,\mathrm{d}x)}.
    \end{align*}
\end{prop}
This weighted version of the Hardy--Littlewood Theorem and the Fefferman--Stein Theorem is very well-known. For instance, see \cite[Theorems 2.2 and 2.3]{Dong_Kim2018}.

\begin{prop}\label{prop:maximal esti}
    Let $p \in (1,\infty)$, $\varepsilon \in [0,1]$, $w \in A_p(\bR^d)$, and $p_0\in(1,2]$ be a constant so that $p_0 \leq R_{p,d}^w$ and
    $$
    \left\lfloor\frac{d}{p_0}\right\rfloor = \left\lfloor\frac{d}{R_{p,d}^w}\right\rfloor.
    $$
    Suppose that $\psi$ is a symbol satisfying the ellipticity condition with $(\gamma,\kappa)$ and having the $\left( \left\lfloor \frac{d}{R_{p,d}^{w}}\right\rfloor+2\right)$-times regular upper bound with $(\gamma,M)$. 
        Then there exist positive constants $N=N(d,\delta,\varepsilon,\gamma,\kappa,M$, $N'=N'(d,p_0,\varepsilon,\gamma,\kappa,M)$, $N''=N''(d,p_0, \delta,\gamma,\kappa,M)$, and $N'''=N'''(d,p_0,\gamma,\kappa,M)$  such that for all $t\in (0,\infty)$, $f\in \cS(\bR^{d})$, and $j \in \bZ$,
    \begin{equation*}
    \begin{gathered}
     \bM^{\sharp}_x\left(\cT_{t,0}^{\varepsilon,j}f\right)(x)
     \leq N2^{j\varepsilon\gamma}\mathrm{e}^{-\kappa t 2^{j\gamma}\times\frac{p_0(1-\delta)}{2^{\gamma}}}\left(\bM\left(|f|^{p_0}\right)(x)\right)^{1/p_0},\quad \forall\delta\in(0,1),\\ 
        \bM^{\sharp}_x\left(\cT_{t,0}^{\varepsilon,\leq 0}f\right)(x)\leq N'\left(\bM\left(|f|^{p_0}\right)(x)\right)^{1/p_0},\\
        \bM^{\sharp}_x\left(\psi(t,-i\nabla)\cT_{t,0}^{j}f\right)(x)
        \leq N''2^{j\gamma}\mathrm{e}^{-\kappa t 2^{j\gamma}\times\frac{p_0(1-\delta)}{2^{\gamma}}}\left(\bM\left(|f|^{p_0}\right)(x)\right)^{1/p_0},\quad \forall\delta\in(0,1),\\ 
        \bM^{\sharp}_x\left(\psi(t,-i\nabla)\cT_{t,0}^{\leq0}f\right)(x)\leq N'''\left(\bM\left(|f|^{p_0}\right)(x)\right)^{1/p_0}.
    \end{gathered}
    \end{equation*}
 \end{prop}

The proof of Proposition \ref{prop:maximal esti} is given in Sections \ref{sec:prop}.

\begin{corollary}
								\label{main ingra}
    Let $p \in (1,\infty)$ and $w \in A_p(\bR^d)$.
    Suppose that $\psi$ is a symbol satisfying the ellipticity condition with $(\gamma,\kappa)$ and having the $\left( \left\lfloor \frac{d}{R_{p,d}^{w}}\right\rfloor+2\right)$-times regular upper bound with $(\gamma,M)$. 
        Then there exist positive constants $N$, $N'$, $N''$, and $N'''$ such that for all $t\in (0,\infty)$, $f\in   \cS(\bR^{d})$, and $j \in \bZ$,
\begin{align*}
\big\|\cT_{t,0}^{\varepsilon,j}f\big\|_{L_p(\bR^d, w\,\mathrm{d}x)} \leq N2^{j\varepsilon\gamma}\mathrm{e}^{-\kappa t 2^{j\gamma}\times\frac{(1-\delta)}{2^{\gamma}}} \big\|f \big\|_{L_p(\bR^d, w\,\mathrm{d}x)} \quad \forall \delta \in (0,1),
\end{align*}
\begin{align*}
\big\|\cT_{t,0}^{\varepsilon,\leq 0}f\big\|_{L_p(\bR^d, w\,\mathrm{d}x)} \leq N' \big\|f \big\|_{L_p(\bR^d, w\,\mathrm{d}x)},
\end{align*}
\begin{align*}
\big\|\psi(t,-i\nabla)\cT_{t,0}^{j}f \big\|_{L_p(\bR^d, w\,\mathrm{d}x)} \leq N''2^{j\gamma}\mathrm{e}^{-\kappa t 2^{j\gamma}\times\frac{(1-\delta)}{2^{\gamma}}} \big\|f \big\|_{L_p(\bR^d, w\,\mathrm{d}x)} \quad \forall \delta \in (0,1),
\end{align*}
\begin{align*}
\big\|\psi(t,-i\nabla)\cT_{t,0}^{\leq0}f \big\|_{L_p(\bR^d, w\,\mathrm{d}x)} \leq N''' \big\|f \big\|_{L_p(\bR^d, w\,\mathrm{d}x)},
\end{align*}
where the dependence of constants $N$, $N'$, $N''$, and $N'''$ is given similarly to those in Proposition \ref{prop:maximal esti} with additional dependence to $p$ and an upper bound of the $A_p$ semi-norm of $w$.
\end{corollary}
\begin{proof}
First, we claim that there exists a $p_0 \in (1,2]$ such that 
$p_0 < R_{p,d}^w \wedge p$, $w \in A_{p/p_0}(\bR^d)$, and $\left\lfloor\frac{d}{p_0}\right\rfloor = \left\lfloor\frac{d}{R_{p,d}^w}\right\rfloor$.
Recall that 
\begin{align*}
R_{p,d}^{w} := \sup \left\{ p_0 \in (1,2] :  w \in A_{p/p_0}(\bR^d) \right\}.
\end{align*}
It is well-known that $ R_{p,d}^{w} >1$ and 
$$
w \in A_{p/p_0}(\bR^d) \quad \forall p_0 \in \left(1,R_{p,d}^{w} \right)
$$
due to the reverse H\"older inequality (\textit{e.g.} see \cite[Remark 2.2]{Choi_Kim2022}).
Note that $\left\lfloor\frac{d}{p}\right\rfloor$ is left-continuous and piecewise-constant with respect to $p$.
Therefore, there exists a $p_0 \in \left(1,R_{p,d}^{w} \right)$ such that $\left\lfloor\frac{d}{p_0}\right\rfloor = \left\lfloor\frac{d}{R_{p,d}^w}\right\rfloor$
and $w \in A_{p/p_0}(\bR^d)$. 
It only remains to show that $p_0 \in \left(1,R_{p,d}^{w} \right)$ above is less than or equal to $p$. 
If $p \geq 2$, then it is obvious since $ p_0< 2$.
Thus, we only consider the case $p \in (1,2)$.
Recall that $A_p(\bR^d)$ is defined only for $p>1$ ($A_1(\bR^d)$-class is not introduced in this paper (see Definition \eqref{def ap}).
Thus any $p_0 \in (1,2)$ such that $w \in A_{p/p_0}(\bR^d)$ is obviously less than to $p$. 

Next, we apply the weighted version of the Hardy--Littlewood Theorem and Fefferman--Stein Theorem.
Let $f \in \cS(\bR^d)$ and choose a $p_0$ in the claim, i.e. $p_0 \in \left(1, R_{p,d}^w \wedge p\right)$, $w \in A_{p/p_0}(\bR^d)$, and
$\left\lfloor\frac{d}{p_0}\right\rfloor = \left\lfloor\frac{d}{R_{p,d}^w}\right\rfloor$.
Then by Proposition \ref{prop:maximal esti}, we have
    \begin{equation*}
    \begin{gathered}
     \bM^{\sharp}_x\left(\cT_{t,0}^{\varepsilon,j}f\right)(x)
     \leq N2^{j\varepsilon\gamma}\mathrm{e}^{-\kappa t 2^{j\gamma}\times\frac{p_0(1-\delta)}{2^{\gamma}}}\left(\bM\left(|f|^{p_0}\right)(x)\right)^{1/p_0},\quad \forall\delta\in(0,1),\\ 
        \bM^{\sharp}_x\left(\cT_{t,0}^{\varepsilon,\leq 0}f\right)(x)\leq N'\left(\bM\left(|f|^{p_0}\right)(x)\right)^{1/p_0},\\
        \bM^{\sharp}_x\left(\psi(t,-i\nabla)\cT_{t,0}^{j}f\right)(x)
        \leq N''2^{j\gamma}\mathrm{e}^{-\kappa t 2^{j\gamma}\times\frac{p_0(1-\delta)}{2^{\gamma}}}\left(\bM\left(|f|^{p_0}\right)(x)\right)^{1/p_0},\quad \forall\delta\in(0,1),\\ 
        \bM^{\sharp}_x\left(\psi(t,-i\nabla)\cT_{t,0}^{\leq0}f\right)(x)\leq N'''\left(\bM\left(|f|^{p_0}\right)(x)\right)^{1/p_0}.
    \end{gathered}
    \end{equation*}
Moreover, recalling $\frac{p}{p_0} \in (1,\infty)$ and applying Proposition \ref{lem:FS ineq}, we obtain
\begin{align*}
\big\|\cT_{t,0}^{\varepsilon,j}f\big\|_{L_p(\bR^d, w\,\mathrm{d}x)} \leq N2^{j\varepsilon\gamma}\mathrm{e}^{-\kappa t 2^{j\gamma}\times\frac{p_0(1-\delta)}{2^{\gamma}}} \big\|f \big\|_{L_p(\bR^d, w\,\mathrm{d}x)} \quad \forall \delta \in (0,1),
\end{align*}
\begin{align*}
\big\|\cT_{t,0}^{\varepsilon,\leq 0}f\big\|_{L_p(\bR^d, w\,\mathrm{d}x)} \leq N' \big\|f \big\|_{L_p(\bR^d, w\,\mathrm{d}x)},
\end{align*}
\begin{align*}
\big\|\psi(t,-i\nabla)\cT_{t,0}^{j}f \big\|_{L_p(\bR^d, w\,\mathrm{d}x)} \leq N''2^{j\gamma}\mathrm{e}^{-\kappa t 2^{j\gamma}\times\frac{p_0(1-\delta)}{2^{\gamma}}} \big\|f \big\|_{L_p(\bR^d, w\,\mathrm{d}x)} \quad \forall \delta \in (0,1),
\end{align*}
\begin{align*}
\big\|\psi(t,-i\nabla)\cT_{t,0}^{\leq0}f \big\|_{L_p(\bR^d, w\,\mathrm{d}x)} \leq N''' \big\|f \big\|_{L_p(\bR^d, w\,\mathrm{d}x)}.
\end{align*}
Since $p_0 \in (1,2)$, the corollary is proved. 
\end{proof}

Based on Corollary \ref{main ingra}, we prove the following main estimate, which yields Theorem \ref{22.12.27.16.53}.
\begin{thm}\label{thm:ep gamma esti}
    Let $p\in(1,\infty)$, $q\in(0,\infty)$, $\varepsilon \in [0,1]$, $a \in (0,\infty)$, $w\in A_p(\bR^d)$, $\mu$ be a nonnegative Borel measrue on $(0,\infty)$, and $\boldsymbol{\mu}:\bZ\to(-\infty,\infty)$ be a sequence having the controlled difference. 
    Suppose that $\psi$ is a symbol satisfying the ellipticity condition with $(\gamma,\kappa)$ and having the $\left( \left\lfloor \frac{d}{R_{p,d}^{w}}\right\rfloor+2\right)$-times regular upper bound with $(\gamma,M)$. 
Additionally, assume that  the Laplace transform of $\mu$ is controlled by a sequence $\boldsymbol{\mu}$ in a $\gamma$-dyadic way with parameter $a$, \textit{i.e.}
\begin{align}
								\label{laplace cond}
\cL_{\mu}(2^{\gamma j})
:= \int_0^\infty \exp\left( - 2^{\gamma j} t \right) \mu(\mathrm{d}t) \leq N_{\cL_\mu} \cdot 2^{j\gamma a}2^{-\boldsymbol{\mu}(j)},\quad \forall j\in\bZ.
\end{align}
Then there exists a positive constant $N$ such that for any $u_0\in C_c^{\infty}(\bR^d)$,
  \begin{align}
									\label{20230118 01}
        \int_{0}^T\left\|\Delta^{\frac{\varepsilon\gamma}{2}}\cT_{t,0} u_0 \right\|_{L_p(\bR^d,w\,\mathrm{d}x)}^q  
        t^a\mu\left(\frac{(1\wedge q) \kappa}{16^\gamma}\mathrm{d}t\right) 
       \leq N(1+ \mu_{a,T,\kappa,\gamma,q})\|u_0\|^q_{B_{p,q}^{\varepsilon\boldsymbol{\gamma}-\frac{\boldsymbol{\mu}}{q}}(\bR^d,w\,\mathrm{d}x)},
    \end{align}
        \begin{align}
									\label{20230118 02}
        \int_{0}^T\left\|\Delta^{\frac{\varepsilon\gamma}{2}} \cT_{t,0} u_0\right\|_{L_p(\bR^d,w\,\mathrm{d}x)}^q t^a \mu\left(\frac{(1\wedge q) \kappa}{16^\gamma}\mathrm{d}t\right)\leq N\|u_0\|^q_{\dot{B}_{p,q}^{\varepsilon \boldsymbol{\gamma}-\frac{\boldsymbol{\mu}}{q}}(\bR^d,w\,\mathrm{d}x)},
    \end{align}
  \begin{equation}
  \begin{aligned}
									\label{20230118 03}
        &\int_{0}^T\left\|\psi(t,-i\nabla)\cT_{t,0} u_0 \right\|_{L_p(\bR^d,w\,\mathrm{d}x)}^q  
        t^a\mu\left(\frac{(1\wedge q) \kappa}{16^\gamma}\mathrm{d}t\right) 
       \\
       \leq& N'(1+\mu_{a,T,\kappa,\gamma,q})\|u_0\|^q_{B_{p,q}^{\boldsymbol{\gamma}-\frac{\boldsymbol{\mu}}{q}}(\bR^d,w\,\mathrm{d}x)},
    \end{aligned}
    \end{equation}
 and   
        \begin{align}
									\label{20230118 04}
        \int_{0}^T\left\|\psi(t,-i\nabla)\cT_{t,0} u_0\right\|_{L_p(\bR^d,w\,\mathrm{d}x)}^q t^a \mu\left(\frac{(1\wedge q) \kappa}{16^\gamma}\mathrm{d}t\right)\leq N'\|u_0\|^q_{\dot{B}_{p,q}^{\boldsymbol{\gamma}-\frac{\boldsymbol{\mu}}{q}}(\bR^d,w\,\mathrm{d}x)},
    \end{align}
    where $[w]_{A_p(\bR^d)}\leq K$,
    $$
    N=N(a,  N_{\cL_\mu},d,\varepsilon, \gamma,\kappa,K,M,p,q),\quad N'=N'(a,  N_{\cL_\mu},d, \gamma,\kappa,K,M,p,q,R_{p,d}^w),$$  $\varepsilon \boldsymbol{\gamma}(j)=\varepsilon \gamma j$ for all $j \in \bZ$, and
\begin{align*}
\mu_{a,T,\kappa,\gamma,q}=\int_0^T  t^a\mu\left(\frac{(1\wedge q) \kappa}{16^\gamma}\mathrm{d}t\right).
\end{align*}
\end{thm}

\begin{proof}
Due to the upper bounds of $L_p$-norms in Corollary \ref{main ingra},
the proofs of \eqref{20230118 03} and \eqref{20230118 04} are very similar to those of \eqref{20230118 01} and \eqref{20230118 02} when $\varepsilon=1$. 
Thus we only prove \eqref{20230118 01} and \eqref{20230118 02}.
We make use of the Littlewood--Paley operators $\Delta_j$.
By using the almost orthogonal property of Littlewood--Paley operators, we have (at least in a distribution sense)
\begin{equation*}
    \begin{aligned}
        \Delta^{\frac{\varepsilon\gamma}{2}} \cT_{t,0} u_0
        &=\sum_{j\in\bZ}\Delta_j(-\Delta)^{\varepsilon\gamma/2}\cT_{t,0}u_0\\
        &=\sum_{j\in\bZ}\sum_{i \in \bZ}\Delta_j(-\Delta)^{\varepsilon\gamma/2}\cT_{t,0}\Delta_i u_0\\
        &=\sum_{j\in\bZ}\sum_{i=-1}^1\cT_{t,0}^{\varepsilon,j+i}(\Delta_{j}u_0)\\
        &
        =\cT_{t,0}^{\varepsilon,\leq0}(S_0u_0)+\cT_{t,0}^{\varepsilon,1}(\Delta_0u_0)+\sum_{j=1}^{\infty}\sum_{i=-1}^1\cT_{t,0}^{\varepsilon,j+i}(\Delta_{j}u_0),
    \end{aligned}
\end{equation*}
    where $\cT_{t,0}^{\varepsilon,\leq0}$ and  $\cT_{t,0}^{\varepsilon, j}$ are defined in \eqref{def T ep j}.
    Thus by Minkowski's inequality,
    \begin{equation}
        \label{ineq:22.02.22.12.37}
        \begin{aligned}
        &\int_0^T \|(-\Delta)^{\varepsilon\gamma/2}\cT_{t,0} u_0\|_{L_p(\bR^d,w\,\mathrm{d}x)}^q t^a \mu\left(\frac{(1\wedge q) \kappa}{16^\gamma}\mathrm{d}t\right)\nonumber\\
        \leq& \int_0^T \bigg(\sum_{j\in\bZ}\sum_{i=-1}^1\|\cT_{t,0}^{\varepsilon,j+i}(\Delta_{j}u_0)\|_{L_p(\bR^d,w\,\mathrm{d}x)}\bigg)^q t^a \mu\left(\frac{(1\wedge q) \kappa}{16^\gamma}\mathrm{d}t\right),
\end{aligned}
    \end{equation}
and
\begin{equation}
\label{ineq:22.02.22.12.38}
\begin{aligned}
        &\int_0^T \|(-\Delta)^{\varepsilon\gamma/2}\cT_{t,0} u_0\|_{L_p(\bR^d,w\,\mathrm{d}x)}^q t^a \mu\left(\frac{(1\wedge q) \kappa}{16^\gamma}\mathrm{d}t\right)\\
        \leq&  \int_0^T \bigg( \big\|\cT_{t,0}^{\varepsilon,\leq0}(S_0u_0)+\cT_{t,0}^{\varepsilon,1}(\Delta_0u_0)\big\|_{L_p(\bR^d, w\,\mathrm{d}x)}\\
        &\quad\quad\quad\quad+\sum_{j=1}^{\infty}\sum_{i=-1}^1 \big\|\cT_{t,0}^{\varepsilon,j+i}(\Delta_{j}u_0)\big\|_{L_p(\bR^d, w\,\mathrm{d}x)} \bigg)^q t^a \mu\left(\frac{(1\wedge q) \kappa}{16^\gamma}\mathrm{d}t\right).
    \end{aligned}
\end{equation}
Moreover, by Corollary \ref{main ingra}, we have
    \begin{align}
        \sum_{i=-1}^1 \|\cT_{t,0}^{\varepsilon,j+i}(\Delta_{j}u_0)\|_{L_p(\bR^d,w\,\mathrm{d}x)} &\leq N2^{j\varepsilon\gamma}\mathrm{e}^{-\kappa t 2^{j\gamma}\times\frac{(1-\delta)}{4^{\gamma}}}\|\Delta_ju_0\|_{L_p(\bR^d,w\,\mathrm{d}x)},\label{claim:hom}
\end{align}
and
\begin{align}        
        \|\cT_{t,0}^{\varepsilon,\leq 0}(S_0u_0) \|_{L_p(\bR^d, w\,\mathrm{d}x)} &\leq N\| S_0 u_0\|_{L_p(\bR^d, w\,\mathrm{d}x)}.
        \label{claim:inhom}
    \end{align}
    Due to \eqref{ineq:22.02.22.12.38} and \eqref{claim:inhom}, to show \eqref{20230118 01}, it is sufficient to show that
    \begin{equation}
        \label{20230120 01}
        \begin{aligned}
        &\int_0^T \bigg(\sum_{j\in\bN}\sum_{i=-1}^1\|\cT_{t,0}^{\varepsilon,j+i}(\Delta_{j}u_0)\|_{L_p(\bR^d,w\,\mathrm{d}x)}\bigg)^q t^a \mu\left(\frac{(1\wedge q) \kappa}{16^\gamma}\mathrm{d}t\right)\\
        \leq&
        N \|u_0\|_{B_{p,q}^{\varepsilon\boldsymbol{\gamma}-\boldsymbol{\mu}/q}(\bR^d,w\,\mathrm{d}x)}^q.
    \end{aligned}
    \end{equation}
Similarly, to show \eqref{20230118 02}, it is sufficient to show
\begin{equation}
\label{ineq:hom}
    \begin{aligned}
        &\int_0^T \bigg(\sum_{j\in\bZ}\sum_{i=-1}^1\|\cT_{t,0}^{\varepsilon,j+i}(\Delta_{j}u_0)\|_{L_p(\bR^d,w\,\mathrm{d}x)}\bigg)^q t^a \mu\left(\frac{(1\wedge q) \kappa}{16^\gamma}\mathrm{d}t\right)\\ 
        \leq&
        N \|u_0\|_{\dot{B}_{p,q}^{\varepsilon\boldsymbol{\gamma}-\boldsymbol{\mu}/q}(\bR^d,w\,\mathrm{d}x)}^q.
    \end{aligned}
    \end{equation}
    Since the proofs of \eqref{20230120 01} and \eqref{ineq:hom} are very similar, we only focus on proving the difficult case \eqref{ineq:hom}.
    To verify \eqref{ineq:hom}, we apply \eqref{claim:hom} to \eqref{ineq:22.02.22.12.37} with $\delta= (1-2^{-\gamma})$ and obtain
    \begin{align}
							\notag
&\int_0^T \bigg(\sum_{j\in\bZ}\sum_{i=-1}^1\|\cT_{t,0}^{\varepsilon,j+i}(\Delta_{j}u_0)(t,\cdot)\|_{L_p(\bR^d,w\,\mathrm{d}x)}\bigg)^q t^a \mu\left(\frac{(1\wedge q) \kappa}{16^\gamma}\mathrm{d}t\right)  \\
								\label{20230120 11}
\leq& N\int_0^T\left(\sum_{j\in\bZ}2^{j\varepsilon\gamma}\mathrm{e}^{-\kappa t2^{j\gamma} 2^{-3\gamma}}\|\Delta_ju_0\|_{L_p(\bR^d,w\,\mathrm{d}x)}\right)^qt^a\mu\left(\frac{(1\wedge q) \kappa}{16^\gamma}\mathrm{d}t\right).
    \end{align}
We estimate it depending on the range of $q$.
First, if $q\in(0,1]$, then simply we have
    \begin{align*}
        &\int_0^T\left(\sum_{j\in\bZ}2^{j\varepsilon\gamma}\mathrm{e}^{-\kappa t2^{j\gamma}2^{-3\gamma}}\|\Delta_ju_0\|_{L_p(\bR^d,w\,\mathrm{d}x)}\right)^qt^a\mu\left(\frac{(1\wedge q) \kappa}{16^\gamma}\mathrm{d}t\right) \\
        \lesssim& \sum_{j\in\bZ}2^{qj\varepsilon\gamma}\|\Delta_ju_0\|_{L_p(\bR^d,w\,\mathrm{d}x)}^q
        \int_0^\infty \exp\left(-2^{j\gamma}2^\gamma  t\right)  t^a \mu(\mathrm{d}t) \\
        \lesssim& \sum_{j\in\bZ}2^{qj\varepsilon\gamma - j\gamma a}\|\Delta_ju_0\|_{L_p(\bR^d,w\,\mathrm{d}x)}^q
        \cL_\mu(2^{j \gamma}), 
    \end{align*}
where the simple inequality that  $t^a\mathrm{e}^{-2^{j\gamma}2^{\gamma}t} \leq N(a,\gamma)2^{-j\gamma a}\mathrm{e}^{-2^{j\gamma}t}$ is used in the last part of the computation above.
Finally applying \eqref{laplace cond}, we have \eqref{ineq:hom}.

Next, we consider the case $q\in(1,\infty)$. Divide $\bZ$ into two parts as follows:
    $$
        \bZ=\{j\in\bZ:2^{j\gamma}t\leq1\}\cup\{j\in\bZ:2^{j\gamma}t>1\}=:\cI_1(t)\cup\cI_2(t).
    $$
    Thus, we have
    \begin{equation}
    \label{cI1 cI2}
    \begin{aligned}
        &\int_0^T \bigg(\sum_{j\in\bZ}\sum_{i=-1}^1\|\cT_{t,0}^{\varepsilon,j+i}(\Delta_{j}u_0)\|_{L_p(\bR^d,w\,\mathrm{d}x)}\bigg)^qt^a\mu\left(\frac{(1\wedge q) \kappa}{16^\gamma}\mathrm{d}t\right)\\
        \leq& N\int_0^T\left(\sum_{j\in\bZ}2^{j\varepsilon\gamma}\mathrm{e}^{-\kappa t2^{j\gamma}2^{-3\gamma}}\|\Delta_ju_0\|_{L_p(\bR^d,w\,\mathrm{d}x)}\right)^qt^a\mu\left(\frac{(1\wedge q) \kappa}{16^\gamma}\mathrm{d}t\right)\\
        \leq& N\int_0^T\left(\sum_{j\in\bZ}2^{j\varepsilon\gamma}\mathrm{e}^{- t2^{j\gamma}2^\gamma}\|\Delta_ju_0\|_{L_p(\bR^d,w\,\mathrm{d}x)}\right)^qt^a\mu\left(\mathrm{d}t\right)\\
    =& N\int_0^T\left(\sum_{j\in\cI_1(t)}\cdots\right)^q\mu(\mathrm{d}t) +N \int_0^T\left(\sum_{j\in\cI_2(t)}\cdots\right)^qt^a\mu\left(\mathrm{d}t\right)
    =: N(I_1 + I_2).
    \end{aligned}
    \end{equation}
    Let $b\in (0,a]$ whose exact value will be chosen later. 
    Then we put $2^{j\gamma b/q} 2^{-j \gamma b/q}$ in the summation with respect to $\cI_1(t)$ and make use of H\"older's inequality to obtain 
    \begin{equation}
    \label{cI1}
    \begin{aligned}
        I_1&=
        \int_0^T\left(\sum_{j\in\cI_1(t)}2^{j\varepsilon\gamma}\mathrm{e}^{- t2^{j\gamma}2^\gamma}\|\Delta_ju_0\|_{L_p(\bR^d,w\,\mathrm{d}x)}\right)^q t^a \mu(\mathrm{d}t)\\
        &\leq \int_0^T\left(\sum_{j\in\cI_1(t)}2^{\frac{j\gamma b}{q-1}}\right)^{q-1}\left(\sum_{j\in\cI_1(t)}2^{jq\varepsilon\gamma-j\gamma b}\mathrm{e}^{-q t2^{j\gamma} 2^\gamma}\|\Delta_ju_0\|_{L_p(\bR^d,w\,\mathrm{d}x)}^q\right)t^a\mu(\mathrm{d}t).
    \end{aligned}
    \end{equation}
    and similarly,
    \begin{equation}
    \label{cI2}
    \begin{aligned}
        I_2&=
        \int_0^T\left(\sum_{j\in\cI_2(t)}2^{j\varepsilon\gamma} \mathrm{e}^{- t2^{j\gamma}2^\gamma} \|\Delta_ju_0\|_{L_p(\bR^d,w\,\mathrm{d}x)}\right)^q t^a\mu(\mathrm{d}t)\\
        &\leq \int_0^T\left(\sum_{j\in\cI_2(t)}2^{\frac{j\gamma b}{q-1}}\mathrm{e}^{- t2^{j\gamma}2^\gamma}\right)^{q-1}\\
        &\quad\quad\quad\times\left(\sum_{j\in\cI_2(t)}2^{jq\varepsilon\gamma-j\gamma b}\mathrm{e}^{-t2^{j\gamma}2^\gamma}\|\Delta_ju_0\|_{L_p(\bR^d,w\,\mathrm{d}x)}^q\right) t^a \mu(\mathrm{d}t).
    \end{aligned}
    \end{equation}
We estimate \eqref{cI1} first.  For each $t>0$, we can choose $j_1(t)\in\bZ$ such that $2^{j_1(t)\gamma}t\leq 1$ and $2^{(j_1(t)+1)\gamma}t>1$.
Roughly speaking, $j_1(t) \approx -\frac{1}{\gamma}\log_2t$.
    Thus, we have 
    \begin{align}
        \sum_{j\in\cI_1(t)}2^{\frac{j\gamma b}{q-1}}=\frac{2^{\frac{j_1(t)\gamma b}{q-1}}}{1-2^{-\frac{\gamma b}{q-1}}}\leq N(b,\gamma ,q)t^{-\frac{b}{q-1}}.\label{ineq:22.02.22.14.06}
    \end{align}
Putting \eqref{ineq:22.02.22.14.06} in \eqref{cI1}, we obtain
    \begin{align}
        I_1 
        &\leq 
            N
            \int_0^T 
                t^{-b} 
                \sum_{j\in\cI_1(t)} 
                    2^{jq\varepsilon\gamma-j\gamma b} 
                    \mathrm{e}^{-q t2^{j\gamma}2^\gamma}
                    \|\Delta_ju_0\|_{L_p(\bR^d,w\,\mathrm{d}x)}^q  
                t^a \mu(\mathrm{d}t)\nonumber\\
        &\leq
        N
        \sum_{j\in\bZ} 
            2^{jq\varepsilon\gamma-j\gamma b} \|\Delta_ju_0\|_{L_p(\bR^d,w\,\mathrm{d}x)}^q
            \int_0^T \mathrm{e}^{-qt 2^{j\gamma} 2^\gamma} t^{a-b}\mu(\mathrm{d}t)\nonumber\\ 
        &\leq
        N
        \sum_{j\in\bZ} 
            2^{jq\varepsilon\gamma-j\gamma b} \|\Delta_ju_0\|_{L_p(\bR^d,w\,\mathrm{d}x)}^q
            2^{-j\gamma(a-b)}\int_0^\infty \mathrm{e}^{-t 2^{j\gamma}} \mu(\mathrm{d}t)\nonumber\\             
        &\leq
        N 
        \sum_{j\in\bZ} 2^{jq\varepsilon\gamma-j\gamma a}\cL_{\mu}(2^{j\gamma})  \|\Delta_ju_0\|_{L_p(\bR^d,w\,\mathrm{d}x)}^q,\nonumber
    \end{align}
where $N=N(b,\gamma,\kappa)$.
Therefore, by \eqref{laplace cond}, 
\begin{align}
\label{ineq:cI1 final}
    I_1\leq N(b,N_{\cL_\mu},\gamma,\kappa)\sum_{j\in\bZ}2^{q\varepsilon\boldsymbol{\gamma}(j)-\boldsymbol{\mu}(j)}\|\Delta_ju_0\|_{L_p(\bR^d,w\,\mathrm{d}x)}^q.
\end{align}
    For $I_2$, we choose a sufficiently small $b\in (0,a]$ satisfying $        1\geq\frac{b}{(q-1) 2^\gamma}$.
    Then we can check that
    $$
        f(\lambda):=\lambda^{\frac{b}{q-1}}\mathrm{e}^{-2^\gamma\lambda}
    $$
    is a decreasing function on $(1,\infty)$.
    Using $f(\lambda)$, we have
    \begin{equation}
    \label{20230120 30}
    \begin{aligned}
        \sum_{j\in\cI_2(t)}2^{\frac{j\gamma b}{q-1}}\mathrm{e}^{-t2^{j\gamma}2^\gamma}&=t^{-\frac{b}{q-1}}\sum_{j\in\cI_2(t)}(2^{j\gamma}t)^{\frac{b}{q-1}}\mathrm{e}^{-t2^{j\gamma}2^\gamma}\\
        &=t^{-\frac{b}{q-1}}\sum_{j\in\cI_2(t)}f(2^{j\gamma}t)\\
        &\leq Nt^{-\frac{b}{q-1}}\int_{-\frac{1}{\gamma}\log_{2}t}^{\infty}(2^{\lambda\gamma}t)^{\frac{b}{q-1}}\mathrm{e}^{- t2^{\lambda\gamma}2^\gamma}\mathrm{d}\lambda.
\end{aligned}
\end{equation}
Moreover, applying the simple change of the variable $t 2^{\lambda \gamma} \to \lambda$, the above term is less than or equal to
\begin{align}
        Nt^{-\frac{b}{q-1}}\int_1^{\infty} \frac{f(\lambda)}{\lambda}\mathrm{d}\lambda=N(q,b,\kappa,\gamma)t^{-\frac{b}{q-1}}.\label{ineq:22.02.22.14.35}
    \end{align}
    Putting \eqref{20230120 30} and \eqref{ineq:22.02.22.14.35} in \eqref{cI2}, we have
    \begin{align}
        I_2
        &\leq N
            \int_0^{\infty} 
                        \sum_{j\in\cI_2(t)} 
                2^{jq\varepsilon\gamma-j\gamma b}
                \mathrm{e}^{- t2^{j\gamma}2^\gamma}
                \|\Delta_ju_0\|_{L_p(\bR^d,w\,\mathrm{d}x)}^q 
            t^{a-b}\mu(\mathrm{d}t)\nonumber\\
        &\leq 
        N
            \sum_{j\in\bZ} 
                2^{jq\varepsilon\gamma-j\gamma a}
                \|\Delta_ju_0\|_{L_p(\bR^d,w\,\mathrm{d}x)}^q
                \cL_{\mu}(2^{j\gamma}),\nonumber
    \end{align}
where $N=N(b,\delta,\gamma,\kappa)$. 
Therefore, by \eqref{laplace cond} again,
\begin{align}
\label{ineq:cI2 final}
    I_2\leq N(b,N_{\cL_\mu},\gamma,\kappa)\sum_{j\in\bZ}2^{q\varepsilon\boldsymbol{\gamma}(j)-\boldsymbol{\mu}(j)}\|\Delta_ju_0\|_{L_p(\bR^d,w\,\mathrm{d}x)}^q.
\end{align}
Finally combining \eqref{cI1 cI2}, \eqref{ineq:cI1 final}, and \eqref{ineq:cI2 final}, we obtain
    \begin{align*}
        \int_0^T \bigg(\sum_{j\in\bZ}\sum_{i=-1}^1\|\cT_{t,0}^{\varepsilon,j+i}(\Delta_{j}u_0)\|_{L_p(\bR^d,w\,\mathrm{d}x)}\bigg)^q t^a\mu(\mathrm{d}t)
        &\leq N\sum_{j\in\bZ} 2^{q\varepsilon\boldsymbol{\gamma}(j)-\boldsymbol{\mu}(j)} \|\Delta_ju_0\|_{L_p(\bR^d,w\,\mathrm{d}x)}^q\\
        &=  N\|u_0\|_{\dot{B}_{p,q}^{\varepsilon\boldsymbol{\gamma}-\boldsymbol{\mu}/q}(\bR^d,w\,\mathrm{d}x)}^q,
    \end{align*}
    which proves \eqref{ineq:hom} since $b$ can be chosen depending only on $a$,$\gamma$, and $q$. 
\end{proof}

\begin{proof}[Proof of Theorem \ref{22.12.27.16.53}]
Due to Proposition \ref{22.05.03.11.34}, without loss of generality, we assume that $\boldsymbol{r}(j)=\boldsymbol{\gamma}(j)=j\gamma$.
First, we prove \textit{a priori} estimates \eqref{main a priori est 0} and \eqref{main a priori est}.
By \cite[Theorem 2.1.5]{choi_thesis}, for any $u_0 \in C_c^\infty(\bR^d)$, 
    there is a unique classical solution $u \in C_p^{1,\infty}([0,T] \times \bR^d)$ to the Cauchy problem
    \begin{equation*}
            \begin{cases}
            \p_tu(t,x)=\psi(t,-i\nabla)u(t,x),\quad &(t,x)\in(0,T)\times\bR^d,\\
            u(0,x)=u_0(x),\quad & x\in\bR^d.
            \end{cases}
    \end{equation*}
and the solution $u$ is given by
    $$
    u(t,x):=\int_{\bR^d}p(t,0,x-y)u_0(y)\mathrm{d}y=\cT_{t,0}u_0(x).
    $$
Thus, due to Theorem \ref{thm:ep gamma esti}, we have \eqref{main a priori est 0} and \eqref{main a priori est}.

Next, we prove the existence of a solution.
By Proposition \ref{22.04.24.20.57}-($ii$), for $u_0\in B_{p,q}^{\boldsymbol{\gamma}-\boldsymbol{\mu}/q}(\bR^d,w\,\mathrm{d}x)$, there exists a $\{u_0^n\}_{n=1}^{\infty}\subseteq C_c^{\infty}(\bR^d)$ such that $u_0^n\to u_0$ in $B_{p,q}^{\boldsymbol{\gamma}-\boldsymbol{\mu}/q}(\bR^d,w\,\mathrm{d}x)$.
 By \cite[Theorem 2.1.5]{choi_thesis} again,  
    $$
    u_n(t,x):=\int_{\bR^d}p(t,0,x-y)u_0^n(y)\mathrm{d}y=\cT_{t,0}u_0^n(x)
    \in C_p^{1,\infty}([0,T]\times\bR^d)
    $$
    becomes a unique classical solution to the Cauchy problem
    \begin{equation*}
            \begin{cases}
            \p_tu_n(t,x)=\psi(t,-i\nabla)u_n(t,x),\quad &(t,x)\in(0,T)\times\bR^d,\\
            u_n(0,x)=u_0^n(x),\quad & x\in\bR^d.
            \end{cases}
    \end{equation*}
Moreover, due to the linearity of the equation, applying Theorem \ref{thm:ep gamma esti} again, for all $n,m \in \bN$, we have
\begin{align*}
&\int_{0}^T\left(\left\|u_n-u_m \right\|_{H_p^{\boldsymbol{\gamma}}(\bR^d,w\,\mathrm{d}x))}^q +\left\|\psi(t,-i\nabla)(u_n - u_m) \right\|_{L_p(\bR^d,w\,\mathrm{d}x)}^q   \right)
        t^a\mu\left(\frac{(1\wedge q) \kappa}{16^\gamma}\mathrm{d}t\right)  \\
\leq& N'(1+\mu_{a,T,\kappa,\gamma,q})\|u^n_0-u^m\|^q_{B_{p,q}^{\boldsymbol{\gamma}-\frac{\boldsymbol{\mu}}{q}}(\bR^d,w\,\mathrm{d}x)}.
\end{align*}
In particular, $u_n$ becomes a Cauchy sequence in 
$$
L_q\left((0,T),t^a \mu\left(\frac{(1\wedge q) \kappa}{16^\gamma}\mathrm{d}t\right);H_p^{\boldsymbol{\gamma}}(\bR^d,w\,\mathrm{d}x)\right).
$$
Since the above space is a quasi-Banach space, there exists a $u$ which is given by the limit of $u_n$ in the space.
Therefore, by Definition \ref{def sol}, this $u$ becomes a solution. 
    
At last, we prove the uniqueness of a solution.
We assume that there exist two solutions $u$ and $v$.
Then by Definition \ref{def sol}, there exist
    $u_n,v_n\in C_p^{1,\infty}([0,T]\times\bR^d)$ such that $u_n(0,\cdot),v_n(0,\cdot)\in C_c^{\infty}(\bR^d)$,
\begin{equation*}
\begin{gathered}
\partial_tu_n(t,x)=\psi(t,-i\nabla)u_n(t,x),\quad \partial_tv_n(t,x)=\psi(t,-i\nabla)v_n(t,x)\quad \forall (t,x)\in(0,T)\times\bR^d,\\
u_n(0,\cdot), v_n(0,\cdot)\to u_0 \quad\text{in}\quad B_{p,q}^{\boldsymbol{\gamma}-\boldsymbol{\mu}/q}(\bR^d,w\,\mathrm{d}x),
\end{gathered}
\end{equation*}
and $u_n\to u$, $v_n\to v$ in 
$$
L_q\left((0,T),t^a \mu\left(\frac{(1\wedge q) \kappa}{16^\gamma}\mathrm{d}t\right);H_p^{\boldsymbol{\gamma}}(\bR^d,w\,\mathrm{d}x)\right).
$$ 
as $n\to\infty$. Then $w_n:=u_n-v_n$ satisfies
    \begin{equation*}
            \begin{cases}
            \p_tw_n(t,x)=\psi(t,-i\nabla)w_n(t,x),\quad &(t,x)\in(0,T)\times\bR^d,\\
            w_n(0,x)=u_n(0,x)-v_n(0,x),\quad & x\in\bR^d.
            \end{cases}
    \end{equation*}
    Due to \cite[Theorem 2.1.5]{choi_thesis} and Theorem \ref{thm:ep gamma esti}, we conclude that $w_n\to0$ in 
    $$
L_q\left((0,T),t^a \mu\left(\frac{(1\wedge q) \kappa}{16^\gamma} \mathrm{d}t\right);H_p^{\boldsymbol{\gamma}}(\bR^d,w\,\mathrm{d}x)\right).
$$  as $n\to\infty$. 
Since the limit is unique,  
$$
0=\lim_{n \to \infty} w_n = \lim_{n \to \infty}u_n - \lim_{n \to \infty}v_n = u -v.
$$
The theorem is proved.
\end{proof}

\mysection{Proof of Proposition \ref{prop:maximal esti}}\label{sec:prop}

Recall
\begin{align}
						\label{def fund}
p(t,s,x):=1_{0 < s< t} \cdot \frac{1}{(2\pi)^{d/2}}\int_{\bR^d} \exp\left(\int_{s}^t\psi(r,\xi)\mathrm{d}r\right)\mathrm{e}^{ix\cdot\xi}\mathrm{d}\xi
\end{align}
and
\begin{align}
						\label{def kernel}
    P_{\varepsilon}(t,s,x):=(-\Delta)^{\varepsilon\gamma/2}p(t,s,x),\quad \varepsilon\in[0,1].
\end{align}
We fix $\varepsilon \in [0,1]$ throughout this section.
The proof of Proposition \ref{prop:maximal esti} is twofold.
In the first subsection, we obtain quantitative estimates for the kernel $ P_{\varepsilon}$.
In the special case $\varepsilon =0$, we obtain an estimate for the fundamental solution $P_{0}(t,s,x):=p(t,s,x)$ to \eqref{eqn:model eqn}, which means
\begin{equation*}
    \begin{cases}
        \p_tp(t,s,x)=\psi(t,-i\nabla)p(t,s,x),\quad &(t,x)\in(s,\infty)\times\bR^d,\\
        \lim_{t \downarrow s}p(t,s,x)=\delta_0(x),\quad &x\in\bR^d.
    \end{cases}
\end{equation*}
Here, $\delta_0$ is the Dirac measure centered on the origin. 
We show some quantitative estimates in the first subsection and 
then, by using these estimates, prove an important lemma which controls mean oscillations of operators $\cT_{t,0}^{\varepsilon, j}$, $\cT_{t,0}^{\varepsilon, \leq 0}$,
$\psi(t,-i\nabla)\cT_{t,0}^{ j}$, and $\psi(t,-i\nabla) \cT_{t,0}^{\leq 0}$ in the second subsection.
Finally, we prove Proposition \ref{prop:maximal esti}  based on the mean oscillation estimates.

\subsection{Estimates on fundamental solutions}
Recall
\begin{align*}
\psi(t,-i \nabla)p(t,s,x):=1_{0 < s< t} \cdot \frac{1}{(2\pi)^{d/2}}\int_{\bR^d} \psi(t,\xi)\exp\left(\int_{s}^t\psi(r,\xi)\mathrm{d}r\right)\mathrm{e}^{ix\cdot\xi}\mathrm{d}\xi.
\end{align*}
Here is our main kernel estimate.
\begin{thm}\label{22.02.15.11.27}
    Let $k$ be an integer such that $k>\lfloor d/2\rfloor$.
    Assume that $\psi(t,\xi)$ satisfies the ellipticity condition with $(\gamma,\kappa)$ and has the $k$-times regular upper bound with $(\gamma,M)$.
    \begin{enumerate}[(i)]
        \item 
        Let $p\in[2,\infty]$, $(n,m,|\alpha|)\in[0,k]\times\{0,1\}\times \{0,1,2\}$, and $\delta\in(0,1)$. Then there exists a positive constant $N=N(|\alpha|,d,\delta,\varepsilon,\gamma,\kappa,M,m,n)$ such that for all $t>s>0$,
        \begin{equation}
            \label{22.01.27.13.46}
            \left\||\cdot|^{n}\p_t^mD_x^{\alpha}\Delta_jP_{\varepsilon}(t,s,\cdot)\right\|_{L_p(\bR^d)}\leq N\mathrm{e}^{-\kappa|t-s|2^{j\gamma}\times\frac{(1-\delta)}{2^{\gamma}}}2^{j\left((m+\varepsilon)\gamma+|\alpha|-n+d/p'\right)},
        \end{equation}
    where $p'$ is the H\"older conjugate of $p$, \textit{i.e.} $1/p+1/p'=1$ $(p'=1$ if $p=\infty)$.
    
        \item 
        Let $p\in[1,2]$ and $(m,|\alpha|)\in\times\{0,1\}\times \{0,1,2\}$, and $\delta\in(0,1)$.
        Then there exists a positive constant $N=N(|\alpha|,d,\delta,\varepsilon,\gamma,\kappa,M,m)$ such that for all $t>s>0$,
        \begin{equation}\label{22.01.27.13.58}
            \left\|\p_t^mD_x^{\alpha}\Delta_jP_{\varepsilon}(t,s,\cdot)\right\|_{L_p(\bR^d)}\leq N\mathrm{e}^{-\kappa|t-s|2^{j\gamma}\times\frac{(1-\delta)}{2^{\gamma}}}2^{j\left((m+\varepsilon)\gamma+|\alpha|+d/p'\right)}.
        \end{equation}

        \item 
        Let $p\in[2,\infty]$, $(n,m,|\alpha|)\in[0,k]\times\{0,1\}\times \{0,1,2\}$, and $\delta\in(0,1)$. Then there exists a positive constant $N=N(|\alpha|,d,\delta,\varepsilon,\gamma,\kappa,M,m,n)$ such that for all $t>s>0$,
        \begin{equation}
							\label{2023012301}								
\left\||\cdot|^{n}\p_t^mD_x^{\alpha}\Delta_j \psi(t,-i\nabla) p(t,s,\cdot)\right\|_{L_p(\bR^d)}\leq N\mathrm{e}^{-\kappa|t-s|2^{j\gamma}\times\frac{(1-\delta)}{2^{\gamma}}}2^{j\left((m+1)\gamma+|\alpha|-n+d/p'\right)}.
        \end{equation}

        \item 
        Let $p\in[1,2]$ and $(m,|\alpha|)\in\times\{0,1\}\times \{0,1,2\}$, and $\delta\in(0,1)$.
        Then there exists a positive constant $N=N(|\alpha|,d,\delta,\varepsilon,\gamma,\kappa,M,m)$ such that for all $t>s>0$,
        \begin{equation}
							\label{2023012302}
            \left\|\p_t^mD_x^{\alpha}\Delta_j\psi(t,-i\nabla)p(t,s,\cdot)\right\|_{L_p(\bR^d)}\leq N\mathrm{e}^{-\kappa|t-s|2^{j\gamma}\times\frac{(1-\delta)}{2^{\gamma}}}2^{j\left((m+1)\gamma+|\alpha|+d/p'\right)}.
        \end{equation}
    \end{enumerate}
\end{thm}
    
\begin{proof} 
The proofs of \eqref{2023012301} and \eqref{2023012302} are very similar to those of \eqref{22.01.27.13.46} and \eqref{22.01.27.13.58} with $\varepsilon =1$ due to \eqref{condi:reg ubound}, \textit{i.e.}
$$
|\psi(t,\xi)| \lesssim |\xi|^\gamma.
$$
Thus, we only focus on proving \eqref{22.01.27.13.46} and \eqref{22.01.27.13.58}.
    The proofs highly rely on the following lemma, whose proof is given in the last part of this subsection.
    \begin{lem}\label{lem:kernel esti}
        Let $k\in \bN$, $\alpha$ be a ($d$-dimensional) multi-index and $m\in\{0,1\}$.
        Assume that $\psi(t,\xi)$ satisfies the ellipticity condition with $(\gamma,\kappa)$ and  has the $k$-times regular upper bound with $(\gamma,M)$.
        Then there exists a  positive constant $N=N(|\alpha|,d,\delta,\varepsilon,\gamma,\kappa,M,n)$ such that for all $n\in\{0,1,\cdots,k\}$, $t>s>0$, $j\in\bZ$, and $\delta\in(0,1)$,
        \begin{equation*}
            \begin{gathered}
                \sup_{x\in\bR^d}|x|^{n} |\p_t^mD^{\alpha}_x\Delta_jP_{\varepsilon}(t,s,x)|\leq N\mathrm{e}^{-\kappa(t-s)2^{j\gamma}\times\frac{(1-\delta)}{2^{\gamma}}}2^{j((m+\varepsilon)\gamma+|\alpha|-n+d)}\\
                \left(\int_{\bR^d}|x|^{2n}|\p_t^mD^{\alpha}_x\Delta_jP_{\varepsilon}(t,s,x)|^2\mathrm{d}x\right)^{1/2}\leq N\mathrm{e}^{-\kappa(t-s)2^{j\gamma}\times\frac{(1-\delta)}{2^{\gamma}}}2^{j\left((m+\varepsilon)\gamma+|\alpha|-n+d/2\right)}.
            \end{gathered}
        \end{equation*}
    \end{lem}
    We temporarily assume that Lemma \ref{lem:kernel esti} holds to complete the proof of Theorem \ref{22.02.15.11.27}.
        We prove \eqref{22.01.27.13.46} first.  We divide the proof into two cases: the integer case and the non-integer case.
    \begin{enumerate}
        \item[Case 1.]
        Assume $n$ is an integer, \textit{i.e.} $n=0,1,2,\cdots,k$.
        If $p=2$ or $p=\infty$, then \eqref{22.01.27.13.46} directly holds due to Lemma \ref{lem:kernel esti}. 
        For $p\in(2,\infty)$, we apply Lemma \ref{lem:kernel esti} again and obtain
        \begin{align*}
            &\int_{\bR^d}|x|^{pn}|\p_t^mD_x^{\alpha}\Delta_jP_{\varepsilon}(t,s,x)|^{p}\mathrm{d}x\\
            =&\int_{\bR^d}|x|^{2n}|\p_t^mD_x^{\alpha}\Delta_jP_{\varepsilon}(t,s,x)|^{2}|x|^{(p-2)n}|\p_t^mD_x^{\alpha}\Delta_jP_{\varepsilon}(t,s,x)|^{p-2}\mathrm{d}x\\
            \leq& N\left(\mathrm{e}^{-\kappa(t-s)2^{j\gamma}\times\frac{(1-\delta)}{2^{\gamma}}}2^{j((m+\varepsilon)\gamma+|\alpha|-n+d)}\right)^{p-2}\int_{\bR^d}|x|^{2n}|\p_t^mD_x^{\alpha}K_{\varepsilon}(t,s,x)|^{2}\mathrm{d}x\\
            \leq& N\left(\mathrm{e}^{-\kappa(t-s)2^{j\gamma}\times\frac{(1-\delta)}{2^{\gamma}}}2^{j((m+\varepsilon)\gamma+|\alpha|-n+d)}\right)^{p-2}\\
            &\times\left(\mathrm{e}^{-\kappa(t-s)2^{j\gamma}\times\frac{(1-\delta)}{2^{\gamma}}}2^{j\left((m+\varepsilon)\gamma+|\alpha|-n+d/2\right)}\right)^{2}\\
            =&N\left(\mathrm{e}^{-\kappa(t-s)2^{j\gamma}\times\frac{(1-\delta)}{2^{\gamma}}}2^{j\left((m+\varepsilon)\gamma+|\alpha|-n+d/p'\right)}\right)^{p}.
        \end{align*}
        
        \item[Case 2.] 
        Assume $n$ is a non-integer, \textit{i.e.} $n\in[0,k]\setminus\{0,1,2,\cdots,k\}$.
        Observe that for any $q\in[1,\infty)$,
        \begin{equation}\label{22.01.27.13.59}
            \begin{aligned}
                &|x|^{qn}|\p_t^mD^{\alpha}_x\Delta_jP_{\varepsilon}(t,s,x)|^q\\
                =&\left(|x|^{q\lfloor n\rfloor}|\p_t^mD^{\alpha}_x\Delta_jP_{\varepsilon}(t,s,x)|^q\right)^{1-(n-\lfloor n\rfloor)}\\
                &\times\left(|x|^{q(\lfloor n\rfloor+1)}|\p_t^mD^{\alpha}_x\Delta_jP_{\varepsilon}(t,s,x)|^q\right)^{n-\lfloor n\rfloor}.
            \end{aligned}
        \end{equation}
        We use the result of Case 1 with \eqref{22.01.27.13.59} repeatedly, \textit{i.e.} we use \eqref{22.01.27.13.46} with $\lfloor n \rfloor$ and $\lfloor n \rfloor+1$ after applying \eqref{22.01.27.13.59} for the proof of this case. 
        Using \eqref{22.01.27.13.59} with $q=1$ and the result of Case 1, \eqref{22.01.27.13.46}  holds if $p=\infty$. 
        For other $p$, \textit{i.e.} $p \in [2,\infty)$, we use \eqref{22.01.27.13.59} with $q=p$ and apply H\"older's inequality with $\frac{1}{n-\lfloor n \rfloor}$. 
        Then, finally, due to the result of Case 1, we have
        \begin{align*}
            &\left(\int_{\bR^d}|x|^{pn}|\p_t^mD^{\alpha}_x\Delta_jP_{\varepsilon}(t,s,x)|^p\mathrm{d}x\right)^{1/p}\\
            \leq& \left(\int_{\bR^d}|x|^{p(\lfloor n\rfloor+1)}|\p_t^mD^{\alpha}_x\Delta_jP_{\varepsilon}(t,s,x)|^p\mathrm{d}x\right)^{(n-\lfloor n\rfloor)/p}\\
            &\times\left(\int_{\bR^d}|x|^{p\lfloor n\rfloor}|\p_t^mD^{\alpha}_x\Delta_jP_{\varepsilon}(t,s,x)|^p\mathrm{d}x\right)^{(1-(n-\lfloor n\rfloor))/p}\\
            \leq& N\mathrm{e}^{-\kappa(t-s)2^{j\gamma}\times\frac{(1-\delta)}{2^{\gamma}}}2^{j\left((m+\varepsilon)\gamma+|\alpha|-n+d/p'\right)}.
        \end{align*}
    \end{enumerate}

    Next, we prove \eqref{22.01.27.13.58}.   The case $p=2$  holds due to \eqref{22.01.27.13.46} with $n=0$. 
    Moreover, we claim that it is sufficient to show that \eqref{22.01.27.13.58} holds for $p=1$. 
    Indeed, for $p \in (1,2)$ there exists a $\lambda \in (0,1)$ such that $p=\lambda+2(1-\lambda)$ and
    \begin{equation}\label{ineq:22 02 28 13 27}
        \begin{aligned}
            |\p_t^mD^{\alpha}_x\Delta_jP_{\varepsilon}(t,s,x)|^p=|\p_t^mD^{\alpha}_x\Delta_jP_{\varepsilon}(t,s,x)|^{\lambda}\times|\p_t^mD^{\alpha}_x\Delta_jP_{\varepsilon}(t,s,x)|^{2(1-\lambda)}.
        \end{aligned}
    \end{equation}
Applying H\"older's inequality with \eqref{ineq:22 02 28 13 27} and $\frac{1}{\lambda}$, we obtain \eqref{22.01.27.13.58}.
    Thus, we focus on showing \eqref{22.01.27.13.58} with $p=1$. 
   Let $j \in \bZ$. We consider $|x| \leq 2^{-j}$ and $|x|>2^{-j}$, separately.
    \begin{align*}
        &\int_{\bR^d}|\p_t^mD^{\alpha}_x\Delta_jP_{\varepsilon}(t,s,x)|\mathrm{d}x\\
        &=\int_{|x|\leq 2^{-j}}|\p_t^mD^{\alpha}_x\Delta_jP_{\varepsilon}(t,s,x)|\mathrm{d}x + \int_{|x|> 2^{-j}}|\p_t^mD^{\alpha}_x\Delta_jP_{\varepsilon}(t,s,x)|\mathrm{d}x.
    \end{align*}
    For $|x| \leq 2^{-j}$ we make use of $(i)$ with $p=\infty$ and $n=0$. Then
    \begin{align*}
        \int_{|x|\leq 2^{-j}}|\p_t^mD^{\alpha}_x\Delta_jP_{\varepsilon}(t,s,x)|\mathrm{d}x&\leq N\mathrm{e}^{-\kappa(t-s)2^{j\gamma}\times\frac{(1-\delta)}{2^{\gamma}}}2^{j((m+\varepsilon)\gamma+|\alpha|)}.
    \end{align*}
    For $|x|>2^{-j}$ we put $d_2:=\lfloor d/2\rfloor+1$ and note that $d_2\leq k$.
    Then by H\"older's inequality and $(i)$ with $p=2$,
    \begin{align*}
       &\int_{|x|> 2^{-j}}|\p_t^mD^{\alpha}_x\Delta_jP_{\varepsilon}(t,s,x)|\mathrm{d}x\\
       &\leq \left(\int_{|x|> 2^{-j}}|x|^{-2d_2}\right)^{1/2} \left(\int_{|x|> 2^{-j}}|x|^{2d_2}|\p_t^mD^{\alpha}_x\Delta_jP_{\varepsilon}(t,s,x)|^2\mathrm{d}x\right)^{1/2}\\
       &\leq N2^{j(d_2-d/2)}\mathrm{e}^{-\kappa(t-s)2^{j\gamma}\times\frac{(1-\delta)}{2^{\gamma}}}2^{j\left((m+\varepsilon)\gamma+|\alpha|-d_2+d/2\right)}\\
       &= N\mathrm{e}^{-\kappa(t-s)2^{j\gamma}\times\frac{(1-\delta)}{2^{\gamma}}}2^{j\left((m+\varepsilon)\gamma+|\alpha|\right)}.
    \end{align*}
    The theorem is proved.
\end{proof}

\begin{corollary}\label{22.02.15.14.36}
    Let $k$ be an integer such that $k>\lfloor d/2\rfloor$. Assume that $\psi(t,\xi)$ satisfies the ellipticity condition with $(\gamma,\kappa)$ and has the $k$-times regular upper bound with $(\gamma,M)$. Then, for all $p\in[1,\infty]$ and $(m,|\alpha|)\in\{0,1\}\times\{0,1,2\}$,
    there exist positive constants $N$ and $N'$ such that for all $t>s>0$
    \begin{align}
									\label{20230130 01}
        \left\|\p_t^mD_x^{\alpha}S_0P_{\varepsilon}(t,s,\cdot)\right\|_{L_p(\bR^d)}\leq N,
    \end{align}
and
    $$
        \left\|\p_t^mD_x^{\alpha}S_0 \psi(t,-i\nabla)p(t,s,\cdot)\right\|_{L_p(\bR^d)}\leq N',
    $$
    where $N=N(|\alpha|,d,\varepsilon,\gamma,\kappa,M,m,p)$ and $N'=N'(|\alpha|,d,\gamma,\kappa,M,m,p)$.
\end{corollary}

\begin{proof}
Due to similarity, we only prove \eqref{20230130 01}.
    Since $S_0:=\sum_{j\leq 0}\Delta_j$, we make use of Minkowski's inequality and Theorem \ref{22.02.15.11.27} to obtain
    \begin{align*}
        \left\|\p_t^mD_x^{\alpha}S_0P_{\varepsilon}(t,s,\cdot)\right\|_{L_p(\bR^d)}&\leq \sum_{j\leq 0}\left\|\p_t^mD_x^{\alpha}\Delta_jP_{\varepsilon}(t,s,\cdot)\right\|_{L_p(\bR^d)}\leq N\sum_{j\leq 0}2^{j((m+\varepsilon)\gamma+|\alpha|+d/p')}.
    \end{align*}
 Note that the summation is finite if $(m+\varepsilon)\gamma+|\alpha|+d/p'>0$.
 The corollary is proved.
\end{proof}

Now we prove Lemma \ref{lem:kernel esti}.
\begin{proof}[Proof of Lemma \ref{lem:kernel esti}]
    We apply some elementary properties of the Fourier inverse transform to obtain an upper bound of $|x|^n|\partial_t^mD^{\alpha}_x\Delta_jP_{\varepsilon}(t,s,x)|$.
    Indeed, recalling \eqref{def fund} and \eqref{def kernel}, we have
    \begin{equation}\label{ineq:22 02 28 14 29}
        \begin{aligned}
            &|x|^n|\partial_t^mD^{\alpha}_x\Delta_jP_{\varepsilon}(t,s,x)|\\
            \leq& N(d,n)\sum_{i=1}^d|x^i|^n|\partial_t^mD^{\alpha}_x\Delta_jP_{\varepsilon}(t,s,x)|\\
            \leq& N(d,n)\sum_{|\beta|=n}\int_{\bR^d}\left|D^{\beta}_{\xi}\left(\xi^{\alpha}|\xi|^{\varepsilon\gamma}\cF[\Psi](2^{-j}\xi)\partial_t^m\exp\left(\int_s^t\psi(r,\xi)\mathrm{d}r\right)\right)\right|\mathrm{d}\xi.
        \end{aligned}
    \end{equation}
    For the integrand, we make use of Leibniz's product rule. Then
    \begin{equation}\label{ineq:22 02 28 14 30}
        \begin{aligned}
            &D_{\xi}^{\beta}\left(\xi^{\alpha}|\xi|^{\varepsilon\gamma}\cF[\Psi](2^{-j}\xi)\p_t^m\exp\left(\int_{s}^t\psi(r,\xi)\mathrm{d}r\right)\right)\\
            =&D_{\xi}^{\beta}\left(\psi(t,\xi)^m\xi^{\alpha}|\xi|^{\varepsilon\gamma}\cF[\Psi](2^{-j}\xi)\exp\left(\int_{s}^t\psi(r,\xi)\mathrm{d}r\right)\right)\\
            =&\sum_{\beta=\beta_0+\beta_1}c_{\beta_0,\beta_1}D^{\beta_0}_{\xi}\left(\psi(t,\xi)^m\xi^{\alpha}|\xi|^{\varepsilon\gamma}\exp\left(\int_{s}^t\psi(r,\xi)\mathrm{d}r\right)\right)2^{-j|\beta_1|}D^{\beta_1}_{\xi}\cF[\Psi](2^{-j}\xi).
        \end{aligned}
    \end{equation}
    To estimate $D^{\beta_0}_{\xi}\left(\psi(t,\xi)^m\xi^{\alpha}|\xi|^{\varepsilon\gamma}\exp\left(\int_{s}^t\psi(r,\xi)\mathrm{d}r\right)\right)$, we borrow the lemma in \cite{Choi_Kim2022} and introduce it below.
    \begin{lem}[\cite{Choi_Kim2022} Lemma 4.1]\label{lem:multiplier esti}
        Let $n\in\bN$ and assume that $\psi(t, \xi)$ satisfies the ellipticity condition with $(\gamma, \kappa)$ and has the $n$-times regular upper bound with $(\gamma, M)$.
        Then there exists a positive constant $N = N(M,n)$ such that for all $t>s>0$ and $\xi\in\bR^d \setminus \{0\}$,
            \begin{align*}
                \bigg| D_\xi^n\biggl(\exp\bigg(\int_s^t \psi(r, \xi)\mathrm{d}r\bigg)\biggr) \bigg|
                \leq N |\xi|^{-n} \exp(-\kappa(t-s)|\xi|^\gamma) \sum_{k=1}^n |t-s|^k |\xi|^{k\gamma}.
            \end{align*}
    \end{lem}
We keep going to prove Lemma \ref{lem:kernel esti}.
By Lemma \ref{lem:multiplier esti}, it follows that
    \begin{equation}\label{ineq:22 02 28 15 13}
        \begin{aligned}
            &\bigg| D^{\beta_0}_{\xi}\left(\psi(t,\xi)^m\xi^{\alpha}|\xi|^{\varepsilon\gamma}\exp\left(\int_{s}^t\psi(r,\xi)\mathrm{d}r\right)\right)\bigg| \\
            \leq&  \sum_{\gamma_1 + \gamma_2 + \gamma_3 = \beta_0} c_{\gamma_1,\gamma_2,\gamma_3}
                \big|D_\xi^{\gamma_1} \big(\psi(t,\xi)^m\big) \big|  \big|D_\xi^{\gamma_2} \big(\xi^\alpha |\xi|^{\varepsilon\gamma}\big) \big|
                \bigg|D_\xi^{\gamma_3} \bigg(\exp\biggl(\int_s^t \psi(r,t)\mathrm{d}r\biggr)\bigg) \bigg|\\
            \leq& N(|\beta_0|,M) |\xi|^{m\gamma + \varepsilon\gamma +|\alpha| - |\beta_0|} \exp(-\kappa(t-s)|\xi|^{\gamma}) \sum_{k=1}^n(t-s)^k |\xi|^{k\gamma}.
        \end{aligned}
    \end{equation}
    By \eqref{ineq:22 02 28 14 30} and \eqref{ineq:22 02 28 15 13}, 
    \begin{equation}\label{21.08.31.13.49}
        \begin{aligned}
            &\left|D_{\xi}^{\beta}\left(\xi^{\alpha}|\xi|^{\varepsilon\gamma}\cF[\Psi](2^{-j}\xi)\p_t^m\exp\left(\int_{s}^t\psi(r,\xi)\mathrm{d}r\right)\right)\right|\\
            \leq& N\mathrm{e}^{-\kappa(t-s)|\xi|^{\gamma}}\sum_{k=1}^{n}(t-s)^k|\xi|^{k\gamma}\sum_{\beta=\beta_0+\beta_1}|\xi|^{(m+\varepsilon)\gamma+|\alpha|-|\beta_0|}2^{-j|\beta_1|}|D^{\beta_1}_{\xi}\cF[\Psi](2^{-j}\xi)|\\
            \leq& N|\xi|^{(m+\varepsilon)\gamma+|\alpha|-|\beta|}\mathrm{e}^{-\kappa(t-s)|\xi|^{\gamma}}\left(\sum_{k=1}^{n}(t-s)^k|\xi|^{k\gamma}\right)1_{2^{j-1}\leq|\xi|\leq2^{j+1}}\\
            \leq& N|\xi|^{(m+\varepsilon)\gamma+|\alpha|-|\beta|}\mathrm{e}^{-\kappa(1-\delta)(t-s)|\xi|^{\gamma}}1_{2^{j-1}\leq|\xi|\leq2^{j+1}}\\
            =&N|\xi|^{(m+\varepsilon)\gamma+|\alpha|-n}\mathrm{e}^{-\kappa(1-\delta)(t-s)|\xi|^{\gamma}}1_{2^{j-1}\leq|\xi|\leq2^{j+1}},
        \end{aligned}
    \end{equation}
    where $N=N(\delta,M,n)$.
Moreover,  one can check that
    \begin{equation}
    \label{ineq:22 02 28 14 31}
    \begin{aligned}								
            &\int_{2^{j-1}\leq |\xi|\leq 2^{j+1}}|\xi|^{(m+\varepsilon)\gamma+|\alpha|-n}\mathrm{e}^{-\kappa(1-\delta)(t-s)|\xi|^{\gamma}}\mathrm{d}\xi\\	
            =&N(d)\int_{2^{j-1}}^{2^{j+1}}l^{(m+\varepsilon)\gamma+|\alpha|-n+d-1}\mathrm{e}^{-\kappa(1-\delta)(t-s)l^{\gamma}}\mathrm{d}l\\		
            \leq& N\mathrm{e}^{-\kappa(t-s)2^{j\gamma}\times\frac{(1-\delta)}{2^{\gamma}}}\int_{2^{j-1}}^{2^{j+1}}l^{(m+\varepsilon)\gamma+|\alpha|-n+d-1}\mathrm{d}l\\
			=&N(|\alpha|,d,\varepsilon,m,n)\mathrm{e}^{-\kappa(t-s)2^{j\gamma}\times\frac{(1-\delta)}{2^{\gamma}}}2^{j((m+\varepsilon)\gamma+|\alpha|-n+d)}.
    \end{aligned}
    \end{equation}
All together with \eqref{ineq:22 02 28 14 29}, \eqref{21.08.31.13.49}, and \eqref{ineq:22 02 28 14 31}, we have
    \begin{align*}
        |x|^n|\partial_t^mD^{\alpha}_x\Delta_jP_{\varepsilon}(t,s,x)|
        \leq 
        N\mathrm{e}^{-\kappa(t-s)2^{j\gamma}\times\frac{(1-\delta)}{2^{\gamma}}}2^{j((m+\varepsilon)\gamma+|\alpha|-n+d}.
    \end{align*}
    Similarly, by \eqref{21.08.31.13.49} and Plancherel's theorem,
    \begin{align*}
        &\int_{\bR^d}|x|^{2n}|\p_t^mD_x^{\alpha} \Delta_jP_{\varepsilon}(t,s,x)|^2\mathrm{d}x\leq N\int_{\bR^d}|x^i|^{2n}|\psi(t,-i\nabla)^mD_x^{\alpha}\Delta_j P_{\varepsilon}(t,s,x)|^2\mathrm{d}x\\
        \leq& N\sum_{|\beta|=n}\int_{\bR^d}\left|D^{\beta}_{\xi}\left(\psi(t,\xi)^m\xi^{\alpha}|\xi|^{\varepsilon\gamma}\cF[\Psi](2^{-j}\xi)\exp\left(\int_{s}^t\psi(r,\xi)\mathrm{d}r\right)\right)\right|^2\mathrm{d}\xi\\
        \leq& \int_{2^{j-1}\leq|\xi|\leq2^{j+1}}|\xi|^{2(m+\varepsilon)\gamma+2|\alpha|-2n}\mathrm{e}^{-2\kappa(1-\delta)(t-s)|\xi|^{\gamma}}\mathrm{d}\xi\\
        \leq& N\mathrm{e}^{-2\kappa(t-s)2^{j\gamma}\times\frac{(1-\delta)}{2^{\gamma}}}2^{j(2(m+\varepsilon)\gamma+2|\alpha|-2n+d)},
    \end{align*}
    where $N=N(|\alpha|,d,\delta,\varepsilon,m,n)$. 
    The lemma is proved.
\end{proof}

\subsection{Proof of Proposition \ref{prop:maximal esti}}

Recall
$$
\cT_{t,0}^{\varepsilon,j} f(x):=\int_{\bR^d}\Delta_jP_{\varepsilon}(t,0,x-y)f(y)\mathrm{d}y,\quad \cT_{t,0}^{\varepsilon,\leq0}f(x):=\int_{\bR^d}S_0P_{\varepsilon}(t,0,x-y)f(y)\mathrm{d}y,
$$
and
\begin{equation*}
    \begin{gathered}
        \psi(t,-i \nabla)\cT_{t,0}^{j} f(x):=\int_{\bR^d}\Delta_jP_{\varepsilon}(t,0,x-y)f(y)\mathrm{d}y,\\
        \psi(t,-i \nabla)\cT_{t,0}^{\leq0}f(x):=\int_{\bR^d}S_0P_{\varepsilon}(t,0,x-y)f(y)\mathrm{d}y.
    \end{gathered}
\end{equation*}

In this subsection, we start estimating mean oscillations of $\cT_{t,0}^{\varepsilon, j}f$ and $\cT_{t,0}^{\varepsilon,\leq0}f$.
\begin{lem}\label{20.12.21.16.26}
    Let $t>0$, $j\in\bZ$, $\delta \in (0,1)$, $p_0\in(1,2]$, $b>0$, and $f\in \cS(\bR^d)$.
    Suppose that $\psi$ is a symbol satisfying the ellipticity condition with $(\gamma,\kappa)$ and having the $\left( \left\lfloor \frac{d}{p_0}\right\rfloor+2\right)$-times regular upper bound with $(\gamma,M)$.
    \begin{enumerate}[(i)]
        \item 
        Then for any $x\in B_{2^{-j}b}(0)$,
        \begin{equation*}
            \begin{aligned}
     &\aint_{B_{2^{-j}b}(0)}\aint_{B_{2^{-j}b}(0)}|\cT_{t,0}^{\varepsilon,j}f(y_0)-\cT_{t,0}^{\varepsilon,j} f(y_1)|^{p_0}\mathrm{d}y_0\mathrm{d}y_1\leq N2^{j\varepsilon\gamma p_0}\mathrm{e}^{-\kappa t 2^{j\gamma}\times\frac{p_0(1-\delta)}{2^{\gamma}}} \bM\left(|f|^{p_0}\right)(x),
     \end{aligned}
        \end{equation*}
        and
        \begin{equation*}
            \begin{aligned}
     &\aint_{B_{2^{-j}b}(0)}\aint_{B_{2^{-j}b}(0)}| \psi(t,-i\nabla)\cT_{t,0}^{j}f(y_0)-\psi(t,-i\nabla)\cT_{t,0}^{j} f(y_1)|^{p_0}\mathrm{d}y_0\mathrm{d}y_1
     \\
     &\leq N'2^{j\gamma p_0}\mathrm{e}^{-\kappa t 2^{j\gamma}\times\frac{p_0(1-\delta)}{2^{\gamma}}} \bM\left(|f|^{p_0}\right)(x),
            \end{aligned}
        \end{equation*}
        where $N=N(d, \delta,\varepsilon,\gamma,\kappa,M,m,p_0)$ and $N'=N'(d,\delta, \gamma,\kappa,M,m,p_0)$.
        \item 
        Then for any $x\in B_{b}(0)$,
        \begin{equation*}
            \begin{gathered}
         \aint_{B_{b}(0)}\aint_{B_{b}(0)}|\cT_{t,0}^{\varepsilon,\leq0} f(y_0)-\cT_{t,0}^{\varepsilon,\leq0} f(y_1)|^{p_0}\mathrm{d}y_0\mathrm{d}y_1\leq N\bM\left(|f|^{p_0}\right)(x),\\
            \aint_{B_{b}(0)}\aint_{B_{b}(0)}|\psi(t,-i\nabla)\cT_{t,0}^{\leq0} f(y_0)-\psi(t,-i\nabla)\cT_{t,0}^{\leq0} f(y_1)|^{p_0}\mathrm{d}y_0\mathrm{d}y_1\leq N'\bM\left(|f|^{p_0}\right)(x), 
            \end{gathered}
        \end{equation*}
        where $N=N(d,\varepsilon,\gamma,\kappa,M,m,p_0)$ and $N'=N'(d,\gamma,\kappa,M,m,p_0)$
    \end{enumerate}
\end{lem}
Once we assume Lemma \ref{20.12.21.16.26}, then Proposition \ref{prop:maximal esti}  follows. 
\begin{proof}[Proof of Proposition \ref{prop:maximal esti}]
    Due to similarity, we only prove it for $\cT_{t,0}^{\varepsilon,j}$. Let $b>0$, $t>0$, and $x\in B_{2^{-j}b}(0)$. By Lemma \ref{20.12.21.16.26},
    \begin{align*}
        &\aint_{B_{2^{-j}b}(0)}\aint_{B_{2^{-j}b}(0)}|\cT_{t,0}^{\varepsilon,j} f(y_0)-\cT_{t,0}^{\varepsilon,j} f(
        y_1)|^{p_0}\mathrm{d}y_0\mathrm{d}y_1 \leq N2^{j\varepsilon\gamma p_0}\mathrm{e}^{-\kappa t 2^{j\gamma}\times\frac{p_0(1-\delta)}{2^{\gamma}}}\bM\left(|f|^{p_0}\right)(x).
    \end{align*}
    For $x_0\in\bR^{d}$, denote 
    $$
        \tau_{x_0}f(t,x):=f(t,x_0+x).
    $$
    Since $\cT_{t,0}^{\varepsilon,j}$ and $\tau_{x_0}$ are commutative,
    \begin{align*}
        &\aint_{B_{2^{-j}b}(x_0)}\aint_{B_{2^{-j}b}(x_0)}|\cT_{t,0}^{\varepsilon,j} f(y_0)-\cT_{t,0}^{\varepsilon,j} f(y_1)|^{p_0}\mathrm{d}y_0\mathrm{d}y_1\\
        =&\aint_{B_{2^{-j}b}(0)}\aint_{B_{2^{-j}b}(0)}|\cT_{t,0}^{\varepsilon,j} (\tau_{x_0}f)(t,y_0)-\cT_{t,0}^{\varepsilon,j} (\tau_{x_0}f)(t,y_1)|^{p_0}\mathrm{d}y_0\mathrm{d}y_1\\
        \leq& N2^{j\varepsilon\gamma p_0}\mathrm{e}^{-\kappa t 2^{j\gamma}\times\frac{p_0(1-\delta)}{2^{\gamma}}}\bM\left(|\tau_{x_0}f|^{p_0}\right)(x)=N2^{j\varepsilon\gamma p_0}\mathrm{e}^{-\kappa t 2^{j\gamma}\times\frac{p_0(1-\delta)}{2^{\gamma}}}\bM\left(|f|^{p_0}\right)(x_0+x).
    \end{align*}
    Therefore, by Jensen's inequality, for $x\in B_{2^{-j}b}(0)$ and $x_0\in\bR^{d}$
    \begin{align*}
        &\left(\aint_{B_{2^{-j}b}(x_0)}\aint_{B_{2^{-j}b}(x_0)}|\cT_{t,0}^{\varepsilon,j}f(y_0)-\cT_{t,0}^{\varepsilon,j} f(y_1)|\mathrm{d}y_0\mathrm{d}y_1\right)^{p_0}\\
        \leq& \aint_{B_{2^{-j}b}(x_0)}\aint_{B_{2^{-j}b}(x_0)}|\cT_{t,0}^{\varepsilon,j}f(y_0)-\cT_{t,0}^{\varepsilon,j} f(y_1)|^{p_0}\mathrm{d}y_0\mathrm{d}y_1\\
        \leq& N2^{j\varepsilon\gamma
        p_0}\mathrm{e}^{-\kappa t 2^{j\gamma}\times\frac{p_0(1-\delta)}{2^{\gamma}}}\bM\left(|f|^{p_0}\right)(x_0+x).
    \end{align*}
    Taking the supremum both sides with respect to all $B_{2^{-j}b}$ containing $x_0+x$, we obtain the desired result. The theorem is proved.
\end{proof}
Therefore, it suffices to prove Lemma \ref{20.12.21.16.26} to complete the proof of Proposition \ref{prop:maximal esti}.
In doing so, we begin with two lemmas that reduce our computational effort.
For readers' convenience, we also present the following scheme, which explains relations among Lemmas \ref{20.12.17.20.21}, \ref{22.02.15.16.27}, \ref{22.01.28.16.57}, \ref{22.01.28.17.18}, \ref{20.12.21.16.26}, and Proposition \ref{prop:maximal esti}.
\begin{equation*}
    \begin{rcases}
        \text{Lemma \ref{22.01.28.16.57}} \to &\text{Lemma \ref{22.01.28.17.18}} \\
        &\text{Lemma \ref{20.12.17.20.21}} \\
        &\text{Lemma \ref{22.02.15.16.27}} 
    \end{rcases}
    \to \text{Lemma \ref{20.12.21.16.26}} \to \text{Proposition \ref{prop:maximal esti}},
\end{equation*}
where $A\to B$ implies that $A$ is used in the proof of $B$. 
Note that Lemmas \ref{20.12.17.20.21}, \ref{22.02.15.16.27}, and \ref{22.01.28.16.57} are simple consequences of Theorem \ref{22.02.15.11.27} and Corollary \ref{22.02.15.14.36}. 

\begin{lem}\label{20.12.17.20.21}
    Let $p_0\in(1,2]$ and $f\in \cS(\bR^d)$.
    Suppose that $\psi$ is a symbol satisfying the ellipticity condition with $(\gamma,\kappa)$ and having the $\left( \left\lfloor \frac{d}{p_0}\right\rfloor+2\right)$-times regular upper bound with $(\gamma,M)$.
    Then for all $t>s>0$ and  $(m,|\alpha|)\in\{0,1\}\times\{0,1,2\}$, there exists a positive constant $N=N(|\alpha|,d,\delta,\varepsilon,\gamma,\kappa,M,m,p_0)$ such that
    for all $a_0,a_1\in(0,\infty)$,
    \begin{equation*}
        \begin{aligned}
            &\int_{a_1}^{\infty}\left(\int_{B_{a_0+\lambda}(x)}|f(z)|^{p_0}\mathrm{d}z\right)^{1/p_0}\left(\lambda^d \int_{\bS^{d-1}}|\p_t^mD_x^{\alpha}\Delta_j P_{\varepsilon}(t,s,\lambda \omega)|^{p_0'}\sigma(\mathrm{d} \omega) \right)^{1/p_0'}\mathrm{d}\lambda\\
            \leq& N\mathrm{e}^{-\kappa|t-s|2^{j\gamma}\times\frac{(1-\delta)}{2^{\gamma}}}2^{j\left((m+\varepsilon)\gamma+|\alpha|-d(p_0)-1+d/p_0\right)}\\
            &\quad\quad\times\left(\int_{a_1}^{\infty}\lambda^{-p_0d(p_0)-1}\int_{B_{a_0+\lambda}(x)}|f(z)|^{p_0}\mathrm{d}z\mathrm{d}\lambda \right)^{1/p_0},
        \end{aligned}
    \end{equation*}
    where $\bS^{d-1}$ denotes the $d-1$-dimensional unit sphere, $\sigma(\mathrm{d}\omega)$ denotes the surface measure on $\bS^{d-1}$, 
    $$
        d(p_0) :=\left\lfloor\frac{d}{p_0}\right\rfloor+1,
    $$
    and $p_0'$ is the H\"older conjugate of $p_0$, \textit{i.e.} $1/p_0+1/p_0'=1$.
\end{lem}
\begin{proof} 
    For notational convenience, we define
    \begin{equation*}
        \begin{gathered}
            \mu:=d(p_0)+\frac{1}{p_0}>\frac{d+1}{p_0}
        \end{gathered}
    \end{equation*}
    and
    $$
        K_{\varepsilon,m,\alpha,j}^{p_0'}(t,s,\lambda):=\int_{\bS^{d-1}}|\p_t^mD_x^{\alpha} \Delta_jP_{\varepsilon}(t,s,\lambda w)|^{p_0'}\sigma(\mathrm{d}\omega).
    $$
    By H\"older's inequality and Theorem \ref{22.02.15.11.27},
    \begin{align*}
        &\int_{a_1}^{\infty}\left(\int_{B_{a_0+\lambda}(x)}|f(z)|^{p_0}\mathrm{d}z\right)^{1/p_0}(\lambda^dK_{\varepsilon,m,\alpha,j}^{p_0'}(t,s,\lambda))^{1/p_0'}\mathrm{d}\lambda\\
        \leq&\left(\int_{a_1}^{\infty}\lambda^{-p_0\mu}\int_{B_{a_0+\lambda}(x)}|f(z)|^{p_0}\mathrm{d}z\mathrm{d}\lambda \right)^{1/p_0}\left(\int_{a_1}^{\infty}\lambda^{d+p_0'\mu}K_{\varepsilon,m,\alpha,j}^{p_0'}(t,s,\lambda)\mathrm{d}\lambda \right)^{1/p_0'}\\
        \leq&\left(\int_{a_1}^{\infty}\lambda^{-p_0\mu}\int_{B_{a_0+\lambda}(x)}|f(z)|^{p_0}\mathrm{d}z\mathrm{d}\lambda \right)^{1/p_0}\left(\int_{\bR^d}|z|^{1+p_0'\mu}|\p_t^mD_x^{\alpha}\Delta_jP_{\varepsilon}(t,s,z)|^{p_0'} \mathrm{d}z\right)^{1/p_0'}\\
        &\leq N\mathrm{e}^{-\kappa|t-s|2^{j\gamma}\times\frac{(1-\delta)}{2^{\gamma}}}2^{j\left((m+\varepsilon)\gamma+|\alpha|-d(p_0)-1+d/p_0\right)}\\
        &\quad\quad\times\left(\int_{a_1}^{\infty}\lambda^{-p_0d(p_0)-1}\int_{B_{a_0+\lambda}(x)}|f(z)|^{p_0}\mathrm{d}z\mathrm{d}\lambda \right)^{1/p_0}.
    \end{align*}
    The lemma is proved.
\end{proof}

\begin{lem}\label{22.02.15.16.27}
    Let $p_0\in(1,2]$ and $f\in \cS(\bR^d)$.
    Suppose that $\psi$ is a symbol satisfying the ellipticity condition with $(\gamma,\kappa)$ and having the $\left( \left\lfloor \frac{d}{p_0}\right\rfloor+2\right)$-times regular upper bound with $(\gamma,M)$.
    Then for all $t>s>0$ and  $(m,|\alpha|)\in\{0,1\}\times\{0,1,2\}$ satisfying
    $$
        (m+\varepsilon)\gamma+|\alpha|+d/p_0-\lfloor d/p_0\rfloor>0,
    $$
    there exists a positive constant $N=N(|\alpha|,d,\varepsilon,\gamma,k,\kappa,M,m,p_0)$ such that
    for all $a_0,a_1\in(0,\infty)$,
    \begin{equation*}
        \begin{aligned}
            &\int_{a_1}^{\infty}\left(\int_{B_{a_0+\lambda}(x)}|f(z)|^{p_0}\mathrm{d}z\right)^{1/p_0}\left(\lambda^d \int_{\bS^{d-1}}|\p_t^mD_x^{\alpha}S_0 P_{\varepsilon}(t,s,\lambda \omega)|^{p_0'}\sigma(\mathrm{d}\omega) \right)^{1/p_0'}\mathrm{d}\lambda\\
            \leq& N\left(\int_{a_1}^{\infty}\lambda^{-p_0d(p_0)-1}\int_{B_{a_0+\lambda}(x)}|f(z)|^{p_0}\mathrm{d}z\mathrm{d}\lambda \right)^{1/p_0},
        \end{aligned}
    \end{equation*}
    where
    $$
        d(p_0) :=\left\lfloor\frac{d}{p_0}\right\rfloor+1
    $$
    and $p_0'$ is the H\"older conjugate of $p_0$, \textit{i.e.} $1/p_0+1/p_0'=1$.
\end{lem}
\begin{proof} 
    First, we choose a $c=c(|\alpha|,d,\varepsilon,\gamma,m,p_0)>1$ so that
    $$
        (m+\varepsilon)\gamma+|\alpha|+d/p_0-\lfloor d/p_0\rfloor>\log_2(c)>0.
    $$
    By H\"older's inequality,
    \begin{equation}\label{22.02.15.16.06}
        \begin{aligned}
            &\int_{\bS^{d-1}}|\p_t^mD_x^{\alpha} S_0P_{\varepsilon}(t,s,\lambda \omega)|^{p_0'}\sigma(\mathrm{d}\omega)\\
            \leq& \left(\sum_{j\leq 0}c^{j/(p_0-1)}\right)^{p_0-1}\int_{\bS^{d-1}}\sum_{j\leq 0}c^{-j}|\p_t^mD_x^{\alpha} \Delta_jP_{\varepsilon}(t,s,\lambda \omega)|^{p_0'}\sigma(\mathrm{d}\omega)\\
            \leq& N(c,p_0)\sum_{j\leq0}c^{-j}K_{\varepsilon,m,\alpha,j}^{p_0'}(t,s,\lambda).
        \end{aligned}
    \end{equation}
Putting
$$
\mu = d(p_0) +\frac{1}{p_0}
$$
and using \eqref{22.02.15.16.06} and H\"older's inequality, we have
    \begin{align*}
        &\int_{a_1}^{\infty}\left(\int_{B_{a_0+\lambda}(x)}|f(z)|^{p_0}\mathrm{d}z\right)^{1/p_0}\left(\lambda^d \int_{\bS^{d-1}}|\p_t^mD_x^{\alpha}S_0 P_{\varepsilon}(t,s,\lambda \omega)|^{p_0'}\sigma(\mathrm{d}\omega) \right)^{1/p_0'}\mathrm{d}\lambda\\
        \leq& N\int_{a_1}^{\infty}\left(\int_{B_{a_0+\lambda}(x)}|f(z)|^{p_0}\mathrm{d}z\right)^{1/p_0}\left(\lambda^d\sum_{j\leq0}c^{-j}K_{\varepsilon,m,\alpha,j}^{p_0'}(t,s,\lambda)  \right)^{1/p_0'}\mathrm{d}\lambda\\
        \leq &N\left(\int_{a_1}^{\infty}\lambda^{-p_0\mu}\int_{B_{a_0+\lambda}(x)}|f(z)|^{p_0}\mathrm{d}z\mathrm{d}\lambda \right)^{1/p_0}\left(\sum_{j\leq0}c^{-j}\int_{a_1}^{\infty}\lambda^{d+p_0'\mu}K_{\varepsilon,m,\alpha,j}^{p_0'}(t,s,\lambda)\mathrm{d}\lambda \right)^{1/p_0'}.
    \end{align*}
    By Theorem \ref{22.02.15.11.27},
    \begin{align*}
        \sum_{j\leq0}c^{-j}\int_{a_1}^{\infty}\lambda^{d+p_0'\mu}K_{\varepsilon,m,\alpha,j}^{p_0'}(t,s,\lambda)\mathrm{d}\lambda& N\leq \sum_{j\leq0}c^{-j}\int_{\bR^d}|z|^{1+p_0'\mu}|\p_t^mD_x^{\alpha}\Delta_jP_{\varepsilon}(t,s,z)|^{p_0'} \mathrm{d}z\\
        &\leq N \sum_{j\leq0}c^{-j}2^{j\left((m+\varepsilon)\gamma+|\alpha|-d(p_0)-1+d/p_0\right)}\\
        &=N\sum_{j\leq0}2^{j\left((m+\varepsilon)\gamma+|\alpha|+d/p_0-\lfloor d/p_0\rfloor-\log_2(c)\right)}=N.
    \end{align*}
    The lemma is proved.
\end{proof}

Making use of Lemmas \ref{20.12.17.20.21}, \ref{22.02.15.16.27}, we want to estimate mean oscillations of $\cT_{t,0}^{\varepsilon,j} f$ and $\cT_{t,0}^{\varepsilon,\leq0}f$.
To do so, we first calculate $L_p$-norms of $\cT_{t,0}^{\varepsilon,j}f$ and $\cT_{t,0}^{\varepsilon,\leq0}f$ with respect to the space variable $x$.

\begin{lem}\label{22.01.28.16.57}
    Suppose that $\psi$ is a symbol satisfying the ellipticity condition with $(\gamma,\kappa)$ and having the $\left( \left\lfloor \frac{d}{2}\right\rfloor+1\right)$-times regular upper bound with $(\gamma,M)$.
    Then for all $t>0$, $p\in[1,\infty]$, and $\delta \in (0,1)$, there exist positive constants $N=N(d,\varepsilon,\gamma,\kappa,M)$, $N(\delta)=N(d,\delta, \varepsilon,\gamma,\kappa,M)$, $N'(\delta)=N'(d,\delta,\gamma,\kappa,M)$, and $N'=N'(d,\gamma,\kappa,M)$ such that for all $f \in \cS(\bR^{d})$ and $ t>0$,
    \begin{equation*}
        \begin{gathered}
         \|\cT_{t,0}^{\varepsilon,j}f\|_{L_p(\bR^{d})}\leq N(\delta)\mathrm{e}^{-\kappa t 2^{j\gamma}\times\frac{(1-\delta)}{2^{\gamma}}}2^{j\varepsilon\gamma}\|f\|_{L_p(\bR^{d})}, \\ \|\cT_{t,0}^{\varepsilon,\leq0}f\|_{L_p(\bR^{d})}\leq N\|f\|_{L_p(\bR^{d})},\\
        \| \psi(t, -i\nabla)\cT_{t,0}^{j}f\|_{L_p(\bR^d)}\leq N'(\delta)\mathrm{e}^{-\kappa t 2^{j\gamma}\times\frac{(1-\delta)}{2^{\gamma}}}2^{j\gamma}\|f\|_{L_p(\bR^{d})},\\
        \|\psi(t, -i\nabla) \cT_{t,0}^{\leq0}f\|_{L_p(\bR^d)}\leq N'\|f\|_{L_p(\bR^{d})}.
        \end{gathered}
    \end{equation*}
\end{lem}
\begin{proof}
Let $ t > 0$.    By Minkowski's inequality and Theorem \ref{22.02.15.11.27},
    \begin{equation*}
    \begin{aligned}
     \|\cT_{t,0}^{\varepsilon,j}f\|_{L_p(\bR^{d})}&\leq \|\Delta_{j}P_{\varepsilon}(t,0,\cdot)\|_{L_1(\bR^d)}\|f\|_{L_p(\bR^d)}\\
     &\leq N\mathrm{e}^{-\kappa t2^{j\gamma}\times\frac{(1-\delta)}{2^{\gamma}}}2^{j\varepsilon\gamma}\|f\|_{L_p(\bR^d)},
         \end{aligned}
    \end{equation*}
    and
    \begin{equation*}
    \begin{aligned}
     \|\psi(t, -i\nabla)\cT_{t,0}^{j}f\|_{L_p(\bR^{d})}&\leq \|\psi(t, -i\nabla) \Delta_{j} p(t,0,\cdot)\|_{L_1(\bR^d)}\|f\|_{L_p(\bR^d)}\\
     &\leq N\mathrm{e}^{-\kappa t2^{j\gamma}\times\frac{(1-\delta)}{2^{\gamma}}}2^{j\gamma}\|f\|_{L_p(\bR^d)}.
    \end{aligned}
    \end{equation*}
    Similarly, using Minkowski's inequality and Corollary \ref{22.02.15.14.36}, we also obtain the other estimates. The lemma is proved.
\end{proof}

\begin{lem}\label{22.01.28.17.18}
    Let $t>0$ and $b>0$.
    \begin{enumerate}[(i)]
        \item 
        Assume that $f\in \cS(\bR^{d})$ has a support in $B_{3\times2^{-j}b}(0)$. 
        Then for any $x\in B_{2^{-j}b}(0)$,
        \begin{equation*}
        \begin{gathered}
         \aint_{B_{2^{-j}b}(0)}|\cT_{t,0}^{\varepsilon,j} f(y)|^{p_0}\mathrm{d}y\leq N2^{j\varepsilon\gamma p_0}\mathrm{e}^{-\kappa t 2^{j\gamma}\times\frac{p_0(1-\delta)}{2^{\gamma}}} \bM\left(|f|^{p_0}\right)(x),\\
         \aint_{B_{2^{-j}b}(0)}|\psi(t,-i\nabla)\cT_{t,0}^{j} f(y)|^{p_0}\mathrm{d}y\leq N'2^{j\gamma p_0}\mathrm{e}^{-\kappa t 2^{j\gamma}\times\frac{p_0(1-\delta)}{2^{\gamma}}} \bM\left(|f|^{p_0}\right)(x),
        \end{gathered}
        \end{equation*}
        where $N=N(d,\delta,\varepsilon,\gamma,\kappa,M)$ and $N=N(d,\delta,\gamma,\kappa,M)$.
        \item 
        Assume that $f\in \cS(\bR^{d})$ has a support in $B_{3b}(0)$. 
        Then for any $x\in B_{b}(0)$,
        \begin{equation*}
        \begin{gathered}
         \aint_{B_{b}(0)}|\cT_{t,0}^{\varepsilon,\leq0} f(y)|^{p_0}\mathrm{d}y\leq N\bM\left(|f|^{p_0}\right)(x),\\
         \aint_{B_{b}(0)}|\psi(t,-i\nabla)\cT_{t,0}^{\leq 0} f(y)|^{p_0}\mathrm{d}y\leq N\bM\left(|f|^{p_0}\right)(x),
        \end{gathered}
        \end{equation*}
        where $N=N(d,\varepsilon,\gamma,\kappa,M)$.
    \end{enumerate}
\end{lem}
\begin{proof}
    By Lemma \ref{22.01.28.16.57},
    \begin{align*}
        \aint_{B_{2^{-j}b}(0)}|\cT_{t,0}^{\varepsilon,j} f(y)|^{p_0}\mathrm{d}y&\leq N\frac{2^{j\varepsilon\gamma p_0}\mathrm{e}^{-\kappa t 2^{j\gamma}\times\frac{p_0(1-\delta)}{2^{\gamma}}}}{|B_{2^{-j}b}(0)|}\int_{\bR^d}|f(y)|^{p_0}\mathrm{d}y\\
        &= N2^{j\varepsilon\gamma p_0}\mathrm{e}^{-\kappa t 2^{j\gamma}\times\frac{p_0(1-\delta)}{2^{\gamma}}}\left(\aint_{B_{3\times 2^{-j}b}(0)}|f(y)|^{p_0}\mathrm{d}y\right)\\
        &\leq N2^{j\varepsilon \gamma p_0}\mathrm{e}^{-\kappa t 2^{j\gamma}\times\frac{p_0(1-\delta)}{2^{\gamma}}}\bM\left(|f|^{p_0}\right)(x).
    \end{align*}
    Similarly, using Lemma \ref{22.01.28.16.57}, we can easily obtain the other results. The lemma is proved.
\end{proof}

With the help of Lemmas \ref{20.12.17.20.21}, \ref{22.02.15.16.27}, and \ref{22.01.28.17.18}, we can prove Lemma \ref{20.12.21.16.26}.
\begin{proof}[Proof of Lemma \ref{20.12.21.16.26}]
    First, we prove $(i)$. Due to similarity, we only prove it for $T_{t,0}^{\varepsilon,j}$. Choose a $\zeta\in C^{\infty}(\bR^d)$ satisfying
    \begin{itemize}
        \item $\zeta(y)\in[0,1]$ for all $y\in\bR^d$
        \item  $\zeta(y)=1$ for all $y\in B_{2\times 2^{-j}b}(0)$
        \item   $\zeta(y)=0$ for all $y\in \bR^d\setminus B_{5\times2^{-j}b/2}(0)$.
    \end{itemize}
    Note that $\cT_{t,0}^{\varepsilon,j} f=\cT_{t,0}^{\varepsilon,j} (f\zeta)+\cT_{t,0}^{\varepsilon,j} (f(1-\zeta))$ and $\cT_{t,0}^{\varepsilon,j}(f\zeta)$ can be estimated by Lemma \ref{22.01.28.17.18}.
    Thus it suffices to estimate  $\cT_{t,0}^{\varepsilon,j} (f(1-\zeta))$ and we may assume that $f(y)=0$ if $|y|<2\times 2^{-j}b$. Hence if $y\in B_{2^{-j}b}(0)$ and $|z|<2^{-j}b$, then $|y-z|\leq 2\times 2^{-j}b$ and $f(y-z)=0$. By \cite[Lemma 6.6]{Choi_Kim2022} and H\"older's inequality,
    \begin{equation}
    \label{22.10.24.11.37}
    \begin{aligned}
        &\left|\int_{\bR^d}(\Delta_{j}P_{\varepsilon}(t,0,y_0-z)-\Delta_{j}P_{\varepsilon}(t,0,y_1-z))f(z)\mathrm{d}z\right|\\
        \leq& N|y_0-y_1| \int_{2^{-j}b}^{\infty}(\lambda^dK_{\varepsilon,0,\alpha,j}^{p_0'}(t,0,\lambda))^{1/p_0'}\left(\int_{B_{2\times 2^{-j}b+\lambda}(x)}|f(z)|^{p_0}\mathrm{d}z\right)^{1/p_0}\mathrm{d}\lambda,
    \end{aligned}
    \end{equation}
    where $N=N(d,p_0)$, $|\alpha|=2$, and
    $$
        K_{\varepsilon,0,\alpha,j}^{p_0'}(t,0,\lambda):=\int_{\bS^{d-1}}|D_x^{\alpha}\Delta_j P_{\varepsilon}(t,0,\lambda \omega)|^{p_0'}\sigma(\mathrm{d}\omega).
    $$
    If $b\leq1$, then by Lemma \ref{20.12.17.20.21} and \eqref{22.10.24.11.37},
    \begin{align*}
        &|\cT_{t,0}^{\varepsilon,j}f(y_0)-\cT_{t,0}^{\varepsilon,j}f(y_1)|^{p_0}\\
        \leq& N|y_0-y_1|^{p_0}\mathrm{e}^{-\kappa t 2^{j\gamma}\times\frac{p_0(1-\delta)}{2^{\gamma}}}2^{jp_0\left(\varepsilon\gamma-d(p_0)+1+d/p_0\right)}\\
        &\quad\quad \times \left(\int_{2^{-j}b}^{\infty}\lambda^{-p_0d(p_0)-1}\int_{B_{2\times 2^{-j}b+\lambda}(x)}|f(z)|^{p_0}\mathrm{d}z\mathrm{d}\lambda \right)\\
        \leq& N2^{j(\varepsilon\gamma p_0-p_0d(p_0)+d)}\mathrm{e}^{-\kappa t 2^{j\gamma}\times\frac{p_0(1-\delta)}{2^{\gamma}}}b^{p_0}\int_{2^{-j}b}^{\infty}\lambda^{-p_0d(p_0)-1}\left(\int_{B_{2\times 2^{-j}b+\lambda}(x)}|f(z)|^{p_0}\mathrm{d}z\right)\mathrm{d}\lambda\\
        \leq &N2^{j(\varepsilon\gamma p_0-p_0d(p_0)+d)}\mathrm{e}^{-\kappa t 2^{j\gamma}\times\frac{p_0(1-\delta)}{2^{\gamma}}}\bM\left(|f|^{p_0}\right)(x)b^{p_0}\int_{2^{-j}b}^{\infty}\lambda^{-p_0d(p_0)+d-1}\mathrm{d}\lambda\\
        \leq&  N2^{j\varepsilon\gamma p_0}b^{-p_0d(p_0)+d+p_0}\mathrm{e}^{-\kappa t 2^{j\gamma}\times\frac{p_0(1-\delta)}{2^{\gamma}}}\bM\left(|f|^{p_0}\right)(x)\\
        \leq& N2^{j\varepsilon\gamma p_0}\mathrm{e}^{-\kappa t 2^{j\gamma}\times\frac{p_0(1-\delta)}{2^{\gamma}}}\bM\left(|f|^{p_0}\right)(x).
    \end{align*}
    If $b>1$ and $y\in B_{2^{-j}b}(0)$, then by \cite[Lemma 6.6]{Choi_Kim2022} with $|\alpha|=1$ and Lemma \ref{20.12.17.20.21},
    \begin{align*}
        &|\cT_{t,0}^{\varepsilon,j}f(y)|^{p_0}\\
        \leq& N\left(\int_{2^{-j}b}^{\infty}\lambda^{-p_0d(p_0)-1}\int_{B_{2\times 2^{-j}b+\lambda}(x)}|f(z)|^{p_0}\mathrm{d}z\mathrm{d}\lambda \right)\mathrm{e}^{-\kappa t 2^{j\gamma}\times\frac{p_0(1-\delta)}{2^{\gamma}}}2^{jp_0\left(\varepsilon\gamma-d(p_0)+d/p_0\right)}\\
        \leq& N2^{j(\varepsilon\gamma p_0-p_0d(p_0)+d)}\mathrm{e}^{-\kappa t 2^{j\gamma}\times\frac{p_0(1-\delta)}{2^{\gamma}}}\int_{2^{-j}b}^{\infty}\lambda^{-p_0d(p_0)-1}\left(\int_{B_{2\times 2^{-j}b+\lambda}(x)}|f(z)|^{p_0}\mathrm{d}z\right)\mathrm{d}\lambda\\
        \leq& N2^{j(\varepsilon\gamma p_0-p_0d(p_0)+d)}\mathrm{e}^{-\kappa t 2^{j\gamma}\times\frac{p_0(1-\delta)}{2^{\gamma}}}\bM\left(|f|^{p_0}\right)(x)\int_{2^{-j}b}^{\infty}\lambda^{-p_0d(p_0)+d-1}\mathrm{d}\lambda\\
        \leq&  N2^{j\varepsilon\gamma p_0}b^{-p_0d(p_0)+d}\mathrm{e}^{-\kappa t 2^{j\gamma}\times\frac{p_0(1-\delta)}{2^{\gamma}}}\bM\left(|f|^{p_0}\right)(x)\\
        \leq& N2^{j\varepsilon\gamma p_0}\mathrm{e}^{-\kappa t 2^{j\gamma}\times\frac{p_0(1-\delta)}{2^{\gamma}}}\bM\left(|f|^{p_0}\right)(x).
    \end{align*}
    For $(ii)$, note that if $|\alpha|\geq1$, then
    $$
        |\alpha|+d/p_0-\lfloor d/p_0\rfloor\geq |\alpha|>0.
    $$
    Thus, applying the similar arguments in $(i)$ with Lemma \ref{22.02.15.16.27} instead of Lemma \ref{20.12.17.20.21}, we also have $(ii)$. 
    The lemma is proved.
\end{proof}

\appendix

\mysection{Weighted multiplier and Littlewood--Paley theorem}

\begin{prop}[Weighted Mikhlin multiplier theorem]
			\label{21.02.24.16.49}
Let $p\in(1,\infty)$, $w\in A_p(\bR^{d})$, and $f \in \cS(\bR^{d})$. 
Suppose that \begin{equation}
		\label{22.05.12.13.14}
\sup_{R>0}\left(R^{2|\alpha|-d}\int_{R<|\xi|<2R}|D^{\alpha}_{\xi}\pi(\xi)|^{2}\mathrm{d}\xi\right)^{1/2}\leq N^*,\quad\forall |\alpha|\leq d.
\end{equation}
Then there exists a constant $N=N(d,p,K,N^*)$ such that
$$
\|\bT_{\pi}f\|_{L_p(\bR^d,w\,\mathrm{d}x)}\leq N\|f\|_{L_p(\bR^d,w\,\mathrm{d}x)},
$$
where $[w]_{A_p(\bR^d)}\leq K$ and
$$
\bT_{\pi}f(x):=\cF^{-1}[\pi\cF[f]](x).
$$
\end{prop}
\begin{proof}
By \cite[Theorem 6.2.7]{grafakos2014classical}, the operator $\bT_{\pi}:L_q(\bR^d)\to L_{q}(\bR^d)$ is bounded for all $q\in(1,\infty)$. It is well-known that there exists $r>1$ such that $w\in A_{p/r}(\bR^d)$.
Applying \cite[Corollaries 6.10, 6.11, and Remark 6.14]{fackler2020weighted}, we have
\begin{align*}
\|\bT_{\pi} f\|_{L_p(\bR^d,w\,\mathrm{d}x)}&\leq N\|f\|_{L_p(\bR^d,w\,\mathrm{d}x)},
\end{align*}
where $N=N(d,p,K,N^*)$. The proposition is proved.
\end{proof}

\begin{prop}[Weighted Littlewood--Paley theorem]
\label{prop:WLP}
    Let $p\in (1, \infty)$ and $w\in A_p(\bR^d)$. Then we have 
    \begin{align}
										\label{20230128 01}
        \|Sf\|_{L_p(\bR^d,w\,\mathrm{d}x)} &\leq C(d,p) [w]_{A_p(\bR^d)}^{\max(\frac{1}{2}, \frac{1}{p-1})} \|f\|_{L_p(\bR^d,w\,\mathrm{d}x)}
\end{align}
and
\begin{align}
						\label{20230128 03}
        \|f\|_{L_p(\bR^d,w\,\mathrm{d}x)} &\leq C(d,p) [w]_{A_p(\bR^d)}^{\frac{\max(1/2, 1/(p'-1))}{p-1}} \|Sf \|_{L_p(\bR^d,w\,\mathrm{d}x)},
    \end{align}
    where 
\begin{align}
							\label{20230128 02}
    Sf(x):=\left(\sum_{j\in\bZ}|\Delta_jf(x)|^2\right)^{1/2}.
\end{align}
\end{prop}

\begin{proof}
This result is already proved in several literature \cite{kurtz1980little,rychkov2001weights}. They did not, however, reveal the growth of implicit constants relative to $A_p(\bR^d)$-seminorm. 
These optimum implicit constants could be obtained from recent general theories.
    The first inequality is proved in \cite{Ler2011weighted,Wil2007square} for various types of Littlewood--Paley operators.
    We show that $Sf$ is one of such Littlewood--Paley operators.
    The second inequality follows from the first inequality and the duality of $L_p(\bR^d,w\,\mathrm{d}x)$, which will be shown in the last part of the proof.

For $\alpha \in (0,1]$, let $\mathcal{C}_\alpha$ be a family of functions $\phi:\bR^d \to \bR$ supported in $\{x\in\bR^d : |x|\leq 1\}$ such that
    \begin{align*}
        \int_{\bR^d} \phi(x) \mathrm{d}x=0,\quad |\phi(x) - \phi(x')| \leq |x-x'|^\alpha,\quad \forall x, x' \in \bR^d.
    \end{align*}
    Then we define a maximal operator over the family $\mathcal{C}_\alpha$ as
    \begin{align}\label{ineq:22 05 06 17 32}
        A_\alpha(f)(x,t) := \sup_{\phi \in \mathcal{C}_\alpha} |\phi_t \ast f(x)|,\quad \phi_t(y) = t^{-d}\phi(t^{-1}y).
    \end{align}
    Using \eqref{ineq:22 05 06 17 32}, we construct intrinsic square functions as follows:
    \begin{align*}
        G_\alpha(f)(x) &:= \biggl(\int_{\Gamma(x)}\big|A_\alpha(f)(y,t)\big|^2 \frac{\mathrm{d}y\mathrm{d}t}{t^{d+1}}\biggr)^{1/2},\\
        g_\alpha(f)(x) &:= \biggl(\int_0^\infty\big|A_\alpha(f)(x,t)\big|^2\frac{\mathrm{d}t}{t}\biggr)^{1/2},\\
        \sigma_\alpha(f)(x) &:= \biggl(\sum_{j\in\bZ} \big| A_\alpha(f)(x, 2^j)\big|^2\biggr)^{1/2},
    \end{align*}
    where $\Gamma(x)$ denotes the conic area $\{(y,t) \in \bR^d\times\bR_+ : |x-y|\leq t\}$. 
    In \cite{Wil2007square}, Wilson  showed pointwise equivalences among $G_\alpha$, $g_\alpha$ and $\sigma_\alpha$, \textit{i.e.}
    \begin{align}\label{ineq:22 05 06 18 12}
        G_\alpha(f)(x) \approx g_\alpha(f)(x) \approx \sigma_\alpha(f)(x),
    \end{align}
where the implicit constants depend only on $\alpha$ and $d$.
    Moreover, Lerner \cite[Theorem 1.1]{Ler2011weighted} proved 
    \begin{align}\label{ineq:WLP}
        \|G_\alpha\|_{L_p(\bR^d,w\,\mathrm{d}x)\to L_p(\bR^d,w\,\mathrm{d}x)} \leq C(\alpha, d, p) [w]_{A_p(\bR^d)}^{\max(\frac{1}{2}, \frac{1}{p-1})}.
    \end{align}
    It should be remarked that \cite[Theorem 1.1]{Ler2011weighted} covers a broad class of operators of Littlewood--Paley type.
However, this result does not give \eqref{20230128 01} directly since the Littlewood--Paley operator considered in this paper is not an integral form (see \eqref{20230128 02}).
Our Littlewood--Paley operator is given  as the summation of $\Delta_j f = \Psi_j \ast f$ over $\bZ$ and the symbol of $\Delta_j$ is supported in $\{\xi\in\bR^d:2^{j-1}\leq|\xi| \leq2^{j+1} \}$.
    Note that $\Psi_j$ is globally defined due to the uncertainty principle, while elements in $\mathcal{C}_\alpha$ are supported in $\{x\in\bR^d:|x|\leq 1\}$.
    To fill this gap, we need a new family $\mathcal{C}_{\alpha,\varepsilon}$  introduced in \cite{Wil2007square}.
    For $\alpha \in (0,1]$ and $\varepsilon>0$, let $\mathcal{C}_{\alpha, \varepsilon}$ be a family of functions $\phi:\bR^d \to \bR$ satisfying
    \begin{equation*}
        \begin{gathered}
        \int_{\bR^d} \phi(x) \mathrm{d}x=0,\quad |\phi(x)| \leq (1+|x|)^{-d-\varepsilon},\\
        |\phi(x) - \phi(x')| \leq |x-x'|^\alpha \bigl((1+|x|)^{-d-\varepsilon} + (1+|x'|)^{-d-\varepsilon}\bigr),\quad \forall x, x' \in \bR^d.
    \end{gathered}
    \end{equation*}
    Then we can define $\widetilde{A}_{\alpha, \varepsilon}$ on the basis of all functions in the above class as in \eqref{ineq:22 05 06 17 32}, \textit{i.e.}
    \begin{align}
        \widetilde{A}_{\alpha,\varepsilon}(f)(x,t) := \sup_{\phi \in \mathcal{C}_{\alpha,\varepsilon}} |\phi_t \ast f(x)|,\quad \phi_t(y) = t^{-d}\phi(t^{-1}y).
    \end{align}
    We can also define $\widetilde{G}_{\alpha, \varepsilon}, \widetilde{g}_{\alpha, \varepsilon}$ and $\widetilde{\sigma}_{\alpha, \varepsilon}$ similarly to $G_\alpha$, $g_\alpha$ and $\sigma_\alpha$.    Likewise, they become equivalent, \textit{i.e.}
    \begin{align}\label{ineq:22 05 06 18 13}
        \widetilde{G}_{\alpha,\varepsilon}(f)(x) \approx \widetilde{g}_{\alpha,\varepsilon}(f)(x) \approx \widetilde{\sigma}_{\alpha,\varepsilon}(f)(x).
    \end{align}
    Here, the implicit constants depend only on $\alpha,\varepsilon$, and $d$. Since every Schwartz function is contained in $\mathcal{C}_{1,1}$, it follows that $\Psi \in \mathcal{C}_{1,1}$.
    Thus we have
    \begin{align*}
        |\Delta_jf(x)| \leq \widetilde{A}_{1,1}(f)(x, 2^j),
    \end{align*}
    which yields 
    \begin{align}\label{ineq:22 05 06 18 27}
        Sf(x) \leq C(d)\widetilde{\sigma}_{1,1}(f)(x)\approx C(d) \widetilde{G}_{1,1}(f)(x).
    \end{align}
It is also known in \cite[Theorem 2]{Wil2007square}, for all $\alpha \in (0,1]$, $\varepsilon>0$, and $0<\alpha' \leq \alpha$, there is a constant $C = C(\alpha, \alpha', \varepsilon, d)$ such that
    \begin{align}\label{ineq:22 05 06 18 28}
        \widetilde{G}_{\alpha, \varepsilon}(f)(x) \leq C G_{\alpha'}(f)(x).
    \end{align}
Thus finally by \eqref{ineq:WLP}, \eqref{ineq:22 05 06 18 13}, \eqref{ineq:22 05 06 18 27} and \eqref{ineq:22 05 06 18 28}, we have
    \begin{equation}
     \label{22.05.12.16.08}
     \begin{aligned}
        \big\|Sf \big\|_{L_p(\bR^d,w\,\mathrm{d}x)} 
        &\leq C(d) \big\|\widetilde{G}_{1,1}(f) \big\|_{L_p(\bR^d,w\,\mathrm{d}x)}\\
        &\leq C(d) \big\|G_{1}(f) \big\|_{L_p(\bR^d,w\,\mathrm{d}x)} \leq C(d, p) [w]_{A_p(\bR^d)}^{\max(\frac{1}{2}, \frac{1}{p-1})} \| f\|_{L^p(\bR^d,w\,\mathrm{d}x)}.
    \end{aligned}
    \end{equation}
    For the converse inequality \eqref{20230128 03}, we recall that the topological dual space of $L_p(\bR^d,w\,\mathrm{d}x)$ is $L_{p'}(\bR^d,\bar{w}\,\mathrm{d}x)$ where
    $$
    \frac{1}{p}+\frac{1}{p'}=1,\quad \bar{w}(x):=(w(x))^{-\frac{1}{p-1}}.
    $$
    For more detail, see \cite[Theorem A.1]{Choi_Kim2022}.
Then by the almost orthogonality of $\Delta_j$, the Cauchy-Schwartz inequality, H\"older's inequality, and \eqref{22.05.12.16.08}, for all $f\in L_p(\bR^d,w\,\mathrm{d}x)$ and $g\in C_c^{\infty}(\bR^d)$,
    \begin{equation}
    \label{22.05.12.16.15}
        \begin{aligned}
    \int_{\bR^d}f(x)g(x)\mathrm{d}x&=\int_{\bR^d}\sum_{j\in\bZ}\Delta_jf(x)(\Delta_{j-1}+\Delta_j+\Delta_{j+1})g(x)\mathrm{d}x\\
    &\leq 3\int_{\bR^d}Sf(x)Sg(x)\mathrm{d}x\leq 3\|Sf\|_{L_p(\bR^d,w\,\mathrm{d}x)}\|Sg\|_{L_{p'}(\bR^d,\bar{w}\,\mathrm{d}x)}\\
    &\leq N(d,p)[\bar{w}]_{A_{p'}(\bR^d)}^{\max(\frac{1}{2},\frac{1}{p'-1})}\|Sf\|_{L_p(\bR^d,w\,\mathrm{d}x)}\|g\|_{L_{p'}(\bR^d,\bar{w}\,\mathrm{d}x)}.
    \end{aligned}
    \end{equation}
Combining \eqref{22.05.12.16.08} and \eqref{22.05.12.16.15}, we have
$$
\|f\|_{L_p(\bR^d,w\,\mathrm{d}x)}=\sup_{\|g\|_{L_{p'}(\bR^d,\bar{w}\,\mathrm{d}x)}\leq 1}\left|\int_{\bR^d}f(x)g(x)\mathrm{d}x\right|\leq N(d,p)[w]_{A_{p}(\bR^d)}^{\frac{\max(1/2,1/(p'-1)
)}{p-1}}\|Sf\|_{L_p(\bR^d,w\,\mathrm{d}x)}.
$$
The proposition is proved.
\end{proof}

\mysection{Properties of function spaces}
\begin{lem}
\label{wbound}
Let $p\in(1,\infty)$ and $w\in A_p(\bR^d)$. Suppose that
$$
[w]_{A_p(\bR^d)}\leq K.
$$
Then there exists a positive constant $N=N(d,p,K)$ such that
$$
\|S_0f\|_{L_p(\bR^d,w\,\mathrm{d}x)}+\sup_{j\in\bZ}\|\Delta_jf\|_{L_p(\bR^d,w\,\mathrm{d}x)}\leq N\|f\|_{L_p(\bR^d,w\,\mathrm{d}x)},\quad \forall f\in L_p(\bR^d,w\,\mathrm{d}x).
$$
\end{lem}
\begin{proof}
By the definition of $S_0$, there exists a $\Phi\in\cS(\bR^d)$ such that
$$
S_0f(x)=\Phi\ast f(x)=\int_{B_1(0)}f(x-y)\Phi(y)\mathrm{d}y+\sum_{k=1}^{\infty}\int_{B_{2^k}(0)\setminus\overline{B_{2^{k-1}}(0)}}f(x-y)\Phi(y)\mathrm{d}y.
$$
Since $\Phi\in\cS(\bR^d)$, there exists $N=N(d,\Phi)$ such that
$$
|\Phi(y)|+|y|^{d+1}|\Phi(y)|\leq N,\quad \forall y\in\bR^d.
$$
Therefore, $| S_0 f (x) | \leq N(d,\Phi)\bM f(x)$ and it leads to the first part of the inequality, \textit{i.e.}
$$
\|S_0f\|_{L_p(\bR^d,w\,\mathrm{d}x)}\leq N\|f\|_{L_p(\bR^d,w\,\mathrm{d}x)},\quad \forall f\in L_p(\bR^d,w\,\mathrm{d}x)
$$
due to the weighted Hardy--Littlewood Theorem. 
Next recall
$$
    \Delta_jf=f\ast\Psi_j=f \ast 2^{jd}\Psi(2^{j}\cdot),
    $$
and we show that $\Delta_j$ is a Calder\'on--Zygmund operator to obtain
        \begin{equation}
    \label{22.04.11.14.12}
        \sup_{j\in\bZ}\|\Delta_jf\|_{L_p(\bR^d,w\,\mathrm{d}x)}\leq N\|f\|_{L_p(\bR^d,w\,\mathrm{d}x)}.
    \end{equation}
In other words, we have to show that $f \to \Delta_j f$ is bounded on $L_2(\bR^d)$ and $\Psi_j(x-y)$ is a standard kernel. 
It is obvious that 
$$
\sup_{\xi \in \bR^d} |\cF[\Psi_j](\xi)|=\sup_{\xi \in \bR^d}|\cF[\Psi](2^{-j} \xi)| = \sup_{\xi \in \bR^d}|\cF[\Psi](\xi)|,
$$
which implies that $f \to \Delta_jf$ becomes a bounded operator on $L_2(\bR^d)$  due to the Plancherel theorem.
More precisely, we have
\begin{align}
								\label{20230128 30}
\|\Delta_j f\|_{L_2(\bR^d)} \leq \frac{1}{ (2\pi)^{d/2} } \sup_{\xi \in \bR^d}|\cF[\Psi](\xi)| \|f\|_{L_2(\bR^d)}
\end{align}
    where $\Psi_j=2^{jd}\Psi(2^{j}\cdot)$ and $\Psi\in\cS(\bR^d)$.
Next, we show that $\Psi_j(x-y)$ is a standard kernel.
   Observe that there exists a $N=N(d,\Psi)$ such that
    \begin{align}
								\label{20230128 10}
    |x|^d|\Psi_j(x)|+|x|^{d+1}|\nabla\Psi_j(x)|=|2^jx|^d|\Psi(2^{j}x)|+|2^jx|^{d+1}|\nabla\Psi(2^jx)|\leq N,\quad \forall x\in\bR^d.
    \end{align}
    Moreover, by the fundamental theorem of calculus,
    $$
    |\Psi_j(x)-\Psi_j(y)|\leq |x-y|\int_0^1|\nabla \Psi_j(\theta x+(1-\theta)y)|\mathrm{d}\theta.
    $$
    If $2|x-y|\leq\max(|x|,|y|)$, then for $\theta\in[0,1]$,
    $$
    |\theta x+(1-\theta)y|\geq \max(|x|,|y|)-|x-y|\geq \frac{\max(|x|,|y|)}{2}\geq \frac{|x|+|y|}{4}.
    $$
    Hence,
    \begin{align}
							\label{20230128 11}
    |\Psi_j(x)-\Psi_j(y)| \leq |x-y|\int_0^1|\nabla \Psi_j(\theta x+(1-\theta)y)|d\theta\leq \frac{N4^{d+1}|x-y|}{(|x|+|y|)^{d+1}}.
    \end{align}
Due to \eqref{20230128 10} and \eqref{20230128 11}, $\Psi_j(x-y)$ becomes a standard kernel with constant $1$ and $N$.
Especially, $N$ can be chosen uniformly for all $j \in \bZ$.
In total, the $L_2$-bounds are given uniformly for all $j \in \bZ$ in \eqref{20230128 30} and $\Psi_j(x-y)$ becomes a standard kernel with the same parameters for all $j \in \bZ$.
Therefore applying \cite[Theorem 7.4.6]{grafakos2014classical}, we have \eqref{22.04.11.14.12}. The lemma is proved.
\end{proof}
\begin{rem}\label{rem RdF vv}
    An operator $T:L_p(\bR^d)\to L_p(\bR^d)$ is linearizable if there exist a Banach space $B$ and a $B$-valued linear operator $U:L_p(\bR^d)\to L_p(\bR^d;B)$ such that
    $$
        |Tf(x)| = \| Uf(x)\|_B,\quad f\in L_p(\bR^d).
    $$
    Hence, any linear operator is linearizable.
    
    Let $\{T_j\}_{j\in\bZ}$ be a sequence of linearizable operators and $K \in (0,\infty)$.
        Assume that there exist a $r \in (1,\infty)$ and $C(K) \in (0,\infty)$ such that 
    \begin{equation}
    \label{22.09.13.16.42}
        \sup_{j\in\bZ}\int_{\bR^d} | T_j f(x)|^r w(x) \mathrm{d}x \leq C(K) \int_{\bR^d} |f(x)|^r w(x)\mathrm{d}x
    \end{equation}
for all $[w]_{A_r} \leq K$ and $f\in L_r(\bR^d,w\,\mathrm{d}x)$.
   Then for all $1<p,q<\infty$, there exist $K' \in (0,\infty)$ and $C$ such that
    \begin{align}
										\label{20230128 40}
        \Big\| \Big( \sum_{j\in\bZ} |T_j f_j|^p\Big)^{1/p} \Big\|_{L_q(\bR^d,w'\,\mathrm{d}x)} \leq C\Big\| \Big( \sum_{j\in\bZ} |f_j|^p\Big)^{1/p} \Big\|_{L_q(\bR^d,w'\,\mathrm{d}x)}
    \end{align}
for all $[w']_{A_q(\bR^d)}\leq K'$ and $f_j \in L_q(\bR^d,w'\,\mathrm{d}x)$. The above statement can be found, for instance, in \cite[p. 521, Remarks 6.5]{RdF1985weighted} with a modification to the dependence of $A_p$-norms. 
In particular, since $\Delta_j$ is linear and satisfies the weighted inequality for any $r>1$ due to Lemma \ref{wbound}, we have \eqref{20230128 40}
with $T_j=\Delta_j$, and it will be used in the proof of the next proposition. 
\end{rem}

For $\boldsymbol{r}:\bZ\to(0,\infty)$, we denote
$$
\pi_{\boldsymbol{r}}:=\sum_{j\in\bZ}2^{\boldsymbol{r}(j)}\Delta_j~\text{and}~ \pi^{\boldsymbol{r}}:=S_0+\sum_{j=1}^{\infty}2^{\boldsymbol{r}(j)}\Delta_j.
$$

\begin{prop}
			\label{22.05.03.11.34}
Let $p\in(1,\infty)$, $q\in(0,\infty)$, and $w\in A_p(\bR^d)$. Suppose that two sequences $\boldsymbol{r},\boldsymbol{r}':\bZ\to(-\infty,\infty)$ satisfy
\begin{align}
							\label{uniform r assumption}
\sup_{j\in\bZ}\left(|\boldsymbol{r}'(j+1)-\boldsymbol{r}'(j)|+|\boldsymbol{r}(j+1)-\boldsymbol{r}(j)|\right)=:C_0<\infty.
\end{align}
Then for any $f\in \cS(\bR^d)$,
    \begin{equation*}
        \begin{gathered}
        \|\pi^{\boldsymbol{r}}f\|_{H_p^{\boldsymbol{r}'}(\bR^d,w\,\mathrm{d}x)}\approx \|f\|_{H_p^{\boldsymbol{r}+\boldsymbol{r}'}(\bR^d,w\,\mathrm{d}x)},\quad \|\pi^{\boldsymbol{r}}f\|_{B_{p,q}^{\boldsymbol{r}'}(\bR^d,w\,\mathrm{d}x)}\approx\|f\|_{B_{p,q}^{\boldsymbol{r}+\boldsymbol{r}'}(\bR^d,w\,\mathrm{d}x)},\\
        \|\pi_{\boldsymbol{r}}f\|_{\dot{H}_p^{\boldsymbol{r}'}(\bR^d,w\,\mathrm{d}x)}\approx \|f\|_{\dot{H}_p^{\boldsymbol{r}+\boldsymbol{r}'}(\bR^d,w\,\mathrm{d}x)},\quad 
    \|\pi_{\boldsymbol{r}}f\|_{\dot{B}_{p,q}^{\boldsymbol{r}'}(\bR^d,w\,\mathrm{d}x)}\approx\|f\|_{\dot{B}_{p,q}^{\boldsymbol{r}+\boldsymbol{r}'}(\bR^d,w\,\mathrm{d}x)},
        \end{gathered}
    \end{equation*}
where  the equivalences depend only on $p$, $q$, $d$, $C_0$ and $[w]_{A_p}(\bR^d)$.
In particular, we have
    $$
        \|\pi^{\boldsymbol{r}}f\|_{L_p(\bR^d,w\,\mathrm{d}x)}\approx \|f\|_{H_p^{\boldsymbol{r}}(\bR^d,w\,\mathrm{d}x)}.
    $$
\end{prop}
\begin{proof}
Due to the existence of $S_0$, the proof of the inhomogeneous case becomes more difficult. 
Even the case $\boldsymbol{r}'\neq\boldsymbol{0}$ is quite similar to the case $\boldsymbol{r}'=\boldsymbol{0}$, where
$\boldsymbol{0}(j)=0$ for all $j\in\bZ$.
Thus, we only prove the inhomogeneous case with the assumption $\boldsymbol{r}'=\boldsymbol{0}$.

First, by the definition of $\pi^{\boldsymbol{r}}$ and the triangle inequality, it is obvious that
    \begin{align}
								\label{20230128 50}
        \|\pi^{\boldsymbol{r}} f\|_{L_p(\bR^d,w\,\mathrm{d}x)} 
        \leq \| S_0 f\|_{L_p(\bR^d,w\,\mathrm{d}x)} + \Big\| \sum_{j=1}^{\infty}2^{\boldsymbol{r}(j)}\Delta_j f \Big\|_{L_p(\bR^d,w\,\mathrm{d}x)}.
    \end{align}
    For the converse inequality of \eqref{20230128 50}, we use Proposition \ref{21.02.24.16.49}. One can easily check that
    $$
    \Pi_0(\xi):=\frac{\cF[\Phi](\xi)}{\cF[\Phi](\xi)+\sum_{j=1}^{\infty}2^{\boldsymbol{r}(j)}\cF[\Psi](2^{-j}\xi)}
    $$
    is infinitely differentiable, $\Pi_0(\xi)=1$ if $|\xi|\leq 1$ and $\Pi_0(\xi)=0$ if $|\xi|\geq2$, where
    $$
    \cF[\Phi](\xi) = \sum_{j=-\infty}^{0}\cF[\Psi](2^{-j}\xi).
    $$
    This certainly implies that $\Pi_0$ satisfies \eqref{22.05.12.13.14}. Hence by Proposition \ref{21.02.24.16.49},
    $$
    \|S_0f\|_{L_p(\bR^d,w\,\mathrm{d}x)}=\|\bT_{\Pi_0}\pi^{\boldsymbol{r}}f\|_{L_p(\bR^d,w\,\mathrm{d}x)}\leq N\|\pi^{\boldsymbol{r}}f\|_{L_p(\bR^d,w\,\mathrm{d}x)},
    $$
where 
$$
\bT_{\Pi_0}f(x):=\cF^{-1}[\Pi_0\cF[f]](x).
$$
    This also yields
    $$
    \Big\| \sum_{j=1}^{\infty}2^{\boldsymbol{r}(j)}\Delta_j f \Big\|_{L_p(\bR^d,w\,\mathrm{d}x)}=\|\pi^{\boldsymbol{r}}f-S_0f\|_{L_p(\bR^d,w\,\mathrm{d}x)}\leq N\|\pi^{\boldsymbol{r}}f\|_{L_p(\bR^d,w\,\mathrm{d}x)}
    $$
    and thus
\begin{align}
								\label{20230128 51}
    \|\pi^{\boldsymbol{r}} f\|_{L_p(\bR^d,w\,\mathrm{d}x)} 
        \approx \| S_0 f\|_{L_p(\bR^d,w\,\mathrm{d}x)} + \Big\| \sum_{j=1}^{\infty}2^{\boldsymbol{r}(j)}\Delta_j f \Big\|_{L_p(\bR^d,w\,\mathrm{d}x)}.
\end{align}
Moreover,   Proposition \ref{prop:WLP} implies 
    \begin{align}
									\label{20230128 52}
       \Big\| S\Big( \sum_{j=1}^{\infty}2^{\boldsymbol{r}(j)}\Delta_j f \Big) \Big\|_{L_p(\bR^d,w\,\mathrm{d}x)}
        \approx
        \Big\| \sum_{j=1}^{\infty}2^{\boldsymbol{r}(j)}\Delta_j f \Big\|_{L_p(\bR^d,w\,\mathrm{d}x)},
    \end{align}
    where the implicit constant depends only on $p$, $d$ and $[w]_{A_p(\bR^d)}$. 
    Therefore by \eqref{20230128 51} and \eqref{20230128 52},
    \begin{align*}
        \|\pi^{\boldsymbol{r}} f\|_{L_p(\bR^d,w\,\mathrm{d}x)} \approx \| S_0 f\|_{L_p(\bR^d,w\,\mathrm{d}x)} + \Big\| S\Big( \sum_{j=1}^{\infty}2^{\boldsymbol{r}(j)}\Delta_j f \Big) \Big\|_{L_p(\bR^d,w\,\mathrm{d}x)}.
    \end{align*}
Next we claim 
    \begin{equation}\label{ineq 220718 1747}
        \begin{aligned}
            \Big\| S\Big( \sum_{j=1}^{\infty}2^{\boldsymbol{r}(j)}\Delta_j f \Big) \Big\|_{L_p(\bR^d,w\,\mathrm{d}x)} 
            \leq N\Big\| \Big( \sum_{j=1}^{\infty} | 2^{\boldsymbol{r}(j)}\Delta_j f |^2 \Big)^{1/2} \Big\|_{L_p(\bR^d,w\,\mathrm{d}x)}.
        \end{aligned}
    \end{equation}
Put $T_j:=\Delta_j$ for all $j\in\bZ$ and
    \begin{equation*}
        \begin{gathered} f_0:=2^{\boldsymbol{r}(1)}\Delta_1f,\quad f_1:=2^{\boldsymbol{r}(1)}\Delta_1f+2^{\boldsymbol{r}(2)}\Delta_2f\\ f_j:=(2^{\boldsymbol{r}(j-1)}\Delta_{j-1}+2^{\boldsymbol{r}(j)}\Delta_{j}+2^{\boldsymbol{r}(j+1)}\Delta_{j+1})f\quad \text{if } j\geq2,\quad f_j:=0 \quad \text{if } j<0.
        \end{gathered}
    \end{equation*}
Considering the almost orthogonal property of $\Delta_j$ and Remark \ref{rem RdF vv}, we have
    \begin{align*}
        \Big\| S\Big( \sum_{j=1}^{\infty}2^{\boldsymbol{r}(j)}\Delta_j f \Big) \Big\|_{L_p(\bR^d,w\,\mathrm{d}x)} 
        &=\Big\| \Big( \sum_{j\in\bZ}|T_jf_j|^2 \Big)^{1/2} \Big\|_{L_p(\bR^d,w\,\mathrm{d}x)}\\
        &\leq N\Big\| \Big( \sum_{j=1}^{\infty} | 2^{\boldsymbol{r}(j)}\Delta_j f |^2 \Big)^{1/2} \Big\|_{L_p(\bR^d,w\,\mathrm{d}x)}.
    \end{align*}
    For the converse of \eqref{ineq 220718 1747}, we make use of the Khintchine inequality (\textit{e.g.} see \cite{Haagerup1981}),
    \begin{align}\label{ineq 220718 1734}
        \Big( \sum_{j=1}^{\infty} \big| 2^{\boldsymbol{r}(j)}\Delta_j f \big|^2 \Big)^{1/2}
        \approx \Big( \mathbb{E} \Big|\sum_{j=1}^{\infty}  X_j 2^{\boldsymbol{r}(j)}\Delta_j f \Big|^p \Big)^{1/p},\quad p \in (0, \infty)
    \end{align}
    where the implicit constants depend only on $p$ and $\{X_j\}_{j\in\bN}$ is a sequence of independent and identically distributed random variables with the Rademacher distribution. Let $\Pi_1$ be a Fourier multiplier defined by
    $$
    \Pi_1(\xi) := 
    \frac{ \sum_{j=1}^\infty X_j 2^{\boldsymbol{r}(j)}\cF[\Psi](2^{-j}\xi)}{\cF[\Phi](\xi)+\sum_{j=1}^{\infty} 2^{\boldsymbol{r}(j)}\cF[\Psi](2^{-j}\xi)}.
    $$
Recall the notation
$$
\bT_{\Pi_1}f(x):=\cF^{-1}[\Pi_1\cF[f]](x)
$$
and
\begin{align*}
\pi^{\boldsymbol{r}}:=S_0+\sum_{j=1}^{\infty}2^{\boldsymbol{r}(j)}\Delta_j.
\end{align*}
Then the right-hand side of \eqref{ineq 220718 1734} equals to 
    \begin{align*}
        \left( \bE \left| \bT_{\Pi_1}(\pi^{\boldsymbol{r}}f) \right|^p \right)^{1/p}.
    \end{align*}
It is easy to check that $\Pi_1$ satisfies \eqref{22.05.12.13.14}, that is, $\Pi_1$ is a weighted Mikhlin multiplier.
    Thus, it follows that 
    \begin{equation}\label{ineq 220718 1744}
    \begin{aligned}
        \Big\| \Big( \mathbb{E} \Big|\sum_{j=1}^{\infty}  X_j 2^{\boldsymbol{r}(j)}\Delta_j f \Big|^p \Big)^{1/p}  \Big\|_{L_p(\bR^d,w\,\mathrm{d}x)}
        &\leq N\Big( \mathbb{E}  \Big\|\sum_{j=1}^{\infty}  X_j 2^{\boldsymbol{r}(j)}\Delta_j f \Big\|_{L_p(\bR^d,w\,\mathrm{d}x)}^p \Big)^{1/p}\\
        &= N\Big( \mathbb{E}  \Big\| \bT_{\Pi_1}(\pi^{\boldsymbol{r}}f) \Big\|_{L_p(\bR^d,w\,\mathrm{d}x)}^p \Big)^{1/p} 
        \\
        &\leq N [w]_{A_p(\bR^d)}^{\max(1,(p-1)^{-1})}\|\pi^{\boldsymbol{r}} f\|_{L_p(\bR^d,w\,\mathrm{d}x)}.
    \end{aligned}
    \end{equation}
    By \eqref{ineq 220718 1747}, \eqref{ineq 220718 1734} and \eqref{ineq 220718 1744}, we conclude that 
    $$
        \|\pi^{\boldsymbol{r}}f\|_{L_p(\bR^d,w\,\mathrm{d}x)}\approx \|f\|_{H_p^{\boldsymbol{r}}(\bR^d,w\,\mathrm{d}x)}.
    $$
    Similarly, we compute $\|\pi^{\boldsymbol{r}}f\|_{B_{p,q}^{\boldsymbol{0}}(\bR^d,w\,\mathrm{d}x)}$.
Recall the definition first:
    \begin{align*}
        \|\pi^{\boldsymbol{r}}f\|_{B_{p,q}^{\boldsymbol{0}}(\bR^d,w\,\mathrm{d}x)}
        &= \| S_0 \pi^{\boldsymbol{r}}f \|_{L_p(\bR^d,w\,\mathrm{d}x)} + \Big( \sum_{j=1}^{\infty} \Big\| \Delta_j \pi^{\boldsymbol{r}}f \Big\|_{L_p(\bR^d,w\,\mathrm{d}x)}^q \Big)^{1/q}
    \end{align*}
and
    \begin{align*}
        \|f\|_{B_{p,q}^{\boldsymbol{r}}(\bR^d,w\,\mathrm{d}x)}
        &= \| S_0 f \|_{L_p(\bR^d,w\,\mathrm{d}x)} + \Big( \sum_{j=1}^{\infty} \Big\| 2^{\boldsymbol{r}(j)}\Delta_j f \Big\|_{L_p(\bR^d,w\,\mathrm{d}x)}^q \Big)^{1/q}.
    \end{align*}
Due to the almost orthogonal property again, it is obvious that
    \begin{equation*}
        \begin{gathered}
        S_0\pi^{\boldsymbol{r}}f=S_0(S_0+2^{\boldsymbol{r}(1)}\Delta_1)f, \\
        \Delta_1\pi^{\boldsymbol{r}}f=\Delta_1(S_0+2^{\boldsymbol{r}(1)}\Delta_1+2^{\boldsymbol{r}(2)}\Delta_2)f, \\
        \Delta_j \pi^{\boldsymbol{r}}f = \Delta_{j}(2^{\boldsymbol{r}(j-1)} \Delta_{j-1} +2^{\boldsymbol{r}(j)} \Delta_{j} +2^{\boldsymbol{r}(j+1)} \Delta_{j+1}) f,\quad j\geq 2.
        \end{gathered}
    \end{equation*}
Moreover, setting
    \begin{equation*}
        \begin{gathered}
\Pi_2:=\frac{\cF[\Phi]+\cF[\Psi](2^{-1}\cdot)}{\cF[\Phi]+\sum_{j=1}^{\infty}2^{\boldsymbol{r}(j)}\cF[\Psi](2^{-j}\cdot)},\\
 \Pi_3:=\frac{2^{\boldsymbol{r}(1)}(\cF[\Phi]+\cF[\Psi](2^{-1}\cdot)+\cF[\Psi](2^{-2}\cdot))}{\cF[\Phi]+\sum_{j=1}^{\infty}2^{\boldsymbol{r}(j)}\cF[\Psi](2^{-j}\cdot)},\\
 \Pi_{j+2}:=\frac{2^{\boldsymbol{r}(j)}(\cF[\Psi](2^{-j+1}\cdot)+\cF[\Psi](2^{-j}\cdot)+\cF[\Psi](2^{-j-1}\cdot))}{\cF[\Phi]+\sum_{j=1}^{\infty}2^{\boldsymbol{r}(j)}\cF[\Psi](2^{-j}\cdot)},\quad j\geq 2,
        \end{gathered}
    \end{equation*}
we have
    \begin{equation*}
        \begin{gathered}
        S_0f=\bT_{\Pi_2}S_0\pi^{\boldsymbol{r}}f,
        \quad 2^{\boldsymbol{r}(1)}\Delta_1f=\bT_{\Pi_3}\Delta_1\pi^{\boldsymbol{r}}f,\\
        2^{\boldsymbol{r}(j)}\Delta_jf=\bT_{\Pi_{j+2}}\Delta_j \pi^{\boldsymbol{r}}f, \quad j\geq 2.
        \end{gathered}
    \end{equation*}
Therefore, by Lemma \ref{wbound},  and Proposition \ref{21.02.24.16.49}, we obtain
    $$
        \|\pi^{\boldsymbol{r}}f\|_{B_{p,q}^{\boldsymbol{0}}(\bR^d,w\,\mathrm{d}x)}\approx\|f\|_{B_{p,q}^{\boldsymbol{r}}(\bR^d,w\,\mathrm{d}x)}
    $$
since all $L_p$-bounds can be chosen uniformly for all $j$ due to the assumption on $\boldsymbol{r}$ \eqref{uniform r assumption}.
    The proposition is proved.
\end{proof}

\begin{corollary}
							\label{classical sobo}
   For $s\in\bR$, put $\boldsymbol{r}(j) = sj$. 
Then 
$$H_p^{\boldsymbol{r}}(\bR^d, w\,\mathrm{d}x)=H_p^s(\bR^d, w\,\mathrm{d}x)
$$
whose norm is given by 
$$
    \| f \|_{H_p^s(\bR^d,w\,\mathrm{d}x)} := \| (1 - \Delta)^{s/2} f \|_{L_p(\bR^d,w\,\mathrm{d}x)}.
$$
\end{corollary}
\begin{proof}
Recall the notation
$$
\pi^{\boldsymbol{r}}:=S_0+\sum_{j=1}^{\infty}2^{\boldsymbol{r}(j)}\Delta_j.
$$
By Proposition~\ref{22.05.03.11.34}, it follows that $\|f\|_{H_p^{\boldsymbol{r}}(\bR^d,w\,\mathrm{d}x)} \approx  \|\pi^{\boldsymbol{r}}f\|_{L_p(\bR^d,w\,\mathrm{d}x)}
$.
Thus, it is sufficient to show
\begin{align*}
\|\pi^{\boldsymbol{r}}f\|_{L_p(\bR^d,w\,\mathrm{d}x)} \approx \| (1 - \Delta)^{s/2} f \|_{L_p(\bR^d,w\,\mathrm{d}x)}.
\end{align*}
Let $\Psi$ be the Schwartz function appearing in the definition of $\Delta_j$. 
Define 
    $$
    \cF[\Phi](\xi) = \sum_{j=-\infty}^{0}\cF[\Psi](2^{-j}\xi).
    $$
and
$$
\Pi(\xi) := \frac{\cF[\Phi](\xi)+\sum_{j=1}^\infty 2^{sj} \cF[\Psi](2^{-j}\xi)}{(1+|\xi|^2)^{s/2}}.
$$
Then $\Pi$ is a weighted Mikhlin multiplier in Proposition~\ref{21.02.24.16.49}.
Thus we have
$$
    \left\|\sum_{j=1}^\infty 2^{\boldsymbol{r}(j)} \Delta_j f\right\|_{L_p(\bR^d,w\,\mathrm{d}x)} = \Big\| \bT_{\Pi} \Bigl( (1-\Delta)^{s/2} f \Bigr)\Big\|_{L_p(\bR^d,w\,\mathrm{d}x)} \leq N\| (1-\Delta)^{s/2} f\|_{L_p(\bR^d,w\,\mathrm{d}x)}.
$$
For the converse direction, it suffices to observe that
$$
\widetilde{\Pi}(\xi) = \frac{(1+|\xi|^2)^{s/2} }{\cF[\Phi](\xi)+\sum_{j=1}^\infty 2^{sj}\cF[\Psi](2^{-j}\xi)}
$$
becomes a weighted Mikhlin multiplier as well. the corollary is proved.
\end{proof}

\begin{defn}
Let $q\in(0,\infty)$, $\boldsymbol{r}:\bZ\to(-\infty,\infty)$, and $B$ be a Banach lattice.
\begin{enumerate}[(i)]
    \item We denote by $B(\ell_2^{\boldsymbol{r}})$ the set of all $B$-valued sequences $x=(x_0,x_1,\cdots)$ such that
$$
\|x\|_{B(\ell_2^{\boldsymbol{r}})}:=\|x_0\|_B+\left\|\Big(\sum_{j=1}^{\infty}2^{2\boldsymbol{r}(j)}|x_j|^2\Big)^{1/2}\right\|_B<\infty,\quad |x|:
=x\vee(-x).
$$
    \item We denote by $\ell_q^{\boldsymbol{r}}(B)$ the set of all $B$-valued sequences $x=(x_0,x_1,\cdots)$ such that
$$
\|x\|_{\ell_q^{\boldsymbol{r}}(B)}:=\|x_0\|_B+\Big(\sum_{j=1}^{\infty}2^{q\boldsymbol{r}(j)}\|x_j\|_{B}^q\Big)^{1/q}<\infty.
$$
In particular, we use the simpler notation $\ell_q^{\boldsymbol{r}}=\ell_q^{\boldsymbol{r}}(\bR)$.
\end{enumerate}

\end{defn}
\begin{defn}
Let $X$ and $Y$ be quasi-Banach spaces. 
We say that $X$ is a retract of $Y$ if there exist linear transformations $\frI:X\to Y$ and $\frP:Y\to X$ such that $\frP\frI$ is an identity operator in $X$.
\end{defn}

\begin{lem}
\label{22.04.19.16.55}
Let $p\in(1,\infty)$, $q\in(0,\infty)$, and $w\in A_p(\bR^d)$. 
Suppose that
\begin{equation}
\label{locsim}
    \sup_{j\in\bZ}|\boldsymbol{r}(j+1)-\boldsymbol{r}(j)|=:C_0<\infty.
\end{equation}
\begin{enumerate}[(i)]
    \item Then $H_p^{\boldsymbol{r}}(\bR^d,w\,\mathrm{d}x)$ is a retract of $L_p(\bR^d,w\,\mathrm{d}x)(\ell_2^{\boldsymbol{r}})$.
    \item Then $B_{p,q}^{\boldsymbol{r}}(\bR^d,w\,\mathrm{d}x)$ is a retract of $\ell_q^{\boldsymbol{r}}(L_p(\bR^d,w\,\mathrm{d}x))$.
\end{enumerate}

\end{lem}
\begin{proof}
We suggest a unified method to prove both (i) and (ii) simultaneously.
Let 
$$
(X,Y)=\left(H_p^{\boldsymbol{r}}(\bR^d,w\,\mathrm{d}x),L_p(\bR^d,w\,\mathrm{d}x)(\ell_2^{\boldsymbol{r}})\right)
$$
or 
$$
(X,Y)=\left(B_{p,q}^{\boldsymbol{r}}(\bR^d,w\,\mathrm{d}x),\ell_q^{\boldsymbol{r}}(L_p(\bR^d,w\,\mathrm{d}x))\right). 
$$
Consider two mappings:
\begin{equation*}
f \mapsto \frI(f):=(S_0f,\Delta_1f,\Delta_2f,\cdots)
\end{equation*}
and
\begin{equation*}
\boldsymbol{f}=(f_0,f_1,\cdots)\mapsto \frP(\boldsymbol{f}):=S_0(f_0+f_1)+\sum_{j=1}^{\infty}\sum_{l=-1}^1\Delta_{j}f_{j+l}.
\end{equation*}
By using the fact that $f= S_0(f) + \sum_{j=1}^\infty \Delta_j f$,
it is each to check that $\frP\frI$ is an identity operator in $X$. Moreover, due to the definitions of the function spaces,
\begin{align}
							\label{J norm equivalence}
\|f\|_{X}=\|\frI(f)\|_{Y}.
\end{align}
In other words, the target space of the mapping $f \mapsto \frI(f)$ is $Y$
and this implies that $\frI$ is a linear transformation from $X$ to $Y$ since the linearity of the mapping is obvious. 

Now, it only remains to show that $\frP$ is a linear transformation from $Y$ to $X$. 
Since the linearity is trivial as before, it is sufficient to show that the target space of the mapping 
$\boldsymbol{f}=(f_0,f_1,\cdots)\mapsto \frP(\boldsymbol{f})$ is $X$. 
By the almost orthogonality of the Littlewood--Paley projections,
\begin{equation}
\label{22.09.13.16.17}
    \begin{gathered}
     S_0\frP(\boldsymbol{f})=S_0f_0+S_0f_1+S_0\Delta_1f_{2},\\
     \Delta_1\frP(\boldsymbol{f})=\Delta_1(S_0+\Delta_1)f_0+\Delta_1f_1+\Delta_1(\Delta_1+\Delta_2)f_2+\Delta_1\Delta_2f_3,
    \end{gathered}
\end{equation}
and
\begin{equation}
    \begin{aligned}
    &\Delta_i\frP(\boldsymbol{f})\\
    =&\,\Delta_i\Delta_{i-1}f_{i-2}+\Delta_i(\Delta_{i-1}+\Delta_i)f_{i-1}+\Delta_if_i+\Delta_i(\Delta_i+\Delta_{i+1})f_{i+1}+\Delta_i\Delta_{i+1}f_{i+2}\\
    =:&T_i^1f_{i-2}+T_i^2f_{i-1}+T_i^3f_i+T_i^4f_{i+1}+T_i^5f_{i+2},\quad i\geq2.
    \end{aligned}
\end{equation}
The remaining part of the proof becomes slightly different depending on $Y=L_p(\bR^d,w\,\mathrm{d}x)(\ell_2^{\boldsymbol{r}})$ or $Y=\ell_q^{\boldsymbol{r}}(L_p(\bR^d,w\,\mathrm{d}x))$.
\begin{enumerate}[(i)]
\item Due to \eqref{22.09.13.16.17},
$$
\sum_{i=2}^{\infty}2^{2\boldsymbol{r}(i)}|\Delta_i\frP(\boldsymbol{f})|^2\leq N\sum_{k=1}^{5}\cI_k,
$$
where
$$
\cI_k=\sum_{i=2}^{\infty}|T_i^k(2^{\boldsymbol{r}(i)}f_{i+k-3})|^2,\quad k=1,2,3,4,5.
$$
By Lemma \ref{wbound}, the sequence of linear operators $\{T_j^k\}_{j=2}^{\infty}$ satisfies \eqref{22.09.13.16.42} for all $k=1,2,3,4,5$. Using Remark \ref{rem RdF vv} and \eqref{locsim}, we have
$$
\Big\|\Big(\sum_{i=2}^{\infty}2^{2\boldsymbol{r}(i)}|\Delta_i\frP(\boldsymbol{f})|^2\Big)^{1/2}\Big\|_{L_p(\bR^d,w\,\mathrm{d}x)}\leq N\|\boldsymbol{f}\|_{L_p(\bR^d,w\,\mathrm{d}x)(\ell_2^{\boldsymbol{r}})}.
$$
This implies that $\frP(\boldsymbol{f})\in H_p^{\boldsymbol{r}}(\bR^d,w\,\mathrm{d}x)$.

\item By Lemma \ref{wbound},
\begin{align}
							\label{2023012880}
 \|S_0\frP(\boldsymbol{f})\|_{L_p(\bR^d,w\,\mathrm{d}x)}\leq N \sum_{j=0}^{2}\|f_{j}\|_{L_p(\bR^d,w\,\mathrm{d}x)}
\end{align}
and
\begin{align}
							\label{2023012881}
 \|\Delta_i\frP(\boldsymbol{f})\|_{L_p(\bR^d,w\,\mathrm{d}x)}\leq N \sum_{j=-2}^{2}\|f_{i+j}\|_{L_p(\bR^d,w\,\mathrm{d}x)},\quad \forall i\in\bN,
\end{align}
where $N$ is independent of $i$ and $f_i=0$ if $i<0$. 
Moreover, due to \eqref{locsim},
\begin{equation}
\label{22.04.11.15.09}
    2^{\boldsymbol{r}(i)}\|\Delta_i\frP(\boldsymbol{f})\|_{L_p(\bR^d,w\,\mathrm{d}x)}\leq N\sum_{j=-2}^{2}2^{\boldsymbol{r}(i+j)}\|f_{i+j}\|_{L_p(\bR^d,w\,\mathrm{d}x)}
\end{equation}
and the constant $N$ is independent of $i$. 
Finally, recalling the definition of $B_{p,q}^{\boldsymbol{r}}(\bR^d,w\,\mathrm{d}x)$ and combining all \eqref{2023012880}, \eqref{2023012881}, and \eqref{22.04.11.15.09}, we have
\begin{equation*}
    \begin{aligned}
    \|\frP(\boldsymbol{f})\|_{B_{p,q}^{\boldsymbol{r}}(\bR^d,w\,\mathrm{d}x)} \leq N\|\boldsymbol{f}\|_{\ell_q^{\boldsymbol{r}}(L_p(\bR^d,w\,\mathrm{d}x))},\quad q\in(0,\infty),
    \end{aligned}
\end{equation*}
where $N$ is independent of $\boldsymbol{f}$. Therefore, $\frP$ is a linear transformation from $Y$ to $X$. 
\end{enumerate}
The lemma is proved.
\end{proof}

\begin{prop}
\label{22.04.24.20.57}
    Let $p\in(1,\infty)$, $q\in(0,\infty)$, $w\in A_p(\bR^d)$, and $\boldsymbol{r}:\bZ\to(-\infty,\infty)$ be a sequence satisfying
    $$
    \sup_{j\in\bZ}|\boldsymbol{r}(j+1)-\boldsymbol{r}(j)|<\infty.
    $$
        \begin{enumerate}[(i)]
        \item The space $X$ is a quasi-Banach space
        \item The closure of $C_c^{\infty}(\bR^d)$ under the quasi-norm $\|\cdot\|_{X}$ is $X$,
    \end{enumerate}
where the space $X$ indicates $H_p^{\boldsymbol{r}}(\bR^d,w\,\mathrm{d}x)$ or $B_{p,q}^{\boldsymbol{r}}(\bR^d,w\,\mathrm{d}x)$.
\end{prop}
\begin{proof}
We put
$$
(X,Y)=\left(H_p^{\boldsymbol{r}}(\bR^d,w\,\mathrm{d}x),L_p(\bR^d,w\,\mathrm{d}x)(\ell_2^{\boldsymbol{r}})\right)
$$
or 
$$
(X,Y)=\left(B_{p,q}^{\boldsymbol{r}}(\bR^d,w\,\mathrm{d}x),\ell_q^{\boldsymbol{r}}(L_p(\bR^d,w\,\mathrm{d}x))\right)
$$
as in the previous lemma.
\begin{enumerate}[(i)]
\item All properties of a quasi-Banach space are obvious except completeness.
    Thus, we only prove the completeness. Let $\{f_n\}_{n=1}^{\infty}$ be a Cauchy sequence in $X$.
    It is obvious that $Y$ is a Banach space due to the completeness of $L_p$-spaces with general measures. 
     Therefore, using the equivalence of the norms in \eqref{J norm equivalence}, one can find a $\boldsymbol{f}\in Y$ such that 
    \begin{align*}
        \frI(f_n)=(S_0f_n,\Delta_1f_n,\cdots)\to\boldsymbol{f}\quad \text{ in } Y.
    \end{align*}
    By Lemma \ref{22.04.19.16.55}, $\frP(\boldsymbol{f})\in X$ and
    $$
    \|f_n-\frP(\boldsymbol{f})\|_{X}\leq N\|\frI(f_n)-\boldsymbol{f}\|_{Y}\to0
    $$
    as $n\to\infty$. Therefore, $X$ is complete.
\item With the help of Proposition \ref{22.05.03.11.34}, it is sufficient to consider the case that the order is $\boldsymbol{0}$.
Then $\cS(\bR^d)$ is dense in $X$ due to \cite[Theorem 2.4]{qui1982weighted}. 
Finally, based on the typical truncation and diagonal arguments, the result is obtained. 
\end{enumerate}
The proposition is proved.
\end{proof}

\vspace{1cm}

\textbf{Declarations of interest.}
Declarations of interest: none

\textbf{Data Availability.}
Data sharing not applicable to this article as no datasets were generated or analysed during the current study.

\textbf{Acknowledgements.}
J.-H. Choi has been supported by a KIAS Individual Grant(MG102701) at Korea Institute for Advanced Study.
I. Kim has been supported by the National Research Foundation of Korea(NRF) grant funded by the Korea government(MSIT) (RS-2025-16065358). 
J.B. Lee has been supported by the National Research Foundation of Korea(NRF) grant funded by the Korea government(MSIT) (No.2021R1C1C2008252).
The authors would like to thank the referees for their careful reading and useful comments.
J.-H. Choi and J. B. Lee are also highly indebted to Dr. Jinsol Seo for the careful reading of the manuscript and the valuable suggestions and remarks.

\bibliographystyle{plain}

\begin{thebibliography}{99}


\bibitem{ALM2022}
A. Agresti, N. Lindemulder, M. C. Veraar, \emph{On the trace embedding and its applications to evolution equations}, Math. Nachr. \textbf{296(4)}, pp 1319-1350, 2023, DOI : 10.1002/mana.202100192.


\bibitem{BL1976}
J. Bergh and J. L\"ofstr\"om, \emph{Interpolation spaces: an introduction}, Springer-Verlag (1976).

\bibitem{choi_thesis}
J.-H. Choi, \emph{On $L_p$-regularity theory for parabolic partial differential equations with time measurable pseudo-differential operators (Ph.D diss.)}, Korea University, 2022.

\bibitem{Choi_Kim2022}
J.-H. Choi, I. Kim, \emph{A weighted $L_p$-regularity theory for parabolic partial differential equations with time measurable pseudo-differential operators}, J. Pseudo-Differ. Oper. Appl. \textbf{14(4)}, Article No. 55, 2023, DOI : 10.1007/s11868-023-00550-6.

\bibitem{Choi_Kim2023}
J.-H. Choi, I. Kim,
\newblock{A maximal $L_p$-regularity theory to initial value problems with time measurable nonlocal operators generated by additive processes,}
\newblock{\em Stoch. Partial Differ. Equ. Anal. Comp.} 
\textbf{12(1)}, pp 352-415, 2024, DOI : 10.1007/s40072-023-00286-w.

\bibitem{CLSW2023trace}
J.-H. Choi, J.B. Lee, J. Seo, K. Woo, \emph{On the trace theorem to Volterra type equations with local or non-local derivatives}, arXiv preprint, arXiv:2309.00370v2.



\bibitem{Dong_Kim2018}
H. Dong, D. Kim, \emph{On $L_p$-estimates for elliptic and parabolic equations with $A_p$ weights}, Trans. Am. Math. Soc. \textbf{370}, pp 5081-5130, 2018, DOI : 10.1090/tran/7161.

\bibitem{Dong_Kim2021}
H. Dong, D. Kim, \emph{An approach for weighted mixed-norm estimates for parabolic equations with local and non-local time derivatives}, Adv. Math. \textbf{377}, 107494, 2021, DOI : 10.1016/j.aim.2020.107494.

\bibitem{Dong_Liu2022}
H. Dong, Y. Liu, \emph{Weighted mixed norm estimates for fractional wave equations with VMO coefficients}, J. Differ. Equ. \textbf{337}, pp 168-254, 2022, DOI : 10.1016/j.jde.2022.07.040.


\bibitem{dong2021sobolev}
H. Dong, Y. Liu, \emph{Sobolev estimates for fractional parabolic equations with space-time non-local operators}, Calc. Var. Partial Differ. Equ. \textbf{62}, Article no.96, 2023, DOI : 10.1007/s00526-023-02431-8.

\bibitem{fackler2020weighted}
S. Fackler, T.P. Hyt\"onen, N. Lindemulder.
\emph{Weighted estimates for operator-valued Fourier multipliers}, Collect. Math. \textbf{71}, pp 511-548, 2020, DOI : 10.1007/s13348-019-00275-0.

\bibitem{Gallarati_Veraar2017}
C. Gallarati, M. C. Veraar, \emph{Maximal regularity for non-autonomous equations with measurable dependence on time}, Potential Anal. \textbf{46(3)}, pp 527-567, 2017, DOI : 10.1007/s11118-016-9593-7.

\bibitem{RdF1985weighted}
J. Garc\'ia-Cuerva, J. L. Rubio de Francia, \emph{Weighted norm inequalities and related topics}, North-Holland, Amsterdam, 1985.

\bibitem{grafakos2014classical}
L. Grafakos, \textit{Classical Fourier Analysis}, 3rd ed., Graduate Texts in Mathematics Vol. 249, Springer, New York, 2014, DOI : 10.1007/978-1-4939-1194-3.

\bibitem{gyongy2021lp}
I. Gy\"ongy, S. Wu, \emph{On $L_p$-solvability of stochastic integro-differential equations}, Stoch. Partial Differ. Equ. Anal. Comput. \textbf{9(2)}, pp 295–342, 2021, DOI : 10.1007/s40072-019-00160-8.


\bibitem{Haagerup1981}
U. Haagerup, \emph{The best constants in the Khintchine inequality}, Studia Math. \textbf{70(3)}, pp 231-283, 1981.

\bibitem{Hemmel_Lindemulder2022}
F. Hummel, N. Lindemulder \emph{Elliptic and parabolic boundary value problems in weighted function spaces}, Potential Anal. \textbf{57}, pp 601–669, 2022, DOI : 10.1007/s11118-021-09929-w.

\bibitem{Hy_Ka2012}
T. Hyt\"onen, A. Kairema, \emph{Systems of dyadic cubes in a doubling metric space}, Colloq. Math. \textbf{126(1)}, pp 1-33, 2012, DOI : 10.4064/cm126-1-1.

\bibitem{kang2021lp}
J. Kang, D. Park. \emph{An $L_q(L_p)$-theory for time-fractional diffusion equations with nonlocal operators generated by L\'evy processes with low intensity of small jumps}, Stoch. Partial Differ. Equ. Anal. Comput. \textbf{12(3)}, pp 1439--1491, 2024, DOI : 10.1007/s40072-023-00309-6.

\bibitem{kim2015parabolic}
I. Kim, K.-H. Kim, S. Lim, \emph{Parabolic BMO estimates for pseudo-differential operators of arbitrary order}, J. Math. Anal. Appl. \textbf{427(2)}, pp 557-580, 2015, DOI : 10.1016/j.jmaa.2015.02.065.

\bibitem{kim2016lplq}
I. Kim, S. Lim, K.-H. Kim, \emph{An $L_q(L_p)$-Theory for parabolic pseudo-differential equations: Calder\'on--Zygmund approach}, Potential Anal. \textbf{45(3)}, pp 463-483, 2016, DOI : 10.1007/s11118-016-9552-3.

\bibitem{kim2016lp}
I. Kim, K.-H. Kim, \emph{An $L_p$-theory for stochastic partial differential equations driven by L\'evy processes with pseudo-differential operators of arbitrary order}, Stoch. Process. Appl. \textbf{126(9)}, pp 2761-2786, 2016, DOI : 10.1016/j.spa.2016.03.001.

\bibitem{kim2018lp}
I. Kim, \emph{An $L_p$-Lipschitz theory for parabolic equations with time measurable pseudo-differential operators}, Commun. Pure Appl. Anal. \textbf{17(6)}, pp 2751-2771, 2018, DOI : 10.3934/cpaa.2018130.

\bibitem{kim2018}
I. Kim,  K.-H. Kim, \emph{On the second order derivative estimates for degenerate parabolic equations}, J. Differ. Equ. \textbf{265(11)}, pp 5959-5983, 2018, DOI : 10.1016/j.jde.2018.07.014.

\bibitem{kim2019lp}
I. Kim, K.-H. Kim, P. Kim, \emph{An $L_p$-theory for diffusion equations related to stochastic processes with non-stationary independent increment}, Trans. Am. Math. Soc. \textbf{371(5)}, pp 3417-3450, 2019, DOI : 10.1090/tran/7410.

\bibitem{kim2012lp}
K.-H. Kim, P. Kim, \emph{An $L_p$-theory of a class of stochastic equations with the random fractional Laplacian driven by L\'evy processes}, Stoch. Process. Appl. \textbf{122(12)}, pp 3921–3952, 2012, DOI : 10.1016/j.spa.2012.08.001.

\bibitem{kim2017}
K.-H. Kim, K. Lee, \emph{On the heat diffusion starting with degeneracy}, J. Differ. Equ. \textbf{262(3)}, pp 2722-2744, 2017, DOI : 10.1016/j.jde.2016.11.013.


\bibitem{kim2021lq}
K.-H. Kim, D. Park, J. Ryu, \emph{An $L_q(L_p)$-theory for diffusion equations with space-time nonlocal operators}, J. Differ. Equ. \textbf{287}, pp 376-427, 2021, DOI : 10.1016/j.jde.2021.04.003.

\bibitem{kim2021sobolev}
K.-H. Kim, D. Park, J. Ryu, \emph{A Sobolev space theory for the Stochastic Partial Differential Equations with space-time non-local operators}, J. Evol. Equ. \textbf{22(3)}, 57, 2022, DOI : 10.1007/s00028-022-00813-7.

\bibitem{Kohne2010}
M. K\"ohne, J. Pr\"uss, M. Wilke, \emph{On quasilinear parabolic evolution equations in weighted $L_p$-spaces}, J. Evol. Equ. \textbf{10(2)}, pp 443-463, 2010, DOI : 10.1007/s00028-010-0056-0.

\bibitem{Kohne2014}
M. K\"ohne, J. Pr\"uss, M. Wilke, \emph{On quasilinear parabolic evolution equations in weighted $L_p$-spaces \uppercase\expandafter{\romannumeral2}}, J. Evol. Equ. \textbf{14}, pp 509–533, 2014, DOI : 10.1007/s00028-014-0226-6.

\bibitem{kurtz1979results}
D.S. Kurtz, R.L. Wheeden, \emph{Results on weighted norm inequalities for multipliers}, Trans. Am. Math. Soc. \textbf{255}, pp 343-362, 1979, DOI : 10.2307/1998180.

\bibitem{kurtz1980little}
D.S. Kurtz, \emph{Littlewood--Paley and multiplier theorems on weighted $L^p$ spaces}, Trans. Am. Math. Soc. \textbf{259(1)}, pp235-257, 1980, DOI : 10.2307/1998156.

\bibitem{Lindemulder2020}
N. Lindemulder, \emph{Maximal regularity with weights for parabolic problems with inhomogeneous boundary conditions}, J. Evol. Equ. \textbf{20}, pp 59–108, 2020, DOI : 10.1007/s00028-019-00515-7.

\bibitem{Ler2011weighted}
A. K. Lerner, \emph{Sharp weighted norm inequalities for Littlewood--Paley operators and singular integrals}, Adv. Math. \textbf{226(5)}, pp 3912-3926, 2011.

\bibitem{Meyries2012}
M. Meyries, R. Schnaubelt, \emph{Maximal regularity with temporal weights for parabolic problems with inhomogeneous boundary conditions}, Math. Nachr. \textbf{285(8-9)}, pp 1032-1051, 2012, DOI : 10.1002/mana.201100057.

\bibitem{Martin_Veraar2014}
M. Meyries, M. C. Veraar, \emph{Traces and embeddings of anisotropic function spaces}, Math. Ann. \textbf{360(3)}, pp 571-606, 2014, DOI : 10.1007/s00208-014-1042-6.

\bibitem{mikulevivcius1992}
R. Mikulevi\v{c}ius, H. Pragarauskas, \emph{On the Cauchy problem for certain integro-
differential operators in Sobolev and H\"older spaces}, Lith. Math. J. \textbf{32(2)}, pp 238–264, 1992. 

\bibitem{mikulevivcius2014}
R. Mikulevi\v{c}ius, H. Pragarauskas, \emph{On the Cauchy problem for integro-differential
operators in Sobolev classes and the martingale problem}, J. Differ. Equ. \textbf{256(4)}, pp 1581–1626, 2014, DOI : 10.1016/j.jde.2013.11.008.

\bibitem{mikulevivcius2017p}
R. Mikulevi\v{c}ius, C. Phonsom, \emph{On $L^p$ theory for parabolic and elliptic integro-differential equations with scalable operators in the whole space}, Stoch. Partial Differ. Equ. Anal. Comput. \textbf{5(4)}, pp 472–519, 2017, DOI : 10.1007/s40072-017-0095-4.

\bibitem{mikulevivcius2019cauchy}
R. Mikulevi\v{c}ius, C. Phonsom, \emph{On the Cauchy problem for integro-differential equations in the scale of spaces of generalized smoothness}, Potential Anal. \textbf{50(3)}, pp 467-519, 2019, DOI : 10.1007/s11118-018-9690-x.

\bibitem{neerven2012maximal}
J.V. Neerven, M. Veraar, L. Weis, \emph{Maximal $L^p$-Regularity for Stochastic Evolution Equations}, SIAM  J. Math. Anal. \textbf{44(4)}, pp 1372–1414, 2012, DOI : 10.1137/110832525.

\bibitem{neerven2020}
J.V. Neerven, M. Veraar,  \emph{Maximal inequalities for stochastic convolutions in 2-smooth Banach spaces and applications to stochastic evolution equations}, {Philos. Trans. Royal Soc. A.} \textbf{378(2185)}, 2020, DOI : 10.1098/rsta.2019.0622.

\bibitem{portal2019stochastic}
P. Portal, M. Veraar, \emph{Stochastic maximal regularity for rough time-dependent problems}, Stoch. Partial Differ. Equ. Anal. Comput. \textbf{7(4)}. pp 541–597, 2019, DOI : 10.1007/s40072-019-00134-w.

\bibitem{qui1982weighted}
B.H. Qui, \emph{Weighted Besov and Triebel spaces: Interpolation by the real method}, Hiroshima Math. J. \textbf{12(3)}, pp 581-605, 1982, DOI : 10.32917/hmj/1206133649.


\bibitem{rychkov2001weights}
V.S. Rychkov, \emph{Littlewood–Paley theory and function spaces with $A_p^{loc}$ weights} Math. Nachr. \textbf{224(1)}, pp 145-180, 2001, DOI : 10.1002/1522-2616(200104)224:1$<$145::AID-MANA145$>$3.0.CO;2-2.



  
\bibitem{stein2016harmonic}
E.M. Stein, \emph{Harmonic analysis : Real-variable methods, Orthogonality, and Oscillatory integrals}, Princeton Mathematical series Vol. 43, Princeton University Press, 2016, DOI : 10.1515/9781400883929.

\bibitem{triebel1978interpolation}
H. Triebel, \emph{Interpolation theory, function spaces, differential operators}, North-Holland, 1978.

\bibitem{Wil2007square}
M. Wilson, \emph{The intrinsic square function}, Rev. Mat. Iberoam. \textbf{23}, pp 771-791, 2007.

\bibitem{zhang2013maximal}
X. Zhang, \emph{$L^p$-maximal regularity of nonlocal parabolic equations and applications}, Ann. l'Inst. Henri Poincar\'e C, Anal. non lin\'eaire \textbf{30(4)}, pp 573-614, 2013, DOI : 10.1016/j.anihpc.2012.10.006.

\bibitem{zhang2013lp}
X. Zhang, \emph{$L^p$-solvability of nonlocal parabolic equations with spatial dependent and non-smooth kernels}, In Emerging Topics on Differential Equations and Their Applications pp 247–262, Nankai Series in Pure, Applied Mathematics and Theoretical Physics Vol. 10, World Scientific, 2013, DOI : 10.1142/9789814449755\_0020.



\end{thebibliography}

\end{document}